\documentclass[12pt,a4paper]{amsart}
\usepackage{amssymb,graphicx}
\usepackage[all,cmtip]{xy}
\usepackage{hyperref}

\calclayout

\DeclareMathOperator{\Hom}{Hom}
\DeclareMathOperator{\Tor}{Tor}

\DeclareMathOperator{\Ext}{Ext}
\DeclareMathOperator{\Br}{Bar}
\DeclareMathOperator{\Cb}{Cob}
\DeclareMathOperator{\ovCb}{\overline{Cob}}
\DeclareMathOperator{\Tot}{Tot}
\DeclareMathOperator{\coker}{coker}
\DeclareMathOperator{\cone}{cone}

\newcommand{\id}{\mathrm{id}}
\newcommand{\gr}{\mathrm{gr}}
\renewcommand{\ss}{\mathrm{ss}}
\newcommand{\rop}{{\mathrm{op}}}

\newcommand{\+}{\protect\nobreakdash-}
\renewcommand{\:}{\colon}
\newcommand{\ot}{\otimes}
\renewcommand{\d}{\partial}
\newcommand{\ocn}{\odot}

\newcommand{\rarrow}{\longrightarrow}
\newcommand{\larrow}{\longleftarrow}

\DeclareFontFamily{U}{mathb}{\hyphenchar\font45}
\DeclareFontShape{U}{mathb}{m}{n}{
      <5> <6> <7> <8> <9> <10> gen * mathb
      <10.95> mathb10 <12> <14.4> <17.28> <20.74> <24.88> mathb12
      }{}
\DeclareSymbolFont{mathb}{U}{mathb}{m}{n}
\DeclareFontSubstitution{U}{mathb}{m}{n}
\DeclareMathSymbol{\blackdiamond}{0}{mathb}{"0C}

\newcommand{\bu}{{\text{\smaller\smaller$\scriptstyle\bullet$}}}
\newcommand{\cu}{{\text{\smaller$\scriptstyle\blackdiamond$}}}
\newcommand{\lrarrow}{\mskip.5\thinmuskip\relbar\joinrel\relbar\joinrel
 \rightarrow\mskip.5\thinmuskip\relax}
\newcommand{\llarrow}{\mskip.5\thinmuskip\leftarrow\joinrel\relbar
 \joinrel\relbar\mskip.5\thinmuskip\relax}

\newcommand{\vect}{{\operatorname{\mathsf{--vect}}}}
\newcommand{\modl}{{\operatorname{\mathsf{--mod}}}}
\newcommand{\comodl}{{\operatorname{\mathsf{--comod}}}}
\newcommand{\contra}{{\operatorname{\mathsf{--contra}}}}
\newcommand{\alg}{{\operatorname{\mathsf{--alg}}}}
\newcommand{\coalg}{{\operatorname{\mathsf{--coalg}}}}

\newcommand{\co}{{\mathsf{co}}}
\newcommand{\ctr}{{\mathsf{ctr}}}
\newcommand{\abs}{{\mathsf{abs}}}
\newcommand{\dg}{{\mathsf{dg}}}
\newcommand{\cdg}{{\mathsf{cdg}}}
\newcommand{\aug}{{\mathsf{aug}}}
\newcommand{\coaug}{{\mathsf{coaug}}}
\newcommand{\conilp}{{\mathsf{conilp}}}
\newcommand{\inj}{{\mathsf{inj}}}
\newcommand{\proj}{{\mathsf{proj}}}
\newcommand{\bco}{{\mathsf{bco}}}
\newcommand{\bctr}{{\mathsf{bctr}}}

\newcommand{\sA}{\mathsf A}
\newcommand{\sB}{\mathsf B}
\newcommand{\sC}{\mathsf C}
\newcommand{\sD}{\mathsf D}
\newcommand{\sS}{\mathsf S}

\newcommand{\Quis}{\mathsf{Quis}}
\newcommand{\FQuis}{\mathsf{FQuis}}
\newcommand{\Hot}{\mathsf{Hot}}
\newcommand{\Ac}{\mathsf{Ac}}

\newcommand{\boZ}{\mathbb Z}
\newcommand{\boR}{\mathbb R}
\newcommand{\boL}{\mathbb L}

\newcommand{\g}{\mathfrak g}

\newcommand{\udT}{\rotatebox[origin=c]{180}{$T$}}

\theoremstyle{plain}
\newtheorem{thm}{Theorem}[section]
\newtheorem{prop}[thm]{Proposition}
\newtheorem{lem}[thm]{Lemma}

\theoremstyle{definition}

\newtheorem{rem}[thm]{Remark}
\newtheorem{rems}[thm]{Remarks}
\newtheorem{ex}[thm]{Example}
\newtheorem{exs}[thm]{Examples}

\theoremstyle{definition}
\newtheorem{defn}[thm]{Definition}

\newcommand{\Section}[1]{\bigskip\section{#1}\medskip}

\begin{document}

\title{Differential graded Koszul duality: \\ an introductory survey}

\author{Leonid Positselski}

\address{Institute of Mathematics of the Czech Academy of Sciences \\
\v Zitn\'a~25, 115~67 Praha~1 (Czech Republic)}

\email{positselski@math.cas.cz}

\begin{abstract}
 This is an overview on derived nonhomogeneous Koszul duality over
a field, mostly based on the author's memoir~\cite{Pkoszul}.
 The paper is intended to serve as a pedagogical introduction and
a summary of the covariant duality between DG\+al\-gebras and curved
DG\+coalgebras, as well as the triality between DG\+mod\-ules,
CDG\+comodules, and CDG\+contramodules.
 Some personal reminiscences are included as a part of historical
discussion.
\end{abstract}

\maketitle

\tableofcontents

\section*{Introduction}
\medskip

 Koszul duality is a fundamental phenomenon in mathematics, including
such fields as algebraic topology, algebraic and differential geometry,
and representation theory.
 This phenomenon is so general that it does not seem to admit
a ``maximal natural'' generality.
 Whatever formulation one comes up with, it can be likely further
extended by building something on top or underneath.

 Koszul duality is a synthetic subject.
 Depending on context, its formulation involves such concepts as
differential graded structures, curved diffferential graded structures,
comodules and contramodules, conilpotency, coderived and contraderived
categories, operads and properads, and whatnot.

 Koszul duality has to be distinguished from
the \emph{comodule-contramodule correspondence}.
 The co-contra correspondence is a less familiar though no less 
fundamental, but simpler phenomenon, often accompanying
the Koszul duality.
 The following rule of thumb for distinguishing the two phenomena
may be helpful.
 Viewed as a covariant category equivalence, the co-contra
correspondence takes injective objects to projective objects or
vice versa.
 Koszul duality takes projective objects to irreducible objects and/or
irreducible objects to injective objects etc.~\cite[Prologue]{Prel}.

 The aim of this paper is to introduce the reader to the subject of
derived Koszul duality in the context of differential graded algebras
and modules, as well as coalgebras, comodules, and contramodules;
and survey some of its main results.
 We also discuss curved differential structures, but \emph{no}
$\mathrm{A}_\infty$\+structures are considered in this paper
(the reader can find discussions of these in the Koszul duality
context in the dissertation~\cite{Lef} and
the memoirs~\cite{Pkoszul,Pweak}).
 \emph{No} relative Koszul duality settings (as in~\cite[Sections~0.4
and~11]{Psemi} or~\cite{Prel}) are discussed in this survey, either;
so everything in this paper happens over a field.
 We also do not consider operads and their generalizations, referring
instead to the book~\cite{LV} and the papers~\cite{HM,Val}.

 This introductory survey is intended for an audience versed in
homological algebra generally, but largely unfamiliar with the subject
of derived Koszul duality.
 It is almost entirely based on the author's memoir~\cite{Pkoszul}.
 The more elementary topic of \emph{underived} quadratic
duality~\cite{Prid,Pcurv,BGS,PP} is only briefly touched in this paper.
 \emph{Homogeneous} versions of derived Koszul duality~\cite{BGS},
\cite[Appendix~A]{Pkoszul} (involing graded modules over positively
or negatively graded algebras, with the differentials preserving
the grading) are not elaborated upon in this paper, either.
 Model structures and compact generators are also only briefly touched.

 We start with posing the problem in a basic particular case,
formulating the main results in an approximate form in the simplest
settings, and then proceed to introduce further ingredients and make
more precise and general assertions.
 Some sketches of proofs are included in this survey, but most
results receive a brief outline of an idea of the proof and
a reference to a more detailed treatment.

 All the Koszul duality functors, triangulated equivalences etc.\ in
this paper are \emph{covariant} functors or category equivalences.
 A discussion of \emph{contravariant} version of \emph{homogeneous}
Koszul duality can be found in~\cite[Section~A.2]{Pkoszul}.

 One disclaimer is in order.
 As usual in differential graded homological algebra, sign rules
(plus or minus) present a tedious problem.
 In this survey, we skip many descriptions of sign rules, referring
the reader to the original publications such as~\cite{Pkoszul}
for the details.
 So our formulations may be imprecise in that not all the signs
are properly spelled out.

\subsection*{Acknowledgements}
 I~am grateful to Andrey Lazarev for suggesting the idea of writing
this survey to me.
 I~wish to thank Jan \v St\!'ov\'\i\v cek for asking me questions
about derived Koszul duality and subsequently offering me to present
the material at his seminar on $\infty$\+categories.
 This stimulated me to come up with a pedagogical introductory
exposition.
 An acknowledgement is also due to Darya Krinitsina for asking me
a question about the proofs in~\cite[Section~6]{Pkoszul} which
is answered by Lemmas~\ref{freely-generated-cdg-module-lemma}\+-%
\ref{acyclic-filtration-lemma} in the last section of this survey.
 Last but not least, I~want to thank the anonymous referee for
careful reading of the manuscript and several helpful suggestions.
 The author is supported by the GA\v CR project 20-13778S and
research plan RVO:~67985840.

\Section{Algebras and Modules} \label{algebras-and-modules-secn}

 Throughout this paper, we work over a fixed ground field~$k$.
 Unless otherwise mentioned, all \emph{algebras} in this paper are
associative and unital, all \emph{modules} are unital, and all 
\emph{homomorphisms of algebras} are presumed to take the unit to
the unit.

 In this section we start posing the problem of derived nonhomogeneous
Koszul duality in the simplest particular case of complexes of modules
over an augmented $k$\+algebra~$A$.

\subsection{Augmented algebras} \label{augmented-algebras-subsecn}
 Let $A$ be an (associative, unital) algebra over a field~$k$.
 An \emph{augmentation}~$\alpha$ on $A$ is a (unital) $k$\+algebra
homomorphism $\alpha\:A\rarrow k$.
 So $\alpha(1)=1$, and the existence of an augmentation implies that
$1\ne0$ in~$A$; hence $\alpha$~is a surjective map.
 We denote by $A^+=\ker(\alpha)\subset A$ the augmentation ideal;
so $A^+$ is a two-sided ideal in $A$, and $A=k\oplus A^+$ as
a $k$\+vector space.
 The $k$\+algebra homomorphism $\alpha\:A\rarrow k$ endows
the one-dimensional $k$\+vector space~$k$ with left and right
$A$\+module structures.

 The following definition goes back to
the papers~\cite[Chapter~II]{EML} and~\cite{Ad}.
 The \emph{bar-construction} $\Br^\bu_\alpha(A)$ of an augmented
$k$\+algebra $A=(A,\alpha)$ is defined as the complex
$$
 k\overset0\llarrow A^+\overset\d\llarrow A^+\ot_k A^+
 \overset\d\llarrow A^+\ot_k A^+\ot_k A^+\llarrow\dotsb
$$
with the diffential given by the formulas $\d(a\ot b)=ab$, \
$\d(a\ot b\ot c)=ab\ot c-a\ot bc$,~\dots,
\begin{multline*}
 \d(a_1\ot\dotsb\ot a_n)=a_1a_2\ot a_3\ot\dotsb\ot a_n-\dotsb \\
 +(-1)^{i+1}a_1\ot\dotsb\ot a_{i-1}\ot a_ia_{i+1}\ot a_{i+2}\ot
 \dotsb\ot a_n+\dotsb \\
 +(-1)^n a_1\ot\dotsb\ot a_{n-2}\ot a_{n-1}a_n,
\end{multline*}
etc., for all $a$, $b$, $c$, $a_i\in A^+$, \ $n\ge1$, and
$1\le i\le n-1$.
 The leftmost differential $\d\:A^+\rarrow k$ is the zero map.

 The complex $\Br^\bu_\alpha(A)$ computes the $k$\+vector spaces
$\Tor^A_*(k,k)$ (where $k$~is endowed with left and right
$A$\+module structures via~$\alpha$).
 In other words, there are natural isomorphisms of $k$\+vector
spaces
$$
 H^{-n}\Br^\bu_\alpha(A)\simeq\Tor^A_n(k,k)
 \qquad\text{for all $n\ge0$}.
$$

 Let $(A,\alpha)$ and $(B,\beta)$ be two augmented $k$\+algebras,
and let $f\:A\rarrow B$ be a $k$\+algebra homomorphism compatible
with the augumentations (i.~e., satisfying the equation $\alpha=
\beta f$).
 Then there is the induced map of augmentation ideals $f^+\:
A^+\rarrow B^+$ and the induced map of the bar-constructions
$\Br^\bu(f)\:\Br^\bu_\alpha(A)\rarrow\Br^\bu_\beta(B)$.
 The latter map is a morphism of complexes of $k$\+vector spaces.

\begin{ex} \label{bar-of-algebras-quasi-isomorphism}
 The morphism of complexes $\Br^\bu(f)$ may well be a quasi-isomorphism
even when a morphism of augmented $k$\+algebras $f\:A\rarrow B$
is \emph{not} an isomorphism.

 For example, the complex $0\larrow B^+\larrow B^+\ot_k B^+\larrow
B^+\ot_k B^+\ot_k B^+\larrow\dotsb$ is acyclic whenever
the augmentation ideal $B^+$, viewed as an associative $k$\+algebra,
has a unit of its own.
 The sequence of $k$\+linear maps $h\:B^+{}^{\ot n}\rarrow
B^+{}^{\ot n+1}$ given by the rules $h(b_1\ot\dotsb\ot b_n)=
e\ot b_1\ot\dotsb\ot b_n$ for all $b_i\in B^+$, \ $1\le i\le n$,
\ $n\ge1$, where $e$~is the unit in $B^+$, provides
a contracting homotopy.
 
 Choose any nonzero associative, unital $k$\+algebra, denote it by
$B^+$, and adjoin a new unit to it formally, producing
the $k$\+algebra $B=k\oplus B^+$.
 Then the natural homomorphisms of augmented $k$\+algebras $k\rarrow B
\rarrow k$ induce quasi-isomorphisms of the bar-constructions.

 More generally, let $A$ and $B$ be two associative, unital
$k$\+algebras.
 Consider the direct sum $A\oplus B$, and endow it with a $k$\+algebra
structure as the product of $A$ and $B$ in the category of
$k$\+algebras.
 Choose an augmentation $\alpha\:A\rarrow k$, and let the augmentation
of $A\oplus B$ be constructed as the composition $A\oplus B\rarrow A
\overset\alpha\rarrow k$.
 Then the natural homomorphism of augmented $k$\+algebras $A\oplus B
\rarrow A$ (the direct summand projection) induces a quasi-isomorphism
of the bar-constructions.
\end{ex}

\subsection{Bar-constructions of modules} \label{bar-of-modules-subsecn}
 Let $A$ be an associative $k$\+algebra with an augmentation
$\alpha\:A\rarrow k$.
 Let $M$ be a left $A$\+module.

 The \emph{bar-construction} $\Br^\bu_\alpha(A,M)$ of $A$ with
the coefficients in $M$ is the complex
$$
 M \overset\d\llarrow A^+\ot_k M \overset\d\llarrow
 A^+\ot_k A^+\ot_k M\llarrow\dotsb
$$
with the diffferential given by the formulas $\d(a\ot m)=am$, \
$\d(a\ot b\ot m)=ab\ot\nobreak m-a\ot bm$,~\dots,
$$
 \d(a_1\ot\dotsb\ot a_n\ot m)=a_1a_2\ot a_3\ot\dotsb\ot a_n\ot m
 - \dotsb +(-1)^{n+1}a_1\ot\dotsb\ot a_{n-1}\ot a_nm,
$$
etc., for all $a$, $b$, $a_i\in A^+$, \ $m\in M$, and $n\ge1$.

 The complex $\Br^\bu_\alpha(A,M)$ computes the $k$\+vector spaces
$\Tor^A_*(k,M)$ (where $k$~is endowed with a right $A$\+module
structure via~$\alpha$).
 In other words, there are natural isomorphisms of $k$\+vector
spaces
$$
 H^{-n}\Br^\bu_\alpha(A,M)\simeq\Tor^A_n(k,M)
 \qquad\text{for all $n\ge0$}.
$$

\begin{exs} \label{bar-of-modules-acyclic}
 The complex $\Br^\bu_\alpha(A,M)$ may well be acyclic even when $M\ne0$.

\smallskip
 (1)~For example, the complex $\Br^\bu_\alpha(A,M)$ is acyclic whenever
the augmentation ideal $A^+$, viewed as an associative $k$\+algebra,
has a unit of its own.
 The sequence of $k$\+linear maps $h\:A^+{}^{\ot n}\ot_kM\rarrow
A^+{}^{\ot n+1}\ot_kM$ given by the rules $h(a_1\ot\dotsb\ot a_n\ot m)=
e\ot a_1\ot\dotsb\ot a_n\ot m$ for all $a_i\in A^+$, \ $m\in $M, \ 
$1\le i\le n$, \ $n\ge0$, where $e$~is the unit in $A^+$, provides
a contracting homotopy.

 More generally, let $A$ and $B$ be two associative, unital
$k$\+algebras, and let $\alpha\:A\rarrow k$ be an augmentation
of~$A$.
 Consider the direct sum $A\oplus B$, and endow it with an augmented
$k$\+algebra structure $(\alpha,0)\:A\oplus B\rarrow k$ as in
Example~\ref{bar-of-algebras-quasi-isomorphism}.
 Pick a left $B$\+module $M$, and endow it with an $(A\oplus B)$\+module
structure via the natural $k$\+algebra homomorphism (the direct
summand projection) $A\oplus B\rarrow B$.
 Then the complex $\Br^\bu_{(\alpha,0)}(A\oplus B,\>M)$ is acyclic.

\smallskip
 (2)~To give a couple of other examples, consider the algebra of
polynomials in one variable $A=k[x]$ over the field~$k$, endowed with
the augmentation $\alpha\:A\rarrow k$ given by the rule $\alpha(x)=0$.
 For any element $a\in k$, denote by $k_a=A/(x-a)$ the one-dimensional
$A$\+module in which the generator $x\in A$ acts by the operator
of multiplication with~$a$.
 In particular, in the $A$\+module $k=k_0$ the algebra $A$ acts via
the augmentation~$\alpha$.
 Put $M=k_a$, where $a\ne0$.
 Then $\Tor^A_n(k,M)=0$ for all $n\ge0$, hence the complex
$\Br^\bu_\alpha(A,M)$ is acyclic.

 Alternatively, consider the $A$\+module $M=k[x,x^{-1}]$ of
Laurent polynomials in~$x$ (or the $A$\+module $M=k(x)$ of rational
functions in~$x$).
 Once again, in these cases $\Tor^A_n(k,M)=0$ for all $n\ge0$,
and the complex $\Br^\bu_\alpha(A,M)$ is acyclic.
\end{exs}

\subsection{Posing the problem}
\label{algebras-modules-posing-the-problem-subsecn}
 Let $A=(A,\alpha)$ be an augmented $k$\+algebra, and let $M^\bu$ be
a complex of left $A$\+modules.
 Then the bar-construction $\Br^\bu_\alpha(A,M^\bu)$ is a bicomplex of
$k$\+vector spaces.
 Let us totalize this bicomplex by taking infinite direct sums along
the diagonals.

 The problem of derived (nonhomogeneous) Koszul duality can be
formulated as follows.
 We would like to endow the complex $\Br^\bu_\alpha(A,M^\bu)$ with
some natural structure, and define an equivalence relation on complexes
with such structures, in a suitable way so that the assignment
$$
 M^\bu\longmapsto\Br^\bu_\alpha(A,M^\bu)
$$
would be a triangulated equivalence between the unbounded derived
category $\sD(A\modl)$ of complexes of $A$\+modules and the triangulated
category of complexes with the said structure up to the said
equivalence relation.

 What structure should it be, and what should be the equivalence
relation?
 The first question is easier to answer: the bar-construction
$\Br^\bu_\alpha(A)$ has a natural \emph{DG\+coalgebra} structure,
and the complexes $\Br^\bu(A,M^\bu)$ are \emph{DG\+comodules} over
$\Br^\bu_\alpha(A)$, as will be explained below in
Sections~\ref{coalgebra-structure-cobar-subsecn}\+-%
\ref{duality-formulated-for-complexes-of-modules-subsecn}.

 The second question is harder, because the conventional notion of
quasi-iso\-mor\-phism is \emph{not} up to the task, as
Examples~\ref{bar-of-modules-acyclic} illustrate.
 The complex $\Br^\bu_\alpha(A,M)$ can be quasi-isomorphic to zero
for a quite nonzero one-term complex of $A$\+modules $M=M^\bu$.
 The relevant definition of the \emph{coderived category} of
DG\+comodules will be spelled out in
Section~\ref{coderived-cdg-comodules-subsecn}.

\Section{Coalgebras and Comodules}

 Unless otherwise mentioned, all \emph{coalgebras} in this paper are
coassociative, counital coalgebras over the field~$k$, all
\emph{comodules} are counital, and all \emph{homomorphisms of
coalgebras} are compatible with the counits.

 In this section we present the simplest initial formulations of
derived nonhomogeneous Koszul duality for complexes of modules over
an augmented algebra $A$ and for complexes of comodules over
a conilpotent coalgebra~$C$.

\subsection{Coalgebras and comodules}
\label{coalgebras-and-comodules-subsecn}
 The standard reference sources on coalgebras and comodules over
a field are the books~\cite{Swe,Mon}.
 The present author's overview~\cite{Prev} can be used as an additional
source.

 A (\emph{coassociative, counital}) \emph{coalgebra} $C$ over
a field~$k$ is a $k$\+vector space endowed with $k$\+linear maps of
\emph{comultiplication} and \emph{counit}
$$
 \mu\:C\rarrow C\ot_k C
 \qquad \text{and} \qquad \epsilon\:C\rarrow k
$$
satisfying the following \emph{coassociativity} and \emph{counitality}
axioms.
 Firstly, the two compositions
$$
 C\rarrow C\ot_k C\rightrightarrows C\ot_k C\ot_k C
$$
must be equal to each other, $(\mu\ot\id_C)\circ\mu=
(\id_C\ot\mu)\circ\mu$.
 Secondly, both the compositions
$$
 C\rarrow C\ot_k C\rightrightarrows C
$$
must be equal to the identity map, $(\epsilon\ot\id_C)\circ\mu=\id_C=
(\id_C\ot\epsilon)\circ\mu$.

 A \emph{left comodule} $M$ over a coalgebra $C$ is a $k$\+vector space
endowed with a $k$\+linear \emph{left coaction} map
$$
 \nu\:M\rarrow C\ot_k M
$$
satisfying the following coassociativity and counitality axioms.
 Firstly, the two compositions
$$
 M\rarrow C\ot_k M\rightrightarrows C\ot_k C\ot_k M
$$
must be equal to each other, $(\mu\ot\id_M)\circ\nu=
(\id_C\ot\nu)\circ\nu$.
 Secondly, the composition
$$
 M\rarrow C\ot_k M\rarrow M
$$
must be equal to the identity map, $(\epsilon\ot\id_M)\circ\nu=\id_M$.

 A \emph{right comodule} $N$ over $C$ is a $k$\+vector space endowed
with a \emph{right coaction} map
$$
 \nu\:N\rarrow N\ot_k C
$$
satisfying the similar coassociativity and counitality axioms.

 Let $V$ be a $k$\+vector space.
 Then the comultiplication map~$\mu$ on a coalgebra $C$ induces a left
coaction map $C\ot_k V\rarrow C\ot_k C\ot_k V$ on the $k$\+vector space
$C\ot_k V$ and a right coaction map $V\ot_k C\rarrow V\ot_k C\ot_k C$
on the $k$\+vector space $V\ot_k C$.
 The left $C$\+comodule $C\ot_k V$ and the right $C$\+comodule $V\ot_k C$
are called the \emph{cofree $C$\+comodules} cogenerated by the vector
space~$V$.

 For any left $C$\+comodule $L$, the $k$\+vector space of all left
$C$\+comodule maps $L\rarrow C\ot_k V$ is naturally isomorphic to
the $k$\+vector space of all $k$\+linear maps $L\rarrow V$,
$$
 \Hom_C(L,\>C\ot_kV)\simeq\Hom_k(L,V).
$$
 Hence the ``cofree comodule'' terminology.

\subsection{DG-coalgebras and DG-comodules}
\label{dg-coalgebras-subsecn}
 The following definitions, going back at least to the paper~\cite{EM2},
can be also found in~\cite[Sections~2.1 and~2.3]{Pkoszul}.

 A \emph{graded coalgebra} over~$k$ is a graded $k$\+vector space
$C=\bigoplus_{i\in\boZ} C^i$ endowed with a coalgebra structure
such that both the comultiplication map $\mu\:C\rarrow C\ot_k C$ and
the counit map $\epsilon\:C\rarrow k$ are morphisms of graded vector
spaces (i.~e., homogeneous linear maps of degree~$0$).
 Here the standard induced grading is presumed on the tensor product
$C\ot_k C$, while the one-dimensional $k$\+vector space~$k$ is
endowed with the grading where it is placed in the degree $i=0$.

 A \emph{graded left comodule} over a graded coalgebra $C$ is
a graded $k$\+vector space $M=\bigoplus_{i\in\boZ} M^i$ endowed with
a left $C$\+comodule structure such that the coaction map
$\nu\:M\rarrow C\ot_k M$ is a morphism of graded vector spaces.
 \emph{Graded right $C$\+comodules} are defined similarly.

 In particular, for any graded $k$\+vector space $V$, the graded
$k$\+vector space $C\ot_k V$ has a natural structure of a cofree graded
left $C$\+comodule, while the tensor product $V\ot_k C$ is a cofree
graded right $C$\+comodule.

 A \emph{DG\+coalgebra} $C^\bu=(C,d)$ over~$k$ is a complex of
$k$\+vector spaces endowed with a coalgebra structure such that both
the comultiplication map $\mu\:C^\bu\rarrow C^\bu\ot_k C^\bu$ and
the counit map $\epsilon\:C^\bu\rarrow k$ are morphisms of complexes
of vector spaces (i.~e., homogeneous linear maps of degree~$0$
commuting with the differentials).
 Here the standard induced differential $d(c'\ot c'')=d(c')\ot c''
+(-1)^i c'\ot d(c'')$ for all $c'\in C^i$ and $c''\in C^j$ is
presumed on the tensor product of complexes $C^\bu\ot_k C^\bu$,
while the differential on the $k$\+vector space~$k$ is zero.

 A \emph{left DG\+comodule} $M^\bu=(M,d_M)$ over a DG\+coalgebra $C^\bu$
is a complex of $k$\+vector spaces endowed with a left $C$\+comodule
structure such that the coaction map $\nu\:M^\bu\rarrow
C^\bu\ot_k M^\bu$ is a morphism of complexes of $k$\+vector spaces.
 \emph{Right DG\+comodules} over $C^\bu$ are defined similarly.

\subsection{Coalgebra structure on the bar-construction}
\label{coalgebra-structure-cobar-subsecn}
 Let $V$ be a $k$\+vector space.
 Then the direct sum of the tensor powers of~$V$,
$$
 k\oplus V\oplus (V\ot_k V)\oplus (V\ot_k V\ot_k V)\oplus\dotsb
 \,\simeq\,\bigoplus\nolimits_{n=0}^\infty V^{\ot n}
$$
can be naturally endowed with a structure of graded associative
algebra over~$k$.
 The multiplication in this algebra, denoted by
$T(V)=\bigoplus_{n=0}^\infty V^{\ot n}$, is given by the rule
$$
 (v_1\ot\dotsb\ot v_p)(w_1\ot\dotsb\ot w_q)=
 v_1\ot\dotsb\ot v_p\ot w_1\ot\dotsb\ot w_q
$$
for all $v_i$, $w_j\in V$, \ $1\le i\le p$, \ $1\le j\le q$, \
$p$, $q\ge0$.
 The unit element in $T(V)$ is $1\in k=V^{\ot 0}$.
 The algebra $T(V)$ is the free associative, unital algebra spanned
by the vector space~$V$.

 The same graded $k$\+vector space $\bigoplus_{n=0}^\infty V^{\ot n}$
also has a natural structure of graded coassociative coalgebra
over~$k$.
 The comultiplication in this coalgebra, denoted by
$\udT(V)=\bigoplus_{n=0}^\infty V^{\ot n}$, is given by the rule
$$
 \mu(v_1\ot\dotsb\ot v_n)=\sum\nolimits_{p+q=n}^{p,q\ge0}
 (v_1\ot\dotsb\ot v_p)\ot(v_{p+1}\ot\dotsb\ot v_{p+q})
$$
for all $v_i\in V$, \ $1\le i\le n$, \ $n\ge0$.
 The counit on $\udT(V)$ is the direct summand projection map
$\udT(V)\rarrow V^{\ot0}=k$.
 We refer to Remark~\ref{tensor-coalgebra-cofree-conilpotent-remark}
for a discussion of a cofreeness property of the tensor coalgebra
$\udT(V)$.

 Let $(A,\alpha)$ be an augmented algebra over~$k$.
 We observe that the underlying graded vector space $\Br_\alpha(A)$
of the bar-construction $\Br^\bu_\alpha(A)$, as defined in
Section~\ref{augmented-algebras-subsecn}, coincides (up to a grading
sign change) with the graded $k$\+vector space
$\bigoplus_{n=0}^\infty A^+{}^{\ot n}$.
 Consequently, the graded vector space $\Br_\alpha(A)$ has natural
structures of an associative algebra and a coassociative coalgebra.

 It turns out that the bar differential~$\d$ on $\Br^\bu_\alpha(A)$
does \emph{not} respect the multiplication on $\Br_\alpha(A)=T(A^+)$,
i.~e., $\d$~does \emph{not} satisfy any kind of Leibniz rule with
respect to the multiplication on $T(A^+)$.
 However, the differential~$\d$ is compatible with the comultiplication
on $\Br_\alpha(A)=\udT(A^+)$.
 In other words, the complex $\Br^\bu_\alpha(A)$ endowed with the graded
coalgebra structure of $\udT(A^+)$ is a \emph{DG\+coalgebra}
in the sense of the definition in Section~\ref{dg-coalgebras-subsecn}.

\subsection{Derived Koszul duality formulated for complexes of modules}
\label{duality-formulated-for-complexes-of-modules-subsecn}
 Let $(A,\alpha)$ be an augmented $k$\+algebra, and let $M^\bu$ be
a complex of left $A$\+modules.
 We will denote simply by $M$ the underlying graded $A$\+module
of~$M^\bu$.

 Then the underlying graded vector space $\Br_\alpha(A,M)$ of
the bar-construction $\Br^\bu_\alpha(A,M^\bu)$, as defined in
Sections~\ref{bar-of-modules-subsecn}\+-%
\ref{algebras-modules-posing-the-problem-subsecn}, can be viewed as
a cofree graded left comodule $\Br_\alpha(A,M)=\udT(A^+)\ot_k M$
over the tensor coalgebra $\udT(A^+)$ (as explained in
Sections~\ref{coalgebras-and-comodules-subsecn}\+-%
\ref{dg-coalgebras-subsecn}).
 It turns out that the total differential on $\Br^\bu_\alpha(A,M^\bu)$
is compatible with the bar differential~$\d$ on the DG\+coalgebra
$\Br^\bu_\alpha(A)$ and the left coaction of $\Br_\alpha(A)$ in
$\Br_\alpha(A,M)$.
 In other words, the complex $\Br^\bu_\alpha(A,M^\bu)$ endowed with
the cofree graded comodule structure of $\udT(A^+)\ot_k M$ is
a \emph{left DG\+comodule} over the DG\+coalgebra $\Br^\bu_\alpha(A)$.

 The following theorem is our first formulation of the derived Koszul
duality.
 At this point in our exposition, it is more of an advertisement than
a precise claim, in that it contains details which have not been
defined yet but will be defined below.

\begin{thm} \label{complexes-of-modules-augmented-bar-construction-thm}
 Let $(A,\alpha)$ be an augmented associative algebra over a field~$k$.
 Then the assignment $M^\bu\longmapsto\Br^\bu_\alpha(A,M^\bu)$ induces
a triangulated equivalence between the derived category of left
$A$\+modules\/ $\sD(A\modl)$ and the \emph{coderived} category of
left DG\+comodules over the DG\+coalgebra\/ $\Br^\bu_\alpha(A)$,
$$
 \sD(A\modl)\simeq\sD^\co(\Br^\bu_\alpha(A)\comodl).
$$
\end{thm}

\begin{proof}
 This is a particular case of~\cite[Theorem~6.3(a)]{Pkoszul};
see~\cite[Th\'eor\`eme~2.2.2.2]{Lef} and~\cite[Section~4]{Kel} for
an earlier approach.
 The inverse functor, $\sD^\co(\Br^\bu_\alpha(A)\comodl)\allowbreak
\rarrow\sD(A\modl)$, assigns to a left DG\+co\-mod\-ule $N^\bu$ over
the DG\+coalgebra $C^\bu=\Br^\bu_\alpha(A)$ the graded left $A$\+module
$A\ot_k N$ endowed with a natural differential whose construction
will be explained in
Sections~\ref{twisted-differential-on-tensor-product-subsecn}\+-%
\ref{augmented-duality-comodule-side-subsecn}.
 The definition of the coderived category will be given in
Section~\ref{coderived-cdg-comodules-subsecn}.
\end{proof}

 The triangulated equivalence of
Theorem~\ref{complexes-of-modules-augmented-bar-construction-thm}
takes the irreducible left $A$\+module~$k$ (with the $A$\+module
structure defined in terms of the augmentation~$\alpha$) to
the cofree left DG\+comodule $C^\bu$ over the DG\+coalgebra
$C^\bu=\Br^\bu_\alpha(A)$ (cf.\ the discussion of cofree comodules
in Section~\ref{coalgebras-and-comodules-subsecn}).
 The same equivalence takes the left DG\+comodule~$k$ over~$C^\bu$
(with the $C$\+comodule structure on~$k$ defined in terms of the unique
coaugmentation of~$C$; see Sections~\ref{cobar-construction-subsecn}
and~\ref{duality-dg-algebras-dg-coalgebras-subsecn} below) to
the free left $A$\+module~$A$.

 For generalizations of
Theorem~\ref{complexes-of-modules-augmented-bar-construction-thm},
see Theorems~\ref{complexes-of-modules-quadratic-linear-thm}
and~\ref{augmented-acyclic-twisting-cochain-duality-thm} below.

\subsection{Nonhomogeneous quadratic dual DG-coalgebra}
\label{nonhomogeneos-quadratic-dual-to-augmented-subsecn}
 The DG\+coalgebra $\Br^\bu_\alpha(A)$ produced by the bar-construction
is rather big.
 In particular, it is essentially \emph{never} finite-dimensional
(unless $A=k$).
 The concept of \emph{nonhomogeneous quadratic duality}, going back
to~\cite{Kos}, \cite[Section~3]{Prid}, \cite[Section~2.3]{FT},
and~\cite{Pcurv}, sometimes allows to produce a smaller DG\+coalgebra
that can be used in lieu of $\Br^\bu_\alpha(A)$ in Koszul duality
theorems such as
Theorem~\ref{complexes-of-modules-augmented-bar-construction-thm}.
{\hbadness=2400\par}

 Let $A=\bigoplus_{n=0}^\infty A_n$ be a nonnegatively graded
$k$\+algebra with $A_0=k$.
 Then the direct summand inclusion $A_1\rarrow A$ can be uniquely
extended to a morphism of graded algebras $\pi\:T(A_1)\rarrow A$.
 The algebra $A$ is said to be \emph{quadratic} if
the homomorphism~$\pi$ is surjective and its kernel is generated by
elements of degree~$2$ (simply speaking, this means that the algebra $A$
is generated by elements of degree~$1$ with relations in degree~$2$).

 The graded algebra $A$ is called \emph{Koszul} if $\Tor^A_{ij}(k,k)=0$
for all $i\ne j$.
 Here the first grading~$i$ on the $\Tor$ spaces is the usual
homological grading, while the second grading~$j$, called
the \emph{internal} grading, is induced by the grading of~$A$.
 This condition for $i=1$ means precisely that the algebra $A$ is
generated by $A_1$; assuming this is the case, the same condition
for $i=2$ means precisely that $A$ is defined by quadratic relations.
 So any Koszul graded algebra is quadratic~\cite[Corollary~5.3 in
Chapter~1]{PP}.

 Let $A=\bigoplus_{n=0}^\infty A_n$ be a quadratic algebra.
 Put $V=A_1$, and denote by $I\subset V\ot_k V$ the kernel of
the (surjective) multiplication map $A_1\ot_k A_1\rarrow A_2$.
 So $V$ is the space of generators of the quadratic algebra $A$,
and $I$ is the space of defining quadratic relations of~$A$.
 The grading components of $A$ can be expressed in terms of $V$ and $I$
by the formula
$$
 A_n=V^{\ot n}\Big/\sum\nolimits_{i=1}^{n-1}
 V^{\ot i-1}\ot_k I\ot_k V^{\ot n-i-1}, \qquad n\ge 1.
$$

 Consider the tensor coalgebra $\udT(V)$, and denote by $C=A^?$
the graded subcoalgebra in $\udT(V)$ with the components $C^0=k$, \
$C^1=V$, \ $C^2=I$, and generally
$$
 C^n=\bigcap\nolimits_{i=1}^{n-1} V^{\ot i-1}\ot_k I\ot_k V^{\ot n-i-1}
 \,\subset\,V^{\ot n}, \qquad n\ge1.
$$
 The coalgebra $C$ is called the \emph{quadratic dual coalgebra} to
a quadratic algebra~$A$ \,\cite[Section~2.1]{PV}.

\begin{rem}
 In most expositions, including the present author's~\cite{Pcurv}
and~\cite{PP}, the quadratic duality is viewed as a contravariant
functor constructed using the passage to the dual vector space.
 The contravariant duality assigns to a quadratic algebra $A$
the quadratic algebra~$A^!$, which can be obtained as the graded dual
$k$\+vector space to the coalgebra $C=A^?$, that is, $A^!_n=(C^n)^*$
for all $n\ge0$.
 The construction of the quadratic algebra $A^!$ works well for
quadratic algebras $A$ with finite-dimensional components (see,
e.~g., \cite[Section~3 in Chapter~1]{PP}), and it can be made to work
without the assumption of locally finite dimension by considering
linearly compact topological vector spaces.
 Our general preference is to use coalgebras instead, and consider
the covariant quadratic duality between algebras and coalgebras as
constructed above.
 In particular, the intended applications in~\cite{PV} were quadratic
algebras with infinite-dimensional components; that is one reason
why the covariant algebra-coalgebra quadratic duality was introduced
in~\cite{PV}.
 In the context of derived nonhomogeneous Koszul duality over a field,
which is the subject of this survey, the language of covariant duality
is more illuminating, in our view.
\end{rem}

 Let $0=F_{-1}A\subset k=F_0A\subset F_1A\subset F_2A\subset\dotsb$ be
a $k$\+algebra endowed with an increasing filtration $F$ such that
$A=\bigcup_{n\ge0} F_nA$, \ $1\in F_0A$, and $F_nA\cdot F_mA\subset
F_{n+m}A$ for all $n$, $m\ge0$.
 Then the associated graded vector space $\gr^FA=\bigoplus_{n=0}^\infty
F_nA/F_{n-1}A$ carries a naturally induced associative algebra
structure.
 A \emph{nonhomogeneous quadratic algebra} is a filtered algebra
$(A,F)$ such that the graded algebra $\gr^FA$ is quadratic.
 A \emph{nonhomogeneous Koszul algebra} is a filtered algebra
$(A,F)$ such that the graded algebra $\gr^FA$ is
Koszul~\cite[Section~2]{Prid}, \cite[Chapter~5]{PP}.

\begin{ex} \label{trivial-filtration-example}
 (1)~Let $V$ be a $k$\+vector space.
 Consider the graded algebra $A=\bigoplus_{n=0}^\infty A_n$ with
the components $A_0=k$, \ $A_1=V$, and $A_n=0$ for all $n\ge2$.
 The multiplication on the graded algebra $A$ is trivial:
the multiplication map $A_n\otimes A_m\rarrow A_{n+m}$ vanishes
for all $n$, $m>0$ (in the cases when $n=0$ or $m=0$
the multiplication map is determined by the condition that
$1\in k=A_0$ is the unit element of~$A$).
 The graded algebra $A$ is well-known to be Koszul (see, e.~g.,
\cite[Corollary~2.4 or~4.3 in Chapter~2]{PP} for much more
general results).
 In this example, the quadratic dual coalgebra $C=A^?$ is the whole
tensor coalgebra, that is, $C=\udT(V)$.

\smallskip
 (2)~Any nonzero associative algebra $A$ can be endowed with
the \emph{trivial} filtration, $F_0A=k\cdot 1$ and $F_1A=A$.
 Then the associated graded algebra $\gr^FA$ has grading components
$\gr^F_0A=k$, \ $\gr^F_1A=A/(k\cdot\nobreak1)$, and $\gr^F_nA=0$
for $n\ge2$.
 So the algebra $\gr^FA$ has the form of example~(1) above (with
$V=A/(k\cdot\nobreak1)$), and the multiplication on it is trivial in
the sense we have explained.
 Thus the passage to the associated graded algebra of an algebra $A$
with respect to the trivial filtration $F$ destroys all the information
about the multiplication in~$A$.

 It is clear from this discussion that \emph{any} nonzero associative
algebra $A$ endowed with the trivial filtration $F$ is a nonhomogeneous
Koszul algebra.
 Some associative algebras $A$ admit nontrivial filtrations making them
nonhomogeneous Koszul algebras with a more interesting associated graded
algebra $\gr^FA$ and a smaller quadratic dual coalgebra $C=(\gr^FA)^?$
(cf.\ Example~\ref{chevalley-eilenberg} below).
\end{ex}

 Let $(A,F)$ be a nonhomogeneous quadratic algebra and $C=(\gr^FA)^?$
be the quadratic dual coalgebra to the quadratic algebra $\gr^FA$.
 Assume additionally that the algebra $A$ is endowed with
an augmentation $\alpha\:A\rarrow k$ with the augmentation ideal
$A^+\subset A$.
 Then there is a natural isomorphism $\gr^F_1A=F_1A/F_0A\simeq
A^+\cap F_1A\subset A$.
 So the vector space $V=\gr^F_1A$ can be viewed as a subspace in $A^+$,
and the graded coalgebra $C$ can be viewed as a subcoalgebra in
$\udT(A^+)$,
$$
 C\subset\udT(V)\subset\udT(A^+)=\Br_\alpha(A).
$$

 One observes that the subcoalgebra $C\subset\Br_\alpha(A)$ is
in fact a DG\+subcoalgebra $C^\bu\subset\Br^\bu_\alpha(A)$, that is
$\d(C)\subset C$.
 Essentially, the first reason for that is because the multiplication
map $A^+\ot_k A^+\rarrow A^+$ restricted to the subspace
$$
 I\subset\gr^F_1A\ot_k\gr^F_1A\simeq (A^+\cap F_1A)\ot_k(A^+\cap F_1A)
 \subset A^+\ot_k A^+
$$
lands within the subspace $A^+\cap F_1 A \subset A^+$ (by the definition
of $I$ as the kernel of the multiplication map $\gr^F_1A\ot_k\gr^F_1A
\rarrow\gr^F_2A=F_2A/F_1A$ in the graded algebra $\gr^FA$).
 Further details need to be checked; we refer
to~\cite[Section~3]{Prid}, \cite[Proposition~2.2]{Pcurv}
or~\cite[Proposition~4.1 in Chapter~5]{PP}.
 A generalization to nonaugmented algebras $A$ will be discussed in
Example~\ref{curved-nonhomog-quadratic-twisting-cochain} below.

 Let $(A,F,\alpha)$ be an augmented nonhomogeneous quadratic algebra
and $M^\bu$ be a complex of left $A$\+modules.
 Then one can easily see that $C\ot_kM\subset\udT(A^+)\ot_kM=
\Br_\alpha(A,M)$ is a subcomplex in $\Br^\bu_\alpha(A,M)$, that is
$\d(C\ot_k M)\subset C\ot_k M$.
 In anticipation of the discussion in
Section~\ref{twisting-cochains-secn}, we denote the complex $C\ot_kM$
with the differential induced by the differential~$\d$ on
$\Br^\bu_\alpha(A,M)$ by $C^\bu\ot^\tau M^\bu$.
 This notation is indented to emphasize the fact that the differential
on $C^\bu\ot^\tau M^\bu$ is \emph{not} simply the tensor product
differential on the tensor product of two complexes $C^\bu\ot_k M^\bu$,
but rather some twisted version of it taking the action of $A$ in $M$
into account.

 Let $(A,F,\alpha)$ be a nonhomogeneous Koszul algebra.
 Then the DG\+coalgebra $C^\bu$ is quasi-isomorphic to the ambient
DG\+coalgebra $\Br^\bu_\alpha(A)$, and the complex $C^\bu\ot^\tau
M^\bu$ is quasi-isomorphic to the ambient bar-complex
$\Br^\bu_\alpha(A,M)$.
 Accordingly, one has
$$
 H^{-n}(C^\bu)\simeq\Tor_n^A(k,k)
 \quad\text{and}\quad
 H^{-n}(C^\bu\ot^\tau M)\simeq\Tor_n^A(k,M)
$$
for any left $A$\+module $M$ and all $n\ge0$.

\begin{thm} \label{complexes-of-modules-quadratic-linear-thm}
 Let $(A,F,\alpha)$ be an augmented nonhomogeneous Koszul algebra
over a field~$k$.
 Then the assignment $M^\bu\longmapsto C^\bu\ot^\tau M^\bu$ induces
a triangulated equivalence between the derived category of left
$A$\+modules\/ $\sD(A\modl)$ and the \emph{coderived} category of left
DG\+comodules over the DG\+coalgebra~$C^\bu$,
$$
 \sD(A\modl)\simeq\sD^\co(C^\bu\comodl).
$$
\end{thm}

\begin{proof}
 This is a particular case of~\cite[Theorems~6.5(a) and~6.6]{Pkoszul}.
 The inverse functor, $\sD^\co(C^\bu\comodl)\rarrow\sD(A\modl)$, assigns
to a left DG\+comodule $N^\bu$ over $C^\bu$ the complex of left
$A$\+modules $A\ot^\tau N^\bu$, which means the graded left $A$\+module
$A\ot_k N$ endowed with a twisted differential, as explained below
in Sections~\ref{twisted-differential-on-tensor-product-subsecn}\+-%
\ref{augmented-duality-comodule-side-subsecn}.
 The definition of the coderived category will be given
in Section~\ref{coderived-cdg-comodules-subsecn}.
\end{proof}

 Any associative algebra $A$ can be viewed as a nonhomogeneous Koszul
algebra with the trivial filtration, as explained in
Example~\ref{trivial-filtration-example}(2).
 For an augmented algebra $A$ endowed with the trivial filtration,
one has $C^\bu=\Br^\bu_\alpha(A)$.
 So Theorem~\ref{complexes-of-modules-augmented-bar-construction-thm}
can be obtained as a particular case of
Theorem~\ref{complexes-of-modules-quadratic-linear-thm} for
the trivial filtration~$F$.

 For a generalization of
Theorem~\ref{complexes-of-modules-quadratic-linear-thm}, see
Theorem~\ref{augmented-acyclic-twisting-cochain-duality-thm} below.

\begin{ex} \label{chevalley-eilenberg}
 Let $\g$ be a Lie algebra over a field~$k$ and $\Lambda(\g)$ be
the exterior algebra spanned by the $k$\+vector space~$\g$.
 Then the Chevalley--Eilenberg differential~$\d$ on $\Lambda(\g)$
makes $\Lambda(\g)$ a complex computing the homology spaces
$H_*(\g,k)$ of the Lie algebra $\g$ with coefficients in
the trivial $\g$\+module~$k$ \,\cite[Chapter~III]{ChE}, \cite{Kos}.
 The differential~$\d$ does not respect the exterior multiplication
on $\Lambda(\g)$, but it is an odd coderivation of the exterior
comultiplication; so $C^\bu=(\Lambda(\g),\d)$ is a DG\+coalgebra.

 The enveloping algebra $U(\g)$ is endowed with the natural
(Poincar\'e--Birkhoff--Witt) filtration $F$, making $U(\g)$
(the thematic example of) a nonhomogeneous Koszul algebra over
a field.
 There is also a natural augmentation~$\alpha$ on $U(\g)$,
corresponding to the action of~$\g$ in trivial $\g$\+modules.
 The construction above assigns to the augmented nonhomogeneous
quadratic algebra $(U(g),F,\alpha)$ the DG\+coalgebra
$(\Lambda(\g),\d)$.

 Moreover, for any $\g$\+module $M$, the homological
Chevalley--Eilenberg complex $(\Lambda(\g)\ot_kM,\>\d)$ of
the Lie algebra~$\g$ with coefficients in $M$ is a DG\+comodule
over the $(\Lambda(\g),\d)$.
 The construction above assigns the DG\+comodule $C^\bu\ot^\tau M
=(\Lambda(\g)\ot_kM,\>\d)$ to the $U(\g)$\+module~$M$.

 Thus Theorem~\ref{complexes-of-modules-quadratic-linear-thm}
claims that the construction of the homological Chevalley--Eilenberg
complex induces an equivalence between the derived category of
$\g$\+modules and the coderived category of DG\+comodules over
the DG\+coalgebra $(\Lambda(\g),\d)$.
 For a finite-dimensional Lie algebra~$\g$, the coalgebra
$\Lambda(\g)$ is finite-dimensional; so one can pass to the dual
algebra $\Lambda(\g^*)$ and say that the construction of
the cohomological Chevalley--Eilenberg complex
$(\Lambda(\g^*)\ot_k M,\>d)$ induces an equivalence between
the derived category of $\g$\+modules and the coderived category
of DG\+modules over the DG\+algebra
$\Lambda^\bu(\g^*)=(\Lambda(\g^*),d)$,
\begin{equation} \label{fin-dim-lie-algebra-koszul-duality}
 \sD(\g\modl)\simeq\sD^\co(\Lambda^\bu(\g^*)\modl).
\end{equation}
 This example can be found in~\cite[Example~6.6]{Pkoszul}.

 Let us emphasize that these assertions are certainly \emph{not}
true for the conventional derived category of DG\+(co)modules
in place of the coderived category.
 In fact, the derived category $\sD(\g\modl)$ is \emph{not}
equivalent to the derived category $\sD(\Lambda^\bu(\g^*)\modl)$
already when $\g$~is a finite-dimensional semisimple Lie algebra
over a field~$k$ of characteristic~$0$, or indeed, a one-dimensional
abelian Lie algebra (cf.\ Example~\ref{bar-of-modules-acyclic}(2)).
 We will continue this discussion in
Example~\ref{semisimple-chevalley-eilenberg} below.
\end{ex}

 A relative version of
Theorem~\ref{complexes-of-modules-quadratic-linear-thm}, with
the ground field~$k$ replaced by an arbitrary associative ring $R$,
can be found in the book~\cite{Prel}.
 See~\cite[Sections~3.8 and~6.6]{Prel}.

\subsection{Cobar-construction; DG-algebra and DG-module structures}
\label{cobar-construction-subsecn}
 Let $C$ be a (coassociative, counital) coalgebra over a field~$k$,
as defined in Section~\ref{coalgebras-and-comodules-subsecn}.
 A \emph{coaugmentation}~$\gamma$ on $C$ is a (counital) coalgebra
homomorphism $\gamma\:k\rarrow C$.
 So the composition $k\overset\gamma\rarrow C\overset\epsilon\rarrow k$
is the identity map.
 We denote by $C^+=\coker(\gamma)$ the cokernel of the augmentation map.
 Generally speaking, $C^+$ is a coassociative coalgebra without counit
(dually to the augmentation ideal, which is an associative algebra
without unit).
 The coalgebra homomorphism $\gamma\:k\rarrow C$ endows
the one-dimensional $k$\+vector space~$k$ with left and right
$C$\+comodule structures.

 The \emph{cobar-construction} $\Cb^\bu_\gamma(C)$ of a coaugmented
coalgebra $C=(C,\gamma)$ is defined as
the complex~\cite[Chapter~II]{EML}, \cite{Ad}
$$
 k\overset0\lrarrow C^+\overset\d\lrarrow C^+\ot_k C^+
 \overset\d\lrarrow C^+\ot_k C^+\ot_k C^+\lrarrow\dotsb
$$
with the differential given by the formulas $\d(c)=\mu^+(c)$, \ 
$\d(c_1\ot c_2)=\mu^+(c_1)\ot c_2-c_1\ot\mu^+(c_2)$,~\dots,
\begin{multline*}
 \d(c_1\ot\dotsb\ot c_n)=\mu^+(c_1)\ot c_2\ot\dotsb\ot c_n - \dotsb \\
 + (-1)^{i-1} c_1\ot\dotsb\ot c_{i-1}\ot\mu^+(c_i)\ot c_{i+1}\ot\dotsb
 \ot c_n+\dotsb+(-1)^{n-1}c_1\ot\dotsb\ot c_{n-1}\ot\mu^+(c_n),
\end{multline*}
etc., for all $c$, $c_i\in C^+$, \ $1\le i\le n$, and $n\ge1$.
 Here $\mu^+\:C^+\rarrow C^+\ot_k C^+$ is the comultiplication map
on~$C^+$ (induced by the comultiplication map $\mu\:C\rarrow C\ot_k C$).
 The leftmost differential $\d\:k\rarrow C^+$ is the zero map.

 Let $(C,\gamma)$ be a coaugmented coalgebra and $M$ be a left
$C$\+comodule.
 The \emph{cobar-construction} $\Cb^\bu_\gamma(C,M)$ of $C$ with
the coefficients in $M$ is the complex
$$
 M\overset\d\lrarrow C^+\ot_k M \overset\d\lrarrow
 C^+\ot_k C^+\ot_k M\lrarrow\dotsb
$$
with the differentials given by the formulas $\d(m)=\nu^+(m)$, \
$\d(c\ot m)=\mu^+(c)\ot m-c\ot\nu^+(m)$,~\dots,
\begin{multline*}
 \d(c_1\ot\dotsb\ot c_n\ot m)=
 \mu^+(c_1)\ot c_2\ot\dotsb\ot c_n\ot m-\dotsb \\
 +(-1)^{n-1} c_1\ot\dotsb\ot c_{n-1}\ot\mu^+(c_n)\ot m
 +(-1)^n c_1\ot\dotsb\ot c_n\ot\nu^+(m),
\end{multline*}
etc., for all $c_i\in C^+$, \ $1\le i\le n$, and $n\ge0$.
 Here $\nu^+\:M\rarrow C^+\ot_k M$ is the left $C^+$\+coaction map
obtained by composing the left $C$\+coaction map $\nu\:M\rarrow
C\ot_k M$ with the natural surjection $C\ot_k M\rarrow C^+\ot_k M$.

 Dually to the discussion in
Sections~\ref{coalgebra-structure-cobar-subsecn}\+-%
\ref{duality-formulated-for-complexes-of-modules-subsecn},
one observes that the underlying graded vector space $\Cb_\gamma(C)$
of the cobar-construction $\Cb^\bu_\gamma(C)$ coincides with
the graded $k$\+vector space $\bigoplus_{n=0}^\infty C^+{}^{\ot n}$.
 Consequently, the graded vector space $\Cb_\gamma(C)$ has natural
structures of an associative algebra and a coassociative coalgebra.

 It turns out that the cobar differential~$\d$ on $\Cb^\bu_\gamma(C)$
does \emph{not} respect the comultiplication on $\Cb_\gamma(C)=
\udT(C^+)$.
 However, the differential~$\d$ is compatible with the multiplication on
$\Cb_\gamma(C)=T(C^+)$, i.~e., it satisfies the Leibniz rule with signs
$$
 \d(ab)=\d(a)b+(-1)^{|a|}a\d(b)
$$
for any elements $a$ and $b\in T(C^+)$ of degrees $|a|$ and~$|b|$.
 In other words, the complex $\Cb^\bu_\gamma(C)$ endowed with
the graded algebra structure of $T(C^+)$ is a DG\+algebra over~$k$.

 The graded algebra of cohomology of the DG\+algebra $\Cb^\bu_\gamma(C)$
is naturally isomorphic to (the graded algebra with the opposite
multiplication to) the Ext algebra of the left $C$\+comodule~$k$,
$$
 H^*\Cb^\bu_\gamma(C)\simeq\Ext_C^*(k,k),
$$
where $k$~is endowed with the left $C$\+comodule structure via~$\gamma$
and the Ext is computed in the abelian category of left $C$\+comodules.
 Similarly, for any left $C$\+comodule $M$, the graded module of
cohomology of the DG\+module $\Cb^\bu_\gamma(C,M)$ over
the DG\+algebra $\Cb^\bu_\gamma(C)$ (see the discussion in the next
Section~\ref{duality-formulated-for-complexes-of-comodules})
is naturally isomorphic to the Ext module $\Ext_C^*(k,M)$ over
the algebra $\Ext_C^*(k,k)$,
$$
  H^*\Cb^\bu_\gamma(C,M)\simeq\Ext_C^*(k,M).
$$

\subsection{Derived Koszul duality formulated for complexes of comodules}
\label{duality-formulated-for-complexes-of-comodules}
 Let $N^\bu$ be a complex of left $C$\+comodules.
 Then the cobar-construction $\Cb^\bu_\gamma(C,N^\bu)$ is a bicomplex
of $k$\+vector spaces.
 We totalize this bicomplex by taking infinite direct sums along
the diagonals.

 Denote simply by $N$ the underlying graded $C$\+comodule of
a complex of left $C$\+comodules~$N^\bu$.
 The underlying graded vector space $\Cb_\gamma(C,N)$ of
the cobar-construction $\Cb^\bu_\gamma(C,N^\bu)$ can be viewed as
a free graded left module $T(C^+)\ot_k N$ over the tensor algebra
$T(C^+)$.
 It turns out that the total differential on $\Cb^\bu_\gamma(C,N^\bu)$
is compatible with the cobar differential~$\d$ on the DG\+algebra
$\Cb^\bu(C)$ and the left action of $\Cb_\gamma(C)$ in
$\Cb_\gamma(C,N)$; i.~e., a Leibniz rule with signs similar to
the one in Section~\ref{cobar-construction-subsecn} is satisfied
for this action.
 In other words, the complex $\Cb^\bu_\gamma(C,N^\bu)$ endowed with
the free graded module structure of $T(C^+)\ot_k N$ is
a left DG\+module over the DG\+algebra $\Cb^\bu_\gamma(C)$.

 We formulate two versions of the derived Koszul duality for
complexes of comodules, the nonconilpotent and the conilpotent one.
 The definitions of the coderived and the absolute derived categories
will be given in Sections~\ref{co-contra-derived-cdg-modules-subsecn}
and~\ref{coderived-cdg-comodules-subsecn}.

\begin{thm} \label{complexes-of-comodules-coaugmented-cobar-thm}
 Let $(C,\gamma)$ be a coaugmented coassociative coalgebra over
a field~$k$.
 Then the assignment $N^\bu\longmapsto\Cb^\bu_\gamma(C,N^\bu)$
induces a triangulated equivalence between the \emph{coderived}
category of left $C$\+comodules\/ $\sD^\co(C\comodl)$ and
the \emph{absolute derived} category of left DG\+modules over
the DG\+algebra\/ $\Cb^\bu_\gamma(C)$,
$$
 \sD^\co(C\comodl)\simeq\sD^\abs(\Cb^\bu_\gamma(C)\modl).
$$
\end{thm}

\begin{proof}
 This is a particular case of~\cite[Theorem~6.7(a)]{Pkoszul}.
 The absolute derived category $\sD^\abs(\Cb^\bu_\gamma(C)\modl)$
coincides with the coderived category $\sD^\co(\Cb^\bu_\gamma(C)\modl)$
by~\cite[Theorem~3.6(a)]{Pkoszul}; see
Theorem~\ref{fin-homol-dim-derived-cdg-modules}(a) below.
\end{proof}

\begin{thm} \label{complexes-of-comodules-conilpotent-cobar-thm}
 Let $C$ be a \emph{conilpotent} coassociative coalgebra over
a field~$k$, endowed with its unique coaugmentation~$\gamma$.
 Then the same assignment $N^\bu\longmapsto\Cb^\bu_\gamma(C,N^\bu)$
induces a triangulated equivalence between the \emph{coderived}
category of left $C$\+comodules\/ $\sD^\co(C\comodl)$ and
the conventional \emph{derived} category of left DG\+modules over
the DG\+algebra\/ $\Cb^\bu_\gamma(C)$,
$$
 \sD^\co(C\comodl)\simeq\sD(\Cb^\bu_\gamma(C)\modl).
$$
\end{thm}

\begin{proof}
 This is a particular case of~\cite[Theorem~6.4(a)]{Pkoszul};
see~\cite[Th\'eor\`eme~2.2.2.2]{Lef} and~\cite[Section~4]{Kel} for
an earlier approach.
 It will be explained in the next
Section~\ref{structure-theory-of-coalgebras-secn} what
a ``conilpotent coalgebra'' is.
\end{proof}

 In both the theorems, the inverse functor assigns to a left
DG\+module $M^\bu$ over the DG\+algebra $A^\bu=\Cb^\bu_\gamma(C)$
the complex of left $C$\+comodules $C\ot^\tau M^\bu$, which is
a notation for the graded left $C$\+comodule $C\ot_k M$ endowed
with a twisted differential, as constructed below in
Sections~\ref{twisted-differential-on-tensor-product-subsecn}\+-%
\ref{augmented-duality-comodule-side-subsecn}.

 In both Theorems~\ref{complexes-of-comodules-coaugmented-cobar-thm}
and~\ref{complexes-of-comodules-conilpotent-cobar-thm},
the triangulated equivalence takes the irreducible left
$C$\+comodule~$k$ (with the $C$\+comodule structure defined in terms
of the coaugmentation~$\gamma$) to the free left DG\+module $A^\bu$
over the DG\+algebra $A^\bu=\Cb^\bu_\gamma(C)$.
 The same equivalences take the left DG\+module~$k$ over $A^\bu$
(with the $A$\+module structure on~$k$ defined in terms of the natural
augmentation of~$A^\bu$; see
Section~\ref{duality-dg-algebras-dg-coalgebras-subsecn}) to
the cofree (injective) left $C$\+comodule~$C$.

 For generalizations of
Theorems~\ref{complexes-of-comodules-coaugmented-cobar-thm}\+-%
\ref{complexes-of-comodules-conilpotent-cobar-thm}, see
Section~\ref{augmented-duality-comodule-side-subsecn} below.

\begin{rem} \label{inverting-the-arrows-remark}
 One can obtain the definition of a coalgebra from the definition
of an algebra over a field by inverting the arrows, but inverting
the arrows in the formulation of
Theorem~\ref{complexes-of-modules-augmented-bar-construction-thm}
does not exactly produce the formulation of
Theorem~\ref{complexes-of-comodules-coaugmented-cobar-thm}
or~\ref{complexes-of-comodules-conilpotent-cobar-thm}.
 The roles of algebras and coalgebras in derived nonhomogeneous
Koszul duality are \emph{not} symmetric up to inverting the arrows,
as there are subtle differences between the constructions on
the algebra and coalgebra side.

 Of course, one can say that the associative algebras are the monoid
objects in the monoidal (tensor) category of vector spaces, while
the coassociative coalgebras are the monoid objects in the monoidal
category opposite to vector spaces.
 Then the point is that the monoidal category of infinite-dimensional
vector spaces is \emph{not} equivalent to its opposite monoidal
category, and moreover, the properties of these categories are different
in a way important for nonhomogeneous Koszul duality.

 To begin with, the tensor product of vector spaces preserves infinite
direct sums, but not infinite products.
 Furthermore, the functor of tensor product of vector spaces (with one
of the two objects fixed) has a right adjoint, but not a left adjoint
functor.
 It is also relevant that the (filtered) direct limits of vector spaces
are exact functors, but the inverse limits aren't.
 These differences cause the lack of symmetry between the Koszul
duality theorems formulated above and in
Sections~\ref{duality-formulated-for-complexes-of-modules-subsecn}\+-%
\ref{nonhomogeneos-quadratic-dual-to-augmented-subsecn}.

 In particular, it is illuminating to observe that the constructions of
the DG\+comod\-ule $\Br^\bu_\alpha(A,M^\bu)$ and the DG\+module
$\Cb^\bu_\gamma(C,N^\bu)$ are \emph{not} transformed into one another
by inverting the arrows, for the subtle reason that the direct sum
(rather than the direct product) totalization of the respective
bicomplex is used in \emph{both} the constructions.
 The fact that the tensor products preserve the direct sums, but not
the direct products of vector spaces explains this choice of
the totalization procedures, but it breaks the symmetry between
the algebra and coalgebra sides.
\end{rem}

\Section{Structure Theory of Coalgebras; Conilpotency}
\label{structure-theory-of-coalgebras-secn}

 All coassociative coalgebras over a field are unions of their
finite-dimensional subcoalgebras.
 For this reason, the structure theory of coalgebras stands
approximately on the same level of complexity as the structure
theory of finite-dimensional algebras.
 So it is much simpler than the structure theory of infinite-dimensional
algebras and rings.
 Some details of this argument are elaborated upon in this section.

\subsection{Brief remarks about coalgebras}
\label{brief-remarks-about-coalgebras-subsecn}
 The theory of coalgebras is notoriously counterintuitive.
 Most algebraists seem to find it difficult to invert the arrows
properly in their mind in order to pass from the familiar ring theory
to the coalgebra theory.
 One needs to get used to that.

 Here we only mention a couple of obvious points.
 In ring theory, subrings are aplenty and quotient rings are rare.
 To pass to the quotient ring, one needs to have an ideal in
the ring.
 All ideals are subalgebras (without unit), but most nonunital
subalgebras are not ideals.
 Furthermore, all two-sided ideals are both left ideals and right
ideals, but most left or right ideals are not (two-sided) ideals.

 In coalgebras, the situation is reversed.
 A quotient coalgebra is the dual concept to a subalgebra, and
a subcoalgebra is the dual concept to a quotient algebra.
 Quotient coalgebras are aplenty, but subcoalgebras are rare.

 To pass to the quotient coalgebra, it suffices to have a coideal
in the coalgebra.
 If one is satisfied with obtaning a noncounital quotient coalgebra,
then the definition of a coideal reduces to this: a vector subspace $J$
in a coalgebra $C$ is said to be a \emph{coideal} (in the noncounital
sense) if $\mu(J)\subset J\ot_k C+C\ot_k J$ (where
$\mu\:C\rarrow C\ot_kC$ is the comultiplication map).
 The definition of a coideal in the counital sense includes
an additional condition that the counit $\epsilon\:C\rarrow k$
must vanish on a coideal~$J$.

 For comparison, a \emph{subcoalgebra} $D\subset C$ is a vector
subspace such that $\mu(D)\subset D\ot_kD$.
 One immediately observes that any subcoalgebra is a coideal (in
the noncounital sense), but most coideals are not subcoalgebras.
 Furthermore, any left coideal (i.~e., a left subcomodule in
the left $C$\+comodule~$C$) is a two-sided coideal, and any right
coideal is a two-sided coideal (in the noncounital sense), but most
(two-sided) coideals are neither left nor right coideals.

 It is also worth mentioning that the comodule structures induced
via homomorphisms of coalgebras are pushed forward rather than
pulled back.
 So if $f\:C\rarrow D$ is a homomorphism of coalgebras, then any
(left or right) $C$\+comodule acquires the induced structure of
a left $D$\+comodule.
 The resulting exact, faithful functor $C\comodl\rarrow D\comodl$
on the abelian categories of comodules is called
the \emph{corestriction of scalars} with respect to~$f$.

 We refer to~\cite[Section~1.4]{Swe} for further details.

\subsection{Local finite-dimensionality}
\label{local-finite-dimensionality-subsecn}
 The following classical result can be found
in~\cite[Corollary~2.1.4 and Theorem~2.2.1]{Swe}.

\begin{lem} \label{coalgebras-comodules-loc-fin-dim-lemma}
\textup{(a)} Any coassociative coalgebra $C$ is the (directed) union of
its finite-dimensional subcoalgebras.
 The sum of any two finite-dimensional subcoalgebras in $C$ is
a finite-dimensional subcoalgebra in~$C$.\par
\textup{(b)} Any (left) comodule over a coassociative coalgebra $C$
is the union of its finite-dimensional $C$\+subcomodules. \par
\textup{(c)} Any finite-dimensional $C$\+comodule $M$ is a comodule
over some finite-di\-men\-sional subcoalgebra $E\subset C$ (so the image
of the coaction map $\nu\:M\rarrow C\ot_k M$ is contained in $E\ot_k M$).
\end{lem}

\begin{proof}
 Let us sketch a proof of the first assertion of part~(a); the other
assertions are similar or simpler.
 Given a coassociative, counital coalgebra $C$, consider an arbitrary
finite-dimensional vector subspace $V\subset C$, and denote by
$E\subset C$ the full preimage of $C\ot_k V\ot_k C$ under the iterated
comultiplication map $\mu^{(2)}=(\mu\ot\id_C)\circ\mu=(\id_C\ot\mu)
\circ\mu\:C\rarrow C\ot_k C\ot_k C$.
 Then $E=E(V)$ is a subcoalgebra in $C$, one has $E\subset V$ (so $E$ is
finite-dimensional), and $C=\bigcup_{V\subset C}E(V)$.
\end{proof}

 Specifying a coalgebra structure on a finite-dimensional $k$\+vector
space $E$ is equivalent to specifying an algebra structure on the dual
vector space~$E^*$.
 So the dual vector space to a finite-dimensional coalgebra is
a finite-dimensional algebra, and the dual vector space to
a finite-dimensional algebra is a finite-dimensional coalgebra.
 For infinite-dimensional vector spaces the situation is more
complicated: the dual vector space to a coalgebra has a natural
algebra structure, but \emph{not} vice versa.

\subsection{Conilpotent coalgebras}
\label{conilpotent-coalgebras-subsecn}
 A nonunital associative algebra $A$ is said to be \emph{nilpotent}
if there exists an integer $n\ge1$ such that the product of any
$n$~elements in $A$ vanishes, that is $A^n=0$.
 An augmented algebra $A$ is said to be \emph{nilpotent} if its
augmentation ideal $A^+$ is a nilpotent algebra without unit, that is,
the product of any $n$~elements in $A^+$ vanishes.

 Notice that the augmentation of a nilpotent augmented algebra $A$
is \emph{unique}.
 Since the product of any $n$~elements in $A^+$ vanishes, there are
no nonzero $k$\+algebra homomorphisms $A^+\rarrow k$; hence only one
unital $k$\+algebra homomorphism $A\rarrow k$.

 Coalgebras are better suited for nilpotency than algebras, in the sense
that one can speak of a coalgebra being conilpotent without assuming
existence of a related finite integer~$n$.
 One can say that a (unital or nonunital) finite-dimensional coalgebra
$E$ is \emph{conilpotent} if its dual finite-dimensional algebra is
nilpotent.
 Then an arbitrary coalgebra $C$ is said to be \emph{conilpotent} if
all its finite-dimensional subcoalgebras are conilpotent, i.~e., $C$
is a union of finite-dimensional conilpotent coalgebras.
 So a finite-dimensional conilpotent coalgebra always has a finite
conilpotency degree, but an infinite-dimensional conilpotent coalgebra
need \emph{not} have it.

 Here is an explicit definition of a conilpotent coalgebra without
counit that is easy to work with on the technical level.
 A noncounital coalgebra $D$ is said to be \emph{conilpotent} if
for every element $x\in D$ there exists an integer $n\ge1$ such that
the image of~$x$ under the iterated comultiplication map
$\mu^{(n)}\:D\rarrow D^{\ot n+1}$ vanishes.
 A coaugmented coalgebra $(C,\gamma)$ is called \emph{conilpotent} if
the noncounital coalgebra $C^+=C/\gamma(k)$ is conilpotent in
the sense of the previous definition.
 A conilpotent coaugmented coalgebra has a unique coaugmentation;
moreover, any counital coalgebra homomorphism between two conilpotent
coalgebras preserves the coaugmentations.

 A conilpotent coalgebra $C$ comes endowed with a natural increasing
filtration~$F$.
 Namely, $F_{-1}C=0$, \ $F_0C=\gamma(k)\simeq k$, and for every $n\ge0$
the subspace $F_nC\subset C$ consists of all elements $c\in C$ whose
image in $(C^+)^{\ot n+1}$ under the composition of maps
$C\rarrow C^{\ot n+1}\rarrow (C^+)^{\ot n+1}$ vanishes.
 Then one has $\mu(F_nC)\subset\sum_{p+q=n}F_pC\ot_k F_qC\subset
C\ot_kC$, so the filtration $F$ on $C$ is compatible with
the comultiplication.
 The equation $C=\bigcup_{n\ge0}F_nC$ expresses the condition that
$C$ is conilpotent.

 A noncounital conilpotent coalgebra $D$ comes endowed with a similarly
defined natural increasing filtration~$F$.
 Namely, one has $F_0D=0$, and for every $n\ge0$ the subspace
$F_nD\subset D$ is the kernel of the iterated comultiplication map
$\mu^{(n)}\:D\rarrow D^{\ot n+1}$.
 Then one has $\mu(F_nD)\subset\sum_{p+q=n}F_pD\ot_k F_qD\subset
D\ot_kD$ and $D=\bigcup_{n\ge1}F_nD$.

\begin{rem} \label{tensor-coalgebra-cofree-conilpotent-remark}
 Let $V$ be a $k$\+vector space.
 In the discussion of the tensor algebra $T(V)$ and tensor coalgebra
$\udT(V)$ in Section~\ref{coalgebra-structure-cobar-subsecn},
it was mentioned that $T(V)$ is the free associative algebra spanned
by $V$, but a discussion of a possible cofreeness property of
the tensor coalgebra $\udT(V)$ was postponed.
 Now it is the time for such a discussion.

 The first important observation is that the coaugmented coalgebra
$\udT(V)$ is always conilpotent.
 In fact, one has $F_n\udT(V)=\bigoplus_{i=0}^n V^{\ot i}$ for
every $n\ge0$.
 Moreover, $\udT(V)$ is the cofree conilpotent coalgebra cospanned
by $V$, in the following sense.
 Let $C$ be a conilpotent coalgebra, and let $f\:C\rarrow V$ be
a $k$\+linear map annihilating the coaugmentation, i.~e., such that
the composition $f\gamma\:k\rarrow V$ vanishes.
 Then there exists a unique coalgebra homomorphism $t\:C\rarrow\udT(V)$
whose composition with the direct summand projection $\udT(V)\rarrow
V^{\ot 1}=V$ is equal to~$f$.

 Let us emphasize that the assertion above is \emph{not} true for
nonconilpotent coalgebras~$C$.
 For such coalgebras, a homomorphism like~$t$ does not exist in general.
\end{rem}

\subsection{Cosemisimple coalgebras and the coradical filtration}
\label{cosemisimple-coalgebras-and-coradical-subsecn}
 The classical structure theory of finite-dimensional associative
algebras over a field tells that any such algebra $A$ has
a nilpotent Jacobson radical $J$, while the quotient algebra $A/J$
decomposes uniquely as a finite product of simple algebras.
 The Jacobson radical $J\subset A$ is simultaneously the unique maximal
nilpotent ideal and the unique minimal ideal in $A$ for which
the quotient algebra $A/J$ is semisimple.

 Moreover, if $f\:A\rarrow B$ is a surjective homomorphism of
finite-dimensional $k$\+algebras, then the Jacobson radicals $J(A)
\subset A$ and $J(B)\subset B$ are well-behaved with respect to~$f$
in the sense that $f(J(A))=J(B)$.
 The induced surjective homomorphism of semisimple algebras
$A/J(A)\rarrow B/J(B)$ is a projection onto a direct factor; in fact,
it represents the algebra $B/J(B)$ as a direct factor of
the algebra $A/J(A)$ in a unique way.

 Passing to the dual coalgebra, one obtains what can be called
``the structure theory of finite-dimensional coassociative
coalgebras''~$E$.
 Any such coalgebra has a unique maximal cosemisimple subcoalgebra
$E^\ss$, which decomposes uniquely as a direct sum of cosimple 
coalgebras.
 The (noncounital) quotient coalgebra $E/E^\ss$ is conilpotent.

 Moreover, if $E\subset D$ is a subcoalgebra in a finite-dimensional
coalgebra, then $E^\ss=E\cap D^\ss$.
 The cosemisimple coalgebra $D^\ss$ can be decomposed as a direct
sum of its subcoalgebra $E^\ss$ and a complementary subcoalgebra in
a unique way.

 Let us formulate the general definition of a \emph{cosemisimple 
coalgebra}, in the form of a list of equivalent conditions in
the next lemma.
 Here a nonzero coassociative, counital coalgebra over~$k$ is called
\emph{cosimple} if it has no nonzero proper subcoalgebras.

\begin{lem}
 For any coassociative, counital coalgebra $C$ over a field~$k$,
the following conditions are equivalent:
\begin{enumerate}
\item the abelian category of left $C$\+comodules is semisimple;
\item the abelian category of right $C$\+comodules is semisimple;
\item $C$ is the sum of its cosimple subcoalgebras;
\item $C$ uniquely decomposes as a direct sum of its
(finite-dimensional) cosimple subcoalgebras, with all the cosimple 
subcoalgebras of $C$ appearing as direct summands in this direct sum
decomposition. \qed
\end{enumerate}
\end{lem}

 Now let $C$ be an arbitrary (infinite-dimensional) coassociative
coalgebra.
 Representing $C$ as the directed union of its finite-dimensional
subcoalgebras $E$ and passing to the direct limit, one obtains
the following structural result.

\begin{prop}
 Any coassociative, counital coalgebra $C$ over a field~$k$ contains
a unique maximal cosemisimple subcoalgebra $C^\ss$, which is can be
obtained as the sum of all co(semi)simple subcoalgebras of~$C$.
 Simultaneously, $C^\ss$ is the unique minimal subcoalgebra among
all the subcoalgebras $D\subset C$ for which that the (noncounital)
quotient coalgebra $C/D$ is conilpotent. \qed
\end{prop}

 Following the discussion in
Section~\ref{conilpotent-coalgebras-subsecn}, the conilpotent
noncounital coalgebra $D=C/C^\ss$ is endowed with a natural increasing
filtration~$F$.
 Denoting by $F_nC$ the full preimage of $F_n(C/C^\ss)$ under
the natural surjective homomorphism $C\rarrow C/C^\ss$, one obtains
a natural increasing filtration $F$ on~$C$, called the \emph{coradical
filtration}.
 What we call the ``maximal cosemisimple subcoalgebra'' $C^\ss$ of
a coalgebra $C$ is otherwise known as the \emph{coradical} of~$C$
\,\cite[Chapters~VIII\+-IX]{Swe}, \cite[Chapter~5]{Mon}.

\begin{rem}
 The definitions of the coradical filtration in~\cite{Swe,Mon} are
formulated in terms of the so-called \emph{wedge} operation on
vector subspaces of a coalgebra.
 Let $C$ be a coassociative coalgebra over~$k$ and $X$, $Y\subset C$
be two vector subspaces.
 According to~\cite[Section~9.0]{Swe}, \cite[proof of
Theorem~5.2.2]{Mon}, the wedge $X\wedge Y\subset C$ is defined as
the kernel of the composition $C\rarrow C\ot_k C\rarrow C/X\ot_k C/Y$
of the comultiplication map $C\rarrow C\ot_k C$ with the natural
surjection $C\ot_k C\rarrow C/X\ot_k C/Y$.
 According to~\cite[Proposition~9.0.0(c)]{Swe}, the wedge operation
is associative.
 By the definition, put $\bigwedge^1X=X$ and
$\bigwedge^nX=(\bigwedge^{n-1}X)\wedge X$ for $n\ge2$.

 Notice that the definition of the wedge makes perfect sense for
noncounital coalgebras~$D$ as well.
 Arguing by induction on~$n$, one easily shows that $\bigwedge^nX
\subset D$ is the kernel of the composition $D\rarrow D^{\ot n}
\rarrow(D/X)^{\ot n}$ of the iterated comultiplication map
$D\rarrow D^{\ot n}$ with the natural surjection $D^{\ot n}\rarrow
(D/X)^{\ot n}$ (for every $n\ge1$).
 Then, for a noncounital conilpotent coalgebra $D$, the definition
of the natural increasing filtration $F$ on $D$ given in
Section~\ref{conilpotent-coalgebras-subsecn} can be rephrased
as $F_nD=\bigwedge^{n+1}0\subset D$.

 Let $p\:C\rarrow D$ be a homomorphism of noncounital coalgebras.
 Then one can easily see that the wedge operation commutes with
taking the full preimages of vector subspaces under~$p$: for any
subspaces $X$, $Y\subset D$ one has
$p^{-1}(X\wedge Y)=p^{-1}X\wedge p^{-1}Y$, hence
$p^{-1}(\bigwedge^n X)=\bigwedge^n(p^{-1}X)$.

 Now let $C$ be a coassociative, counital coalgebra over~$k$.
 The coradical filtration is defined in~\cite[Section~9.1]{Swe} by
the rule $C_n=\bigwedge^{n+1}C^\ss\subset C$, where $C^\ss\subset C$
is what we call the maximal cosemisimple subcoalgebra (called
the coradical in~\cite{Swe,Mon}).
 In the book~\cite[\S5.2]{Mon}, the coradical filtration is defined by
induction using the rules $C_0=C^\ss$ and, essentially,
$C_n=C^\ss\wedge C_{n-1}\subset C$; this is clearly equivalent to
the definition in~\cite{Swe}.
 Taking $D=C/C^\ss$, one concludes that $C_n=\bigwedge^{n+1}C^\ss=
\bigwedge^{n+1}p^{-1}(0)=p^{-1}(\bigwedge^{n+1}0)=p^{-1}(F_nD)
=F_nC$ (where $p\:C\rarrow D$ is the natural surjection).
 Thus the definition of the coradical filtration by the rule
$C_n=\bigwedge^{n+1}C^\ss$ in~\cite{Swe,Mon} agrees with our definition
of it by the rule $F_nC=p^{-1}(F_nD)$ given above in this section.
\end{rem}

 A counital coalgebra $C$ is conilpotent (in the sense of
the definition in Section~\ref{conilpotent-coalgebras-subsecn}) if
and only if there is an isomorphism of coalgebras $C^\ss\simeq k$,
or in other words, if and only if the coalgebra $C^\ss$
is one-dimensional.

\Section{DG-Algebras and DG-Coalgebras}

 The aim of this section is to formulate a simple version of
Koszul duality between algebras and coalgebras (rather than modules
and comodules).
 The notion of a \emph{filtered quasi-isomorphism} of
DG-coalgebras~\cite{Hin2} plays a key role.

\subsection{Augmented DG-algebras} \label{augmented-dg-algebras-subsecn}
 Let $A^\bu=(A,d)$ be a DG\+algebra over a field~$k$.
 So $A=\bigoplus_{i\in\boZ} A^i$ is a graded $k$\+algebra and
$d\:A\rarrow A$ is a homogeneous $k$\+linear map of degree~$1$
satisfying the Leibniz rule with signs with respect to
the multiplication (in other words, $d$~is an odd derivation of~$A$)
such that $d^2=0$.

 An \emph{augmentation}~$\alpha$ on $A^\bu$ is a (unital) DG\+algebra
homomorphism $\alpha\:A^\bu\rarrow k$.
 So $\alpha(a)=0$ for all $a\in A^i$, \,$i\ne0$, \ $\alpha(ab)=
\alpha(a)\alpha(b)$ for all $a$, $b\in A$, \ $\alpha(1)=1$, and
$\alpha(d(a))=0$ for all $a\in A^{-1}$.
 We denote by $A^{\bu,+}=\ker(\alpha)\subset A^\bu$ the augmentation
ideal.
 So $A^{\bu,+}$ is a homogeneous two-sided DG\+ideal in $A^\bu$, and
$A^\bu=k\oplus A^{\bu,+}$ as a complex of $k$\+vector spaces.

 The \emph{bar-construction} $\Br^\bu_\alpha(A^\bu)$ of an augmented
DG\+algebra $(A^\bu,\alpha)$ is the bicomplex
$$
 k\overset0\llarrow A^{\bu,+}\llarrow A^{\bu,+}\ot_k A^{\bu,+}
 \llarrow A^{\bu,+}\ot_k A^{\bu,+}\ot_k A^{\bu,+}\llarrow\dotsb
$$
with the bar differential~$\d$ given by the formulas similar to those
in Section~\ref{augmented-algebras-subsecn} (except for the $\pm$~signs,
which are more complicated in the case of a DG\+algebra~$A^\bu$) and
the differential~$d$ induced by the differential~$d$ on~$A^\bu$.
 The differential~$d$ on $\Br^\bu_\alpha(A^\bu)$ acts on every
tensor power $A^{\bu,+}{}^{\ot n}$ of the augmentation ideal
$A^{\bu,+}$ by the tensor product of $n$~copies of the differential~$d$
on~$A^{\bu,+}$ (in the sense of the tensor product of complexes).

 The total complex of the bar-construction $\Br^\bu_\alpha(A^\bu)$
is produced by taking infinite direct sums along the diagonals.
 The underlying graded vector space of $\Br^\bu_\alpha(A^\bu)$ is
endowed with the graded coalgebra structure of the tensor coalgebra
$\Br_\alpha(A)=\udT(A^+[1])$.
 Here the shift of cohomological grading~$[1]$ in the latter formula is
a way to express the passage to the total complex of the bicomplex
$\Br^\bu_\alpha(A^\bu)$, with its total grading equal to the grading
induced by the grading of $A$ minus the grading by the number of
the tensor factors.
 This makes $\Br^\bu_\alpha(A^\bu)$ a DG\+coalgebra (as defined in
Section~\ref{dg-coalgebras-subsecn}).
 We refer to~\cite[Section~6.1]{Pkoszul} for further details, incuding
the (not so trivial) sign rules for the differentials~$\d$ and~$d$.

\subsection{Coaugmented DG-coalgebras}
\label{coaugmented-dg-coalgebras-subsecn}
 Let $C^\bu=(C,d)$ be a DG\+coalgebra over~$k$
(see Section~\ref{dg-coalgebras-subsecn}).
 A \emph{coaugmentation}~$\gamma$ on $C^\bu$ is a (counital)
DG\+coalgebra homomorphism $\gamma\:k\rarrow C$.
 So $\gamma(k)\subset C^0$, \ $\gamma$~is a homomorphism of coalgebras,
the composition $k\overset\gamma\rarrow C\overset\epsilon\rarrow k$ is
the identity map, and $d(\gamma(1))=0$.
 We put $C^{\bu,+}=\coker(\gamma)$; so $C^+$ is, generally speaking,
a DG\+coalgebra without counit, and $C^\bu=k\oplus C^{\bu,+}$ as
a complex of $k$\+vector spaces.
 
 The \emph{cobar-construction} $\Cb^\bu_\gamma(C^\bu)$ of
a coaugmented DG\+coalgebra $(C^\bu,\gamma)$ is the bicomplex
$$
 k\overset0\lrarrow C^{\bu,+}\lrarrow C^{\bu,+}\ot_k C^{\bu,+}
 \lrarrow C^{\bu,+}\ot_k C^{\bu,+}\ot_k C^{\bu,+}\lrarrow\dotsb
$$
with the cobar differential~$\d$ given by the formulas similar to
those in Section~\ref{cobar-construction-subsecn} (except for
the $\pm$~signs) and the differential~$d$ induced by
the differential~$d$ on~$C^\bu$.
 The differential~$d$ on $\Cb^\bu_\gamma(C^\bu)$ acts on every
tensor power $C^{\bu,+}{}^{\ot n}$ of the complex $C^{\bu,+}$ by
the tensor product of $n$~copies of the differential~$d$ on~$C^{\bu,+}$.

 The total complex of the cobar-construction $\Cb^\bu_\gamma(C^\bu)$
is produced by taking infinite direct sums along the diagonals.
 The underlying graded vector space of $\Cb^\bu_\gamma(C^\bu)$ is
endowed with the graded algebra structure of the tensor algebra
$\Cb_\gamma(C)=T(C^+[-1])$.
 Here the shift of cohomological grading~$[-1]$ is a way to express
the passage to the total complex of the bicomplex
$\Cb^\bu_\gamma(C^\bu)$, with its total grading equal to the grading
induced by the grading of $C$ plus the grading by the number of
tensor factors.
 This makes $\Cb^\bu_\gamma(C^\bu)$ a DG\+algebra.
 We refer to~\cite[Section~6.1]{Pkoszul} for further details,
including the sign rules for the differentials~$\d$ and~$d$.

\subsection{Conilpotent DG-coalgebras}
\label{conilpotent-dg-coalgebras-subsecn}
 A noncounital graded coalgebra $D$ is called \emph{conilpotent} if its
underlying noncounital ungraded coalgebra is conilpotent (see
the definition in Section~\ref{conilpotent-coalgebras-subsecn}).
 A coaugmented graded coalgebra $(C,\gamma)$ is called
\emph{conilpotent} if its underlying coaugmented ungraded coalgebra
is conilpotent; equivalently, this means that the noncounital
graded coalgebra $D=C/\gamma(k)$ is conilpotent.

 A coaugmented DG\+coalgebra $(C^\bu,\gamma)$ is called
\emph{conilpotent} if its underlying graded coalgebra $C$ is conilpotent.
 In other words, a DG\+coalgebra $C^\bu$ is conilpotent if and only if
the graded coalgebra $C$ is conilpotent \emph{and} its coaugmentation
$\gamma\:k\rarrow C$ is a DG\+coalgebra map, that is, $d(\gamma(1))=0$.

 By construction, the canonical increasing filtration $F$ on a conilpotent
(noncounital or coaugmented) graded coalgebra is a filtration by
homogeneous vector subspaces.
 The canonical increasing filtration $F$ on a conilpotent DG\+coalgebra
$C^\bu$ is a filtration by subcomplexes $F_nC^\bu\subset C^\bu$.

\subsection{Duality between DG-algebras and DG-coalgebras}
\label{duality-dg-algebras-dg-coalgebras-subsecn}
 Denote by $k\alg_\dg$ the category of DG\+algebras over~$k$ (with
the usual DG\+algebra homomorphisms).
 Furthermore, let $k\alg_\dg^\aug$ denote the category of augmented
DG\+algebras (with DG\+algebra homomorphisms preserving
the augmentations).

 Similarly, denote by $k\coalg_\dg$ the category of DG\+coalgebras
over~$k$, and by $k\coalg_\dg^\coaug$ the category of coaugmented
DG\+coalgebras.
 Let $k\coalg_\dg^\conilp\subset k\coalg_\dg^\coaug$ denote the full
subcategory whose objects are the conilpotent DG\+coal\-ge\-bras.

 One easily observes that the cobar-construction $\Cb^\bu_\gamma(C^\bu)$
of any coaugmented DG\+coalgebra $C^\bu$ is naturally an augmented
DG\+algebra (with the direct summand projection $\Cb_\gamma(C)\rarrow
(C^+)^{\ot0}=k$ providing the augmentation).
 The bar-construction $\Br^\bu_\alpha(A^\bu)$ of any augmented
DG\+algebra $A^\bu$ is not only naturally coaugmented, but even
a conilpotent DG\+coalgebra (essentially, because
the tensor coalgebra is conilpotent; see
Remark~\ref{tensor-coalgebra-cofree-conilpotent-remark}).

 So the bar-construction is a functor
$$
 \Br_\alpha^\bu\:k\alg_\dg^\aug\lrarrow k\coalg_\dg^\conilp,
$$
while the cobar-construction is a functor
$$
 \Cb_\gamma^\bu\:k\coalg_\dg^\coaug\lrarrow k\alg_\dg^\aug.
$$

\begin{lem} \label{augmented-bar-cobar-adjunction}
 The restriction of the cobar-construction to conilpotent
DG\+coalgebras,
$$
 \Cb_\gamma^\bu\:k\coalg_\dg^\conilp\lrarrow k\alg_\dg^\aug
$$
is naturally a left adjoint functor to the bar-construction\/
$\Br_\alpha^\bu\:k\alg_\dg^\aug\rarrow k\coalg_\dg^\conilp$.
\end{lem}

\begin{proof}
 Essentially, the assertion holds because $T(V)$ is the free (graded)
algebra spanned by a (graded) vector space $V$, while $\udT(V)$ is
the cofree conilpotent (graded) coalgebra cospanned by~$V$.
 A more detailed explanation, based on the notion of a twisting cochain,
will be offered in
Section~\ref{bar-cobar-and-acyclic-twisting-cochains-subsecn} below.
\end{proof}

 One would like to define equivalence relations (i.~e., the classes of
morphisms to be inverted) in the categories of augmented DG\+algebras
and conilpotent DG\+coalgebras so that, after inverting these classes
of morphisms, the adjoint functors $\Br_\alpha^\bu$ and $\Cb_\gamma^\bu$
become equivalences of categories.
 One important obstacle is that it is \emph{not} good enough to
just invert the usual quasi-isomorphisms in both the categories
$k\alg_\dg^\aug$ and $k\coalg_\dg^\conilp$, for the reason demonstrated
by Example~\ref{bar-of-algebras-quasi-isomorphism}.

 The problem is resolved by using the concept of a \emph{filtered
quasi-isomorphism} introduced by Hinich in~\cite[Section~4]{Hin2}.
 In the terminology of~\cite{Hin2}, an \emph{admissible filtration}
on a coaugmented DG\+coalgebra $(C^\bu,\gamma)$ is an exhaustive
comultiplicative increasing filtration by subcomplexes $F_nC^\bu$
such that $F_{-1}C^\bu=0$ and $F_0C^\bu=\gamma(k)$.
 This means that $C^\bu=\bigcup_{n\ge0}F_nC^\bu$ and
$\mu(F_nC^\bu)\subset\sum_{p+q=n}^{p,q\ge0}F_pC^\bu\ot F_qC^\bu$
for every $n\ge0$.

 In particular, any coaugmented DG\+coalgebra admitting an admissible
filtration is conilpotent.
 Conversely, the canonical increasing filtration on any conilpotent
DG\+coalgebra is admissible.

 Let $f\:C^\bu\rarrow D^\bu$ be a morphism of conilpotent
DG\+coalgebras.
 The morphism~$f$ is said to be a \emph{filtered quasi-isomorphism}
if there exist admissible filtrations $F$ on both the DG\+coalgebras
$C^\bu$ and $D^\bu$ such that $f(F_nC)\subset F_nD$ and the induced map
$F_nC^\bu/F_{n-1}C^\bu\rarrow F_nD^\bu/F_{n-1}D^\bu$ is
a quasi-isomorphism of complexes of $k$\+vector spaces
for every $n\ge0$.

 Notice that there is \emph{no} reason to expect that the composition
of filtered quasi-isomorphisms should be a filtered quasi-isomorphism.
 Nevertheless, leaving set-theoretical issues aside, to any class of
morphisms $\sS$ in a category $\sC$ one can assign the category
$\sC[\sS^{-1}]$ obtained by formally inverting all morphisms from~$\sS$;
in particular, one can formally invert all filtered quasi-isomorphisms
of DG-coalgebras.

 By contrast, a morphism of DG\+algebras $f\:A\rarrow B$ is said to be
a quasi-iso\-mor\-phism if it is a quasi-isomorphism of the underlying
complexes of vector spaces.
 Clearly, for any pair of composable morphisms of DG\+algebras $f$
and~$g$, if two of the morphisms $f$, $g$, and~$gf$ are
quasi-isomorphisms, then so is the third one.

 The following theorem goes back to~\cite[Theorem~3.2]{Hin2}
and~\cite[Th\'eor\`eme~1.3.1.2]{Lef}.
 In the stated form, it can be found in~\cite[Theorem~6.10(b)]{Pkoszul}.

\begin{thm} \label{augmented-algebras-coalgebras-duality}
 Let\/ $\Quis$ be the class of all quasi-isomorphisms of augmented
DG\+algebras and\/ $\FQuis$ be the class of all filtered
quasi-isomorphisms of conilpotent DG\+coalgebras over~$k$.
 Then the adjoint functors\/ $\Br^\bu_\alpha$ and\/ $\Cb^\bu_\gamma$
induce mutually inverse equivalences of categories
$$
 \Br^\bu_\alpha\:k\alg_\dg^\aug[\Quis^{-1}]\,\simeq\,
 k\coalg_\dg^\conilp[\FQuis^{-1}]\,:\!\Cb^\bu_\gamma.
$$
\qed
\end{thm}

\begin{rems} \label{cobar-not-preserves-quasi-isomorphisms-remarks}
 Clearly, any filtered quasi-isomorphism of DG\+coalgebras is
a quasi-isomorphism of their underlying complexes of vector spaces
(i.~e., a quasi-isomorphism in the usual sense of the word).
 By the converse is \emph{not} true.

 The functor $\Br^\bu_\alpha\:k\alg_\dg^\aug\rarrow k\coalg_\dg^\conilp$
takes quasi-isomorphisms to filtered quasi-isomorphisms.
 In particular, it follows that it takes quasi-isomorphisms to
quasi-isomorphisms.

 The functor $\Cb^\bu_\gamma\:k\coalg_\dg^\conilp\rarrow k\alg_\dg^\aug$
takes filtered quasi-isomorphisms to quasi-isomorphisms.
 This important property of filtered quasi-isomorphisms of DG\+coalgebras
is a part of Theorem~\ref{augmented-algebras-coalgebras-duality}.

 However, the functor $\Cb^\bu_\gamma$ does \emph{not} take
quasi-isomorphisms to quasi-iso\-mor\-phisms, generally speaking.
 For example, let $f\:(A,\alpha)\rarrow(B,\beta)$ be a morphism of
augmented algebras (or augmented DG\+algebras) as in
Example~\ref{bar-of-algebras-quasi-isomorphism}, that is, $f$~is
\emph{not} a quasi-isomorphism, but $\Br^\bu(f)$ is.
 Put $C^\bu=\Br^\bu_\alpha(A^\bu)$, \ $D^\bu=\Br^\bu_\beta(B^\bu)$,
and $g=\Br^\bu(f)$; so $g\:(C^\bu,\gamma)\rarrow(D^\bu,\delta)$ is
a quasi-isomorphism (but not a filtered quasi-isomorphism!) of
conilpotent DG\+coalgebras.

 It is a part of Theorem~\ref{augmented-algebras-coalgebras-duality}
that the adjunction morphism $\Cb^\bu_\eta(\Br^\bu_\varepsilon(E^\bu))
\rarrow E^\bu$ is a quasi-isomorphism of DG\+algebras for any
augmented DG\+algebra $(E^\bu,\varepsilon)$.
 So the adjunction morphisms $\Cb^\bu_\gamma(\Br^\bu_\alpha(A^\bu))
\rarrow A^\bu$ and $\Cb^\bu_\delta(\Br^\bu_\beta(B^\bu))\rarrow B^\bu$
are quasi-isomorphisms.
 It follows that $\Cb^\bu(g)\:\Cb^\bu_\gamma(C^\bu)\rarrow
\Cb^\bu_\delta(D^\bu)$ is \emph{not} a quasi-isomorphism
of DG\+algebras.
\end{rems}

\subsection{Quillen equivalence between DG-algebras and DG-coalgebras}
\label{quillen-equivalence-dg-algebras-dg-coalgebras-subsecn}
 The derived Koszul duality and comodule-contramodule correspondence
results can be usually expressed in the language of Quillen 
equivalences between model categories~\cite[Section~8.4]{Pkoszul}.
 In particular, the Koszul duality between augmented DG\+algebras and
conilpotent DG\+coalgebras stated in
Theorem~\ref{augmented-algebras-coalgebras-duality} can be expressed
as a Quillen equivalence.
 Let us briefly sketch the related assertions.

 We suggest the book~\cite{Hov} as a standard reference source on
model categories and related concepts.
 The category $k\alg_\dg$ of DG\+algebras over~$k$ has a standard
model category structure in which the quasi-isomorphisms are the weak
equivalences and the surjective DG\+algebra maps (i.~e.,
the DG\+algebra maps that are surjective as maps of graded vector
spaces) are the fibrations~\cite{Hin1,Jar}.
 It is clear from the explicit descriptions of the classes of 
cofibrations and trivial cofibrations in $k\alg_\dg$, which can be
found in the formulation of~\cite[Theorem~9.1(a)]{Pkoszul}, that
the model category $k\alg_\dg$ is cofibrantly generated.
 The similar assertions hold for the category of augmented DG\+algebras
$k\alg_\dg^\aug$ \,\cite[Theorem~9.1(b)]{Pkoszul}.

 The category $k\coalg_\dg^\conilp$ of conilpotent DG\+coalgebras
over~$k$ also has a natural model
structure~\cite[Theorem~3.1]{Hin2}, \cite[Th\'eor\`eme~1.3.1.2(a)]{Lef},
\cite[Theorem~9.3(b)]{Pkoszul}.
 The weak equivalences can be described as the morphisms which get
inverted when one inverts the filtered quasi-isomorphisms.
 The cofibrations are the injective DG\+coalgebra maps (i.~e.,
the DG\+coalgebra maps that are injective as maps of graded vector
spaces).
 Explicit descriptions of all the classes of morphisms involved
can be found in the formulation of~\cite[Theorem~9.3(b)]{Pkoszul}.

 The model category $k\coalg_\dg^\conilp$ is likewise cofibrantly
generated.
 In fact, the injective morphisms of finite-dimensional DG\+coalgebras
are the generating cofibrations (as one can see from the fact that
any DG\+coalgebra is the directed union of its finite-dimensional
DG\+subcoalgebras;
cf.\ Section~\ref{local-finite-dimensionality-subsecn}).
 The injective filtered quasi-isomorphisms strictly compatible with
the filtrations, acting between finite-dimensional DG\+coalgebras,
are the generating trivial cofibrations (this follows from
the description of the trivial cofibrations as the retracts of
the injective filtered quasi-isomorphisms strictly compatible with
the filtrations, as per~\cite[Theorem~9.3(b)]{Pkoszul}).

 The pair of adjoint functors in
Lemma~\ref{augmented-bar-cobar-adjunction} and
Theorem~\ref{augmented-algebras-coalgebras-duality} is a Quillen
equivalence between $k\alg_\dg^\aug$ and $k\coalg_\dg^\conilp$
\,\cite[Theorem~3.2]{Hin2}, \cite[Th\'eor\`eme~1.3.1.2(b)]{Lef},
\cite[end of Section~9.3]{Pkoszul}.

\Section{Twisting Cochains}  \label{twisting-cochains-secn}

 The concept of a \emph{twisting cochain} explains the adjunction of
Lemma~\ref{augmented-bar-cobar-adjunction} and the constructions
of inverse equivalences in
Theorems~\ref{complexes-of-modules-augmented-bar-construction-thm},
\ref{complexes-of-modules-quadratic-linear-thm},
\ref{complexes-of-comodules-coaugmented-cobar-thm},
and~\ref{complexes-of-comodules-conilpotent-cobar-thm}.
 More generally, the notion of an \emph{acyclic twisting cochain},
defined in this section, allows to conveniently formulate
the Koszul duality theorems in wide contexts.

\subsection{The Hom DG-algebra and twisting cochains}
\label{hom-dg-algebra-and-twisting-cochains-subsecn}
 Let $C$ be a coassociative coalgebra and $A$ be an associative
algebra over a field~$k$.
 Then the vector space $\Hom_k(C,A)$ of all $k$\+linear maps
$C\rarrow A$ acquires an associative $k$\+algebra structure.
 Given two linear maps $f$, $g\in\Hom_k(C,A)$, their product
$fg\in\Hom_k(C,A)$ is constructed as the composition
$$
 C\overset\mu\lrarrow C\ot_kC\overset{f\ot g}\lrarrow A\ot_kA
 \overset m\lrarrow A,
$$
where $\mu\:C\rarrow C\ot_kC$ is the comultiplication and
$m\:A\ot_kA\rarrow A$ is the multiplication map.
 Given a counit $\epsilon\:C\rarrow k$ on $C$ and a unit $e\:k\rarrow A$
in $A$, the unit element $1\in\Hom_k(C,A)$ is constructed as
the composition $C\overset\epsilon\rarrow k\overset e\rarrow A$.

 Given two graded vector spaces $U$ and $V$ over~$k$, one denotes
by $\Hom_k(U,V)$ the graded Hom space.
 So $\Hom_k(U,V)=\bigoplus_{i\in\boZ}\Hom_k^i(U,V)$, where
$\Hom_k^i(U,V)$ is the vector space of all homogeneous $k$\+linear
maps $U\rarrow V$ of degree~$i$.
 Then, for any graded coalgebra $C$ and any graded algebra $A$,
the graded Hom space $\Hom_k(C,A)$ is a graded algebra over~$k$.

 Given two complexes of vector spaces $U^\bu$ and $V^\bu$,
the graded Hom space $\Hom_k(U,V)$ is endowed with a differential
in the usual way, producing the complex $\Hom_k^\bu(U^\bu,V^\bu)$.
 For any DG\+coalgebra $C^\bu$ and a DG\+algebra $A^\bu$ over~$k$,
the complex $\Hom_k^\bu(C^\bu,A^\bu)$ is a DG\+algebra.

 Let $E^\bu=(E,d)$ be a DG\+algebra.
 A homogeneous element $a\in E^1$ of degree~$1$ in $E$ is said to be
a \emph{Maurer--Cartan element} (or a \emph{Maurer--Cartan cochain})
if it satisfies the \emph{Maurer--Cartan equation} $a^2+d(a)=0$
in~$E^2$.

 Let $C^\bu$ be a DG\+coalgebra and $A^\bu$ be a DG\+algebra over~$k$.
 By the definition, a \emph{twisting cochain}~$\tau$ for $C^\bu$
and $A^\bu$ is a Maurer--Cartan cochain in the DG\+algebra
$\Hom_k^\bu(C^\bu,A^\bu)$.
 So $\tau\:C^\bu\rarrow A^\bu$ is a homogeneous $k$\+linear map of
degree~$1$ (i.~e., $\tau_i\:C^i\rarrow A^{i+1}$, \ $i\in\boZ$)
satisfying the Maurer--Cartan equation.

 Various conditions of compatibility with (co)augmentations are usually
imposed on twisting cochains.
 In particular, let $(C^\bu,\gamma)$ be a coaugmented DG\+coalgebra and
$(A^\bu,\alpha)$ be an augmented DG\+algebra.
 Twisting cochains $\tau\:C^\bu\rarrow A^\bu$ such that
$\alpha\circ\tau=0=\tau\circ\gamma$ will be important for us
in this Section~\ref{twisting-cochains-secn}.

\subsection{Bar-cobar adjunction and acyclic twisting cochains}
\label{bar-cobar-and-acyclic-twisting-cochains-subsecn}
 We start with some examples before passing to the general case.

\begin{exs} \label{bar-construction-twisting-cochains-examples}
 (1)~Let $(A,\alpha)$ be an augmented algebra over~$k$,
and let $C^\bu=\Br^\bu_\alpha(A)$ be its bar-construction,
as in Section~\ref{augmented-algebras-subsecn}.
 Let $\gamma\:k\rarrow C^\bu$ be the natural coaugmentation of
$\Br^\bu_\alpha(A)$ (as in
Section~\ref{duality-dg-algebras-dg-coalgebras-subsecn}).
 Then the composition~$\tau$ of the direct summand projection
$\Br^\bu_\alpha(A)\rarrow A^+{}^{\ot 1}=A^+$ and the inclusion
$A^+\rarrow A$ is a twisting cochain for the DG\+coalgebra $C^\bu$
and the algebra~$A$ (viewed as a DG\+algebra in the obvious way).
 The equations of compatibility with the augmentation and
the coaugmentation $\alpha\circ\tau=0=\tau\circ\gamma$ are satisfied
in this example.

\smallskip
 (2)~More generally, let $(A^\bu,\alpha)$ be an augmented DG\+algebra,
and let $C^\bu=\Br^\bu_\alpha(A^\bu)$ be its bar-construction,
as in Section~\ref{augmented-dg-algebras-subsecn}.
 Denote by $\gamma\:k\rarrow C^\bu$ the natural coaugmentation of
$\Br^\bu_\alpha(A^\bu)$.
 Then the composition $\tau\:C^\bu\rarrow A^\bu$ of the direct summand
projection $\Br^\bu_\alpha(A^\bu)\rarrow A^{\bu,+}{}^{\ot 1}=A^{\bu,+}$
and the inclusion $A^{\bu,+}\rarrow A^\bu$ is a twisting cochain for
the DG\+coalgebra $C^\bu$ and the DG\+algebra~$A^\bu$.
 The equations $\alpha\circ\tau=0=\tau\circ\gamma$ are satisfied
for this twisting cochain.
\end{exs}

\begin{exs} \label{cobar-construction-twisting-cochains-examples}
 (1)~Let $(C,\gamma)$ be a coaugmented coalgebra over~$k$,
and let $A^\bu=\Cb^\bu_\gamma(C)$ be its cobar-construction,
as in Section~\ref{cobar-construction-subsecn}.
 Let $\alpha\:A^\bu\rarrow k$ be the natural augmentation of
$\Cb^\bu_\gamma(C)$ (as in
Section~\ref{duality-dg-algebras-dg-coalgebras-subsecn}).
 Then the composition~$\tau$ of the natural surjection $C\rarrow C^+$
and the direct summand inclusion $C^+\rarrow\Cb^\bu_\gamma(C)$
is a twisting cochain for the coalgebra~$C$ (viewed as a DG\+coalgebra
in the obvious way) and the DG\+algebra~$A^\bu$.
 The equations of compatibility with the augmentation and
the coaugmentation $\alpha\circ\tau=0=\tau\circ\gamma$ are satisfied
in this example.

\smallskip
 (2)~More generally, let $(C^\bu,\gamma)$ be a coaugmented
DG\+coalgebra, and let $A^\bu=\Cb^\bu_\gamma(C^\bu)$ be
its cobar-construction, as in
Section~\ref{coaugmented-dg-coalgebras-subsecn}.
 Denote by $\alpha\:A^\bu\rarrow k$ the natural augmentation of
$\Cb^\bu_\gamma(C^\bu)$.
 Then the composition $\tau\:C^\bu\rarrow A^\bu$ of the natural
surjection $C^\bu\rarrow C^{\bu,+}$ and the direct summand inclusion
$C^{\bu,+}\rarrow\Cb^\bu_\gamma(C^\bu)$ is a twisting cochain for
the DG\+coalgebra $C^\bu$ and the DG\+algebra~$A^\bu$.
 The equations $\alpha\circ\tau=0=\tau\circ\gamma$ are satisfied
for this twisting cochain.
\end{exs}

\begin{ex} \label{nonhomogeneous-quadratic-twisting-cochain}
 Let $(A,F,\alpha)$ be an augmented nonhomogeneous quadratic algebra,
as defined in
Section~\ref{nonhomogeneos-quadratic-dual-to-augmented-subsecn},
and let $C^\bu\subset\Br^\bu_\alpha(A)$ be the related (nonhomogeneous
quadratic dual) DG\+subcoalgebra.
 Then the composition $C^\bu\rarrow\Br^\bu_\alpha(A)\rarrow A$ of
the inclusion map $C^\bu\rarrow\Br^\bu_\alpha(A)$ with
twisting cochain $\Br^\bu_\alpha(A)\rarrow A$ from
Example~\ref{bar-construction-twisting-cochains-examples}(1) is
a twisting cochain $\tau\:C^\bu\rarrow A$.
 Equivalently, the twisting cochain~$\tau$ can be constructed as
the composition of the direct summand projection $C\rarrow C^1=V=
F_1A/F_0A$ and the inclusion $F_1A/F_0A\simeq A^+\cap F_1A\rarrow A$.

 The natural coaugmentation $k\rarrow\Br^\bu_\alpha(A)$ of
the cobar DG\+algebra $\Br^\bu_\alpha(A)$ factorizes as
$k\rarrow C^\bu\rarrow\Br^\bu_\alpha(A)$, making $C^\bu$
a coaugmented (in fact, conilpotent) DG\+coalgebra with
the coaugmentation $\gamma\:k\rarrow C^\bu$.
 The twisting cochain $\tau\:C^\bu\rarrow A$ satisfies
the equations $\alpha\circ\tau=0=\tau\circ\gamma$.
\end{ex}

 Now we can return to the proof of
Lemma~\ref{augmented-bar-cobar-adjunction}.

\begin{proof}[Proof of Lemma~\ref{augmented-bar-cobar-adjunction}:
further details]
 Let $(A^\bu,\alpha)$ be an augmented DG\+algebra and $(C^\bu,\gamma)$
be a conilpotent DG\+coalgebra.
 Then the underlying graded coalgebra of the bar-construction
$\Br^\bu_\alpha(A^\bu)$ is $\Br_\alpha(A)=\udT(A^+[1])$ and
the underlying graded algebra of the cobar-construction
$\Cb^\bu_\gamma(C^\bu)$ is $\Cb_\gamma(C)=T(C^+[-1])$.

 Since $T(C^+[-1])$ is the free graded algebra spanned by $C^+[-1]$,
graded algebra homomorphisms $\Cb_\gamma(C)\rarrow A$ correspond
bijectively to homogeneous $k$\+linear maps $C^+[-1]\rarrow A$ of
degree~$0$.
 Among these, the graded algebra homomorphisms compatible with
the augmentations on $\Cb_\gamma(C)$ and $A$ correspond precisely
to the linear maps $C^+[-1]\rarrow A^+$.
 This means elements of the vector space $\Hom_k^1(C^+,A^+)$.

 Since $\udT(A^+[1])$ is the cofree conilpotent graded coalgebra
cospanned by $A^+[1]$ and $C$ is a conilpotent graded coalgebra,
graded coalgebra homomorphisms $C\rarrow\Br_\alpha(A)$ correspond
bijectively to homogeneous $k$\+linear maps $C^+\rarrow A^+[1]$
of degree~$0$.
 This means elements of the same vector space $\Hom_k^1(C^+,A^+)$.

 Finally, one has to check that a graded algebra homomorphism
$\Cb_\gamma(C)\rarrow A$ commutes with the differentials on
$\Cb^\bu_\gamma(C^\bu)$ and $A^\bu$ if and only if the related
element of $\Hom_k^1(C^+,A^+)\subset\Hom_k^1(C,A)$ satisfies
the Maurer--Cartan equation.
 Similarly, a graded coalgebra homomorphism $C\rarrow\Br_\alpha(A)$
commutes with the differentials on $C^\bu$ and $\Br^\bu_\alpha(A^\bu)$
if and only if the related element of $\Hom_k^1(C^+,A^+)\subset
\Hom_k^1(C,A)$ satisfies the same Maurer--Cartan equation.

 To sum up, both the augmented DG\+algebra homomorphisms
$\Cb^\bu_\gamma(C^\bu)\rarrow A^\bu$ and the (conilpotent)
DG\+coalgebra homomorphisms $C^\bu\rarrow\Br^\bu_\alpha(A^\bu)$
correspond bijectively to twisting cochains $\tau\:C^\bu\rarrow A^\bu$
satisfying the equations of compatibility with the augmentation
and the coaugmentation $\alpha\circ\tau=0=\tau\circ\gamma$.
\end{proof}

 Let $(A^\bu,\alpha)$ be an augmented DG\+algebra, $(C^\bu,\gamma)$ be
a conilpotent DG\+coalgebra, and $\tau\:C^\bu\rarrow A^\bu$ be
a twisting cochain satisfying the equations $\alpha\circ\tau=0=
\tau\circ\gamma$.
 The twisting cochain~$\tau$ is said to be \emph{acyclic} if one of
(or equivalently, both) the related DG\+algebra homomorphism
$\Cb^\bu_\gamma(C^\bu)\rarrow A^\bu$ and the DG\+coalgebra
homomorphism $C^\bu\rarrow\Br^\bu_\alpha(A^\bu)$ become isomorphisms
after the quasi-isomorphisms of DG\+algebras, or respectively
the filtered quasi-isomorphisms of conilpotent DG\+coalgebras, are
inverted, as in Theorem~\ref{augmented-algebras-coalgebras-duality}.
 Simply put, $\tau$~is called acyclic if the related homomorphism
of DG\+algebras $\Cb^\bu_\gamma(C^\bu)\rarrow A^\bu$ is
a quasi-isomorphism.

 For example, the twisting cochain from
Examples~\ref{cobar-construction-twisting-cochains-examples}
is acyclic, by the definition, whenever the coaugmented coalgebra
$(C,\gamma)$ or the coaugmented DG\+coalgebra $(C^\bu,\gamma)$
is conilpotent.
 It is a part of Theorem~\ref{augmented-algebras-coalgebras-duality}
that the twisting cochain from 
Examples~\ref{bar-construction-twisting-cochains-examples}
is acyclic for any augmented algebra $(A,\alpha)$ or
an augmented DG\+algebra $(A^\bu,\alpha)$.

 Concerning the twisting cochain~$\tau$ from
Example~\ref{nonhomogeneous-quadratic-twisting-cochain},
it is acyclic for any (augmented) nonhomogeneous \emph{Koszul}
algebra~$(A,F)$.
 This claim is a corollary of the proof of Poincar\'e--Birkhoff--Witt
theorem for nonhomogeneous quadratic/Koszul algebras
in~\cite[Sections~3.2\+-3.3]{Pcurv}
or~\cite[Proposition~7.2(ii) in Chapter~5]{PP}.

\subsection{Twisted differential on the tensor product}
\label{twisted-differential-on-tensor-product-subsecn}
 Let $(A,\alpha)$ be an augmented algebra and $M^\bu$ be a complex of
left $A$\+modules.
 Let us make some basic observations about the bar-constructions
$\Br^\bu_\alpha(A)$ and $\Br^\bu_\alpha(A,M^\bu)$, as defined in
Sections~\ref{augmented-algebras-subsecn}
and~\ref{algebras-modules-posing-the-problem-subsecn}.

 The underlying graded vector space $\Br_\alpha(A,M)$ of the complex
$\Br^\bu_\alpha(A,M^\bu)$ is isomorphic to the tensor product
$\Br_\alpha(A)\ot_kM$ of the underlying graded vector spaces of
the complexes $\Br^\bu_\alpha(A)$ and~$M^\bu$.
 But the differential on $\Br^\bu_\alpha(A,M^\bu)$ is \emph{not}
the tensor product of the differentials on $\Br^\bu_\alpha(A)$
and~$M^\bu$.
 Rather, in the formula for the differential of
$\Br^\bu_\alpha(A,M^\bu)$, there are the summands coming from
the differential of $\Br^\bu_\alpha(A)$, there is the summand coming
from the differential of $M^\bu$, but there is also one additional
summand, defined in terms of the action of $A$ in~$M$.

 Similarly, let $(C,\gamma)$ be a coaugmented coalgebra and $N^\bu$ be
a complex of left $C$\+comodules.
 Consider the cobar-constructions $\Cb^\bu_\gamma(C)$ and
$\Cb^\bu_\gamma(C,N^\bu)$, as defined in
Sections~\ref{cobar-construction-subsecn}
and~\ref{duality-formulated-for-complexes-of-comodules}.
 
 The underlying graded vector space $\Cb_\gamma(C,N)$ of the complex
$\Cb^\bu_\gamma(C,N^\bu)$ is isomorphic to the tensor product
$\Cb_\gamma(C)\ot_kN$ of the underlying graded vector spaces of
the complexes $\Cb^\bu_\gamma(C)$ and~$N^\bu$.
 But the differential on $\Cb^\bu_\gamma(C,N^\bu)$ is \emph{not}
the tensor product of the differentials on $\Cb^\bu_\gamma(C)$
and~$N^\bu$.
 Rather, in the formula for the differential of
$\Cb^\bu_\gamma(C,N^\bu)$, there are the summands coming from
the differential of $\Cb^\bu_\gamma(C)$, there is the summand coming
from the differential on $N^\bu$, but there is also one additional
summand, defined in terms of the coaction of $C$ in~$N$.

 The construction of the twisted differential on the tensor product
of a DG\+module and a DG\+comodule (twisted by a twisting
cochain~$\tau$) is abstracted from these observations.
 Let $A^\bu$ be a DG\+algebra and $C^\bu$ be a DG\+coalgebra over~$k$.
 Let $\tau\:C^\bu\rarrow A^\bu$ be a twisting cochain.

 Let $M^\bu=(M,d_M)$ be a right DG\+module over $A^\bu$, and let
$N^\bu=(N,d_N)$ be a left DG\+comodule over~$C^\bu$.
 Consider the tensor product $M\ot_k N$ of the underlying graded
vector spaces of $M^\bu$ and $N^\bu$, and endow it with
the differential given by the formula
$$
 d(x\ot y)=d_M(x)\ot y + (-1)^{|x|}x\ot d_N(y)\pm d^\tau(x\ot y)
$$
for all homogeneous elements $x\in M$ and $y\in N$ of
degrees~$|x|$ and~$|y|$.
 Here $d^\tau\:M\ot_k N\rarrow M\ot_kN$ is the composition
$$
 M\ot_k N\lrarrow M\ot_k C\ot_k N\lrarrow M\ot_k A\ot_k N
 \lrarrow M\ot_k N
$$
of the map induced by the comultiplication map $N\rarrow C\ot_k N$,
the map induced by the twisting cochain $\tau\:C\rarrow A$, and
the map induced by the multiplication map $M\ot_k A\rarrow M$.
 The reader can consult with~\cite[Section~6.2]{Pkoszul} for
the sign rule.

 Then one can check that $d^2(x\ot y)=0$.
 So the tensor product of graded vector spaces $M\ot_k N$ endowed
with the differential~$d$ is a complex.
 We denote this complex by $M^\bu\ot^\tau N^\bu$.

 Similarly, let $M^\bu=(M,d_M)$ be a left DG\+module over $A^\bu$,
and let $N^\bu=(N,d_N)$ be a right DG\+comodule over~$C^\bu$.
 Consider the tensor product $N\ot_kM$ of the respective underlying
graded vector spaces, and endow it with the differential given by
the formula
$$
 d(y\ot x)=d_N(y)\ot x + (-1)^{|y|}y\ot d_M(x)\pm d^\tau(y\ot x),
$$
where $d^\tau\:N\ot_k M\rarrow N\ot_kM$ is the composition
$$
 N\ot_k M\lrarrow N\ot_k C\ot_k M\overset\tau\lrarrow N\ot_k A\ot_k M
 \lrarrow N\ot_k M.
$$

 Once again, choosing the sign properly
(cf.~\cite[Section~6.2]{Pkoszul}), one can can check that
$d^2(y\ot x)=0$.
 So the tensor product $N\ot_k M$ endowed with the differential~$d$
is a complex.
 We denote this complex by $N^\bu\ot^\tau M^\bu$.

\subsection{Derived Koszul duality on the comodule side in
the augmented case} \label{augmented-duality-comodule-side-subsecn}
 Let $A^\bu$ be a DG\+algebra and $C^\bu$ be a DG\+coalgebra
over~$k$.
 Suppose that we are given a twisting cochain $\tau\:C^\bu\rarrow
A^\bu$.

 Given a left DG\+comodule $N^\bu$ over $C^\bu$, we consider
the complex of vector spaces $A^\bu\ot^\tau N^\bu$ constructed
in Section~\ref{twisted-differential-on-tensor-product-subsecn}.
 Then the structure of right DG\+module over $A^\bu$ on $A^\bu$ has
been eaten up in the construction of the twisted differential on
the tensor product, but the left DG\+module structure of $A^\bu$
over $A^\bu$ is inherited by the twisted tensor product, making
$A^\bu\ot^\tau N^\bu$ a left DG\+module over~$A^\bu$.

 Similarly, given a left DG\+module $M^\bu$ over $A^\bu$, we consider
the complex of vector spaces $C^\bu\ot^\tau M^\bu$ constructed
in Section~\ref{twisted-differential-on-tensor-product-subsecn}.
 Then the right DG\+comodule structure over $C^\bu$ on $C^\bu$ has
been consumed in the construction of the twisted differential on
the tensor product, but the left DG\+comodule structure of $C^\bu$
over $C^\bu$ is inherited by the twisted tensor product, making
$C^\bu\ot^\tau M^\bu$ a left DG\+comodule over~$C^\bu$.

 Denoting by $A^\bu\modl$ the DG\+category of left DG\+modules over
$A^\bu$ and by $C^\bu\comodl$ the DG\+category of left DG\+comodules
over $C^\bu$, one observes that the DG\+functor
$$
 A^\bu\ot^\tau{-}\,\:C^\bu\comodl\lrarrow A^\bu\modl
$$
is left adjoint to the DG\+functor
$$
 C^\bu\ot^\tau{-}\,\:A^\bu\modl\lrarrow C^\bu\comodl.
$$

 Similarly to
Section~\ref{duality-formulated-for-complexes-of-comodules},
we formulate two versions of the derived Koszul duality theorem,
a conilpotent and a nonconilpotent one.
 Let $(A^\bu,\alpha)$ be an augmented DG\+algebra and $(C^\bu,\gamma)$
be a coaugmented DG\+coalgebra.
 Suppose that we are given a twisting cochain $\tau\:C^\bu\rarrow A^\bu$
satisfying the equations $\alpha\circ\tau=0=\tau\circ\gamma$.

\begin{thm} \label{augmented-acyclic-twisting-cochain-duality-thm}
 In the context of the previous paragraph, assume further that $C^\bu$
is a conilpotent DG\+coalgebra (as defined in
Section~\ref{conilpotent-dg-coalgebras-subsecn})
and $\tau$~is an acyclic twisting cochain (as defined in
Section~\ref{bar-cobar-and-acyclic-twisting-cochains-subsecn}).
 Then the adjoint functors $M^\bu\longmapsto C^\bu\ot^\tau M^\bu$
and $N^\bu\longmapsto A^\bu\ot^\tau N^\bu$ induce a triangulated
equivalence between the conventional \emph{derived} category of
left DG\+modules over $A^\bu$ and the \emph{coderived} category of
left DG\+comodules over~$C^\bu$,
$$
 \sD(A^\bu\modl)\simeq\sD^\co(C^\bu\comodl).
$$
\end{thm}

\begin{proof}
 This is a particular case of~\cite[Theorem~6.5(a)]{Pkoszul};
see~\cite[Th\'eor\`eme~2.2.2.2]{Lef} and~\cite[Section~4]{Kel} for
an earlier approach.
 The definition of the coderived category will be explained below
in Section~\ref{coderived-cdg-comodules-subsecn}.
\end{proof}

\begin{thm} \label{augmented-nonconilpotent-duality-thm}
 Let $(C^\bu,\gamma)$ be a coaugmented DG\+coalgebra, and let
$\tau\:C^\bu\rarrow\Cb^\bu_\gamma(C^\bu)=A^\bu$ be the twisting
cochain constructed in
Example~\ref{cobar-construction-twisting-cochains-examples}(2).
  Then the adjoint functors $M^\bu\longmapsto C^\bu\ot^\tau M^\bu$
and $N^\bu\longmapsto A^\bu\ot^\tau N^\bu$ induce a triangulated
equivalence between the \emph{absolute derived} category of
left DG\+modules over $A^\bu$ and the \emph{coderived} category of
left DG\+comodules over~$C^\bu$,
$$
 \sD^\abs(A^\bu\modl)\simeq\sD^\co(C^\bu\comodl).
$$
\end{thm}

\begin{proof}
 This is a particular case of~\cite[Theorem~6.7(a)]{Pkoszul}.
 For the definitions of the coderived and absolute derived categories,
see Sections~\ref{co-contra-derived-cdg-modules-subsecn}\+-%
\ref{coderived-cdg-comodules-subsecn} below.
 The absolute derived category $\sD^\abs(A^\bu\modl)$ coincides with
the coderived category $\sD^\co(A^\bu\modl)$
by~\cite[Theorem~3.6(a)]{Pkoszul}; see
Theorem~\ref{fin-homol-dim-derived-cdg-modules}(a) below.
\end{proof}

 In both Theorems~\ref{augmented-acyclic-twisting-cochain-duality-thm}
and~\ref{augmented-nonconilpotent-duality-thm}, the triangulated
equivalence takes the left DG\+module~$k$ over~$A^\bu$ (with
the $A$\+module structure on~$k$ defined in terms of
the augmentation~$\alpha$) to the cofree left DG\+comodule $C^\bu$
over~$C^\bu$.
 The same equivalences take the free left DG\+module $A^\bu$ over
$A^\bu$ to the left DG\+comodule~$k$ over~$C^\bu$ (with
the $C$\+comodule structure on~$k$ defined in terms of
the coaugmentation~$\gamma$).

 For generalizations of
Theorems~\ref{augmented-acyclic-twisting-cochain-duality-thm}
and~\ref{augmented-nonconilpotent-duality-thm}, see
Section~\ref{koszul-duality-comodule-side-subsecn} below.

\Section{CDG-Rings and CDG-Coalgebras}

 Under Koszul duality, lack of chosen (co)augmentation on one side
corresponds to curvature on the other side.
 For example, to a nonaugmented algebra a curved DG\+coalgebra is
assigned.
 Moreover, a change of (co)augmentation corresponds to a change of
flat connection, or in other words, to a Maurer--Cartan twist of
the DG\+structure on the other side.
 The aim of this section is to extend the Koszul duality theorems
of Sections~\ref{duality-dg-algebras-dg-coalgebras-subsecn}
and~\ref{augmented-duality-comodule-side-subsecn} to
the nonaugmented context.

\subsection{Posing the problem of nonaugmented Koszul duality}
\label{posing-the-problem-nonaugmented-duality}
 Suppose that we are given an associative algebra $A$, and we would
like to compute the derived category of modules over it, in terms of
some kind of Koszul duality.
 Following the approach of Sections~\ref{algebras-and-modules-secn}
and~\ref{duality-formulated-for-complexes-of-modules-subsecn}, we
have to choose an augmentation of $A$ and produce the bar-construction
$C^\bu=\Br^\bu_\alpha(A)$ using an augmentation $\alpha\:A\rarrow k$.

 But an augmentation does not seem to be relevant to the problem of
describing $A$\+modules or complexes of $A$\+modules.
 What if $A$ does not admit an augmentation?
 A $k$\+algebra homomorphism $A\rarrow k$ does not always exist.
 Or if $A$ has many possible augmentations, what is the point of
choosing one of them?

 To be sure, the bar-complex $\Br^\bu_\alpha(A,M)$ computes
the vector spaces $\Tor^A_n(k,M)$ for a given $A$\+module $M$,
as mentioned in Section~\ref{bar-of-modules-subsecn}.
 But one can have $\Tor^A_n(k,M)=0$ for all $n\in\boZ$, while
$M\ne0$, as per Examples~\ref{bar-of-modules-acyclic}.
 This is the reason why the \emph{coderived category} of DG\+comodules,
rather than the conventional derived category, appears in
Theorem~\ref{complexes-of-modules-augmented-bar-construction-thm}.
 What is the point of choosing an $A$\+module structure on~$k$ and
considering the homological functor $\Tor^A_*(k,{-})$, only to
discover that this functor annihilates a big part of the desired
derived category $\sD(A\modl)$\,?

 The approach worked out in~\cite{Pkoszul} and going back
to~\cite{Pcurv} can be briefly stated as follows.
 Instead of looking for an augmentation $\alpha\:A\rarrow k$, let us
choose an arbitrary $k$\+linear map $v\:A\rarrow k$ for which $v(1)=1$.
 Surely any nonzero $k$\+algebra $A$ admits plenty of such maps~$v$.

 Let us extend the bar-construction $\Br^\bu_\alpha$ to the context
of arbitrary $k$\+linear maps~$v$ as above.
 Then the resulting object $\Br^\cu_v(A)$ is not a DG\+algebra, and
\emph{not even a complex} at all.
 But it is a graded coalgebra $C$ endowed with an odd coderivation
$\d\:C^n\rarrow C^{n+1}$ with a \emph{nonzero} square.
 The square of the differential~$\d$ is described as the commutator
with the \emph{curvature linear function} $h\:C^{-2}\rarrow k$.

 The resulting algebraic object $C^\cu=(C,\d,h)$ is called
a \emph{curved DG\+coalgebra}, or a \emph{CDG\+coalgebra} for brevity.
 To any $A$\+module $M$, we assign a \emph{CDG\+comodule}
$\Br^\cu_v(A,M)$ over~$C^\cu$; and similarly to any complex of
$A$\+modules~$M^\bu$.
 The differential on $\Br^\cu_v(A,M)$ does \emph{not} square to zero;
so the homology spaces of $\Br^\cu_v(A,M)$ (as well as of
$\Br^\cu_v(A)$) are \emph{undefined}.
 But we have already decided that we are not too much interested in
the Tor spaces $\Tor^A_*(k,M)$ anyway.

 Replacing the $k$\+linear map $v\:A\rarrow k$ by another such map
$v'\:A\rarrow k$, satisfying the same condition $v'(1)=1$, leads
to a CDG\+coalgebra $(C,d',h')$ naturally isomorphic to $(C,d,h)$.
 The DG\+categories of CDG\+comodules over $(C,d',h')$ and
$(C,d,h)$ are isomorphic.

 Finally, though one cannot speak of quasi-isomorphisms of
CDG\+comodules in the conventional sense of the word, and therefore
the conventional derived category of CDG\+comodules over $(C,d,h)$
is \emph{undefined}, the definition of the \emph{coderived} category
makes perfect sense for CDG\+comodules.
 The problem of computing the derived category of $A$\+modules is
solved by constructing a triangulated equivalence between the derived
category $\sD(A\modl)$ and the coderived category
$\sD^\co(C^\cu\comodl)$ of CDG\+comodules over $C^\cu=(C,d,h)$.

 Unrelated to the Koszul duality theory, curved DG\+modules also
appear in the literature in connection with the popular topic of
\emph{matrix factorizations} (which are rather special particular
cases of CDG\+modules) \cite{Eis,Buch,KL,Or0}.
 The coderived and absolute derived categories are an important
technical tool in the matrix factorization
theory~\cite{Or,PP2,EP,BDFIK,Pedg}.
 (Weakly) curved $\mathrm{A}_\infty$\+algebras play a fundamental
role in the Fukaya theory~\cite{Fu,FOOO,Cho,Pweak}.

\subsection{CDG-rings and CDG-modules}  \label{cdg-rings-subsecn}
 The following definitions go back to the paper~\cite{Pcurv}.
 The terminology ``curvature'' and ``connection'' comes from
an analogy with the respective concepts from differential geometry,
based on examples from differential geometry~\cite[Section~4]{Pcurv},
\cite[Sections~10.2\+-10.8]{Prel}.
 The latter class of examples, namely, the duality between the rings of
differential operators and (curved) DG\+rings of differential forms,
is an instance of a more complicated \emph{relative} version of
nonhomogeneous Koszul duality~\cite{Kap}, \cite[Section~7.2]{BD2},
\cite[Section~0.4 and Chapter~11]{Psemi}, \cite[Appendix~B]{Pkoszul},
\cite{Prel}, which falls outside of the scope of this survey.
 We refer to the book~\cite{Prel} for a definitive treatment.

 A \emph{curved DG\+ring} (\emph{CDG\+ring}) $B^\cu=(B,d,h)$ is a graded
ring $B=\bigoplus_{i\in\boZ}B^i$ endowed with the following data:
\begin{itemize}
\item $d\:B\rarrow B$ is an odd derivation of degree~$1$, that is,
for every $i\in\boZ$ an additive map $d_i\:B^i\rarrow B^{i+1}$ is
given such that the Leibniz rule with signs
$$
 d(bc)=d(b)c+(-1)^{|b|}bd(c)
$$
is satisfied for all homogeneous elements $b$ and $c\in B$ of
degrees~$|b|$ and~$|c|$;
\item $h\in B^2$ is an element.
\end{itemize}
 The following axioms relating $d$ and~$h$ must be satisfied:
\begin{enumerate}
\renewcommand{\theenumi}{\roman{enumi}}
\item\label{cdg-ring-differential-squared} the square of
the differential~$d$ on $B$ is described by the formula $d^2(b)=hb-bh$
for all $b\in B$;
\item $d(h)=0$.
\end{enumerate}
 The element $h\in B^2$ is called the \emph{curvature element}.

 A DG\+ring is the same thing as a CDG\+ring with $h=0$.
 The category of DG\+rings is a subcategory, but \emph{not} a full
subcategory of the category of CDG\+rings; the passage from
the uncurved to the curved DG\+rings involves not only adding new
objects to the category, but also adding new morphisms between
previously existing objects.
 Let us define the morphisms of CDG\+rings now.

 Let $B^\cu=(B,d_B,h_B)$ and $A^\cu=(A,d_A,h_A)$ be two CDG\+rings.
 A \emph{morphism of CDG\+rings} $B^\cu\rarrow A^\cu$ is a pair $(f,a)$,
where
\begin{itemize}
\item $f\:B\rarrow A$ is a homomorphism of graded rings;
\item $a\in A^1$ is an element
\end{itemize}
such that
\begin{enumerate}
\renewcommand{\theenumi}{\roman{enumi}}
\setcounter{enumi}{2}
\item $f(d_B(z))=d_A(f(z))+[a,f(z)]$ for all $z\in B$ (where the
graded commutator $[{-},{-}]$ is defined by the usual rule
$[x,y]=xy-(-1)^{|x||y|}yx$ for all homogeneous elements $x$ and~$y$
of degrees $|x|$ and~$|y|$);
\item $f(h_B)=h_A+d_A(a)+a^2$.
\end{enumerate}
 The element $a\in A^1$ is called the \emph{change-of-connection}
element.

 The composition of two morphisms of CDG\+rings $(C,d_C,h_C)\rarrow
(B,d_B,h_B)\rarrow(A,d_A,h_A)$ is given by the rule $(f,a)\circ(g,b)
=(f\circ g,\>a+f(b))$.
 The identity morphism $(B,d_B,h_B)\rarrow(B,d_B,h_B)$ is
the morphism $(\id_B,0)$.

 Morphisms of CDG\+rings $(\id_B,a)\:(B,d',h')\rarrow(B,d,h)$ are
called the \emph{change-of-connection} morphisms.
 All such morphisms of CDG\+rings are isomorphisms.
 Moreover, for any CDG\+ring $(B,d,h)$ and an element $a\in B^1$
there exists a unique CDG\+ring structure $(B,d',h')$ on the graded
ring $B$ such that $(\id_B,a)\:(B,d',h')\rarrow(B,d,h)$ is
a (change-of-connection iso)morphism of CDG\+rings.
 The twisted differential and curvature element $d'\:B\rarrow B$
and $h'\in B^2$ are given by the formulas $d'(z)=d(z)+[a,z]$ and
$h'=h+d(a)+a^2$.

 In particular, one can assume that $h=0$; so $(B,d)$ is a DG\+ring.
 Then one has $h'=0$ (i.~e., the pair $(B,d')$ is a DG\+ring again)
if and only if the equation $d(a)+a^2=0$ is satisfied, i.~e.,
$a\in B^1$ is a Maurer--Cartan element (in the sense of
Section~\ref{hom-dg-algebra-and-twisting-cochains-subsecn}).
 So any DG\+ring structure can be twisted by any Maurer--Cartain
cochain, producing a new DG\+ring structure on the same graded ring.
 The resulting DG\+ring $(B,d')$ is naturally isomorphic to
the original DG\+ring $(B,d)$ \emph{as a CDG\+ring}, but not
as a DG\+ring.

 The following definition of a \emph{CDG\+module} can be found
in~\cite[Section~4 of Chapter~5]{PP} or~\cite[Section~3.1]{Pkoszul}.
 Let $B^\cu=(B,d,h)$ be a CDG\+ring.
 A \emph{left CDG\+module} $M^\cu=(M,d_M)$ over $(B,d,h)$ is a graded
left $B$\+module $M=\bigoplus_{i\in\boZ}M^i$ endowed with
the following datum:
\begin{itemize}
\item $d_M\:M\rarrow M$ is an odd derivation of degree~$1$ compatible
with the odd derivation~$d$ on $B$, that is, for every $i\in\boZ$
an additive map $d_{M,i}\:M^i\rarrow M^{i+1}$ is given such that
the Leibniz rule with signs
$$
 d_M(bx)=d(b)x+(-1)^{|b|}bd_M(x)
$$
is satisfied for all homogeneous elements $b\in B$ and $x\in M$ of
degrees $|b|$ and~$|x|$.
\end{itemize}
 The following axiom must be satisfied:
\begin{enumerate}
\renewcommand{\theenumi}{\roman{enumi}}
\setcounter{enumi}{4}
\item\label{left-cdg-module-differential-squared} the square of
the differential~$d_M$ on $M$ is described by the formula $d_M^2(x)=hx$
for all $x\in M$.
\end{enumerate}

 Similarly, a \emph{right CDG\+module} $N^\cu=(N,d_N)$ over $(B,d,h)$
is a graded right $B$\+module $N=\bigoplus_{i\in\boZ}N^i$ endowed with
the following datum:
\begin{itemize}
\item $d_N\:N\rarrow N$ is an odd derivation of degree~$1$ compatible
with the odd derivation~$d$ on $B$, i.~e., $d_N$~acts on
the grading components of $N$ as $d_{N,i}\:N^i\rarrow N^{i+1}$ and
the Leibniz rule with signs
$$
 d_N(yb)=d_N(y)b+(-1)^{|y|}yd(b)
$$
is satisfied for all homogeneous elements $b\in B^{|b|}$ and
$y\in N^{|y|}$.
\end{itemize}
 The following axiom is imposed:
\begin{enumerate}
\renewcommand{\theenumi}{\roman{enumi}}
\setcounter{enumi}{5}
\item\label{right-cdg-module-differential-squared} the square of
the differential~$d_N$ on $N$ is given by the formula $d_N^2(y)=-yh$
for all $y\in N$.
\end{enumerate}

 Comparing the equation for the square of the differential
in~\eqref{cdg-ring-differential-squared}
with~\eqref{left-cdg-module-differential-squared}
and~\eqref{right-cdg-module-differential-squared}, one observes that
they do not agree.
 So \emph{there is no natural left} (or \emph{right}) \emph{CDG\+module
structure on the underlying graded ring $B$ of a CDG\+ring $(B,d,h)$}.
 However, any CDG\+ring $B^\cu$ is naturally a \emph{CDG\+bimodule}
over itself, in the sense of the definition
in~\cite[Section~3.10]{Pkoszul} or~\cite[Section~6.1]{Prel}.

 To any two (say, left) CDG\+modules $L^\cu=(L,d_L)$ and $M^\cu=(M,d_M)$
over a CDG\+ring $(B,d,h)$, one can assign the \emph{complex of
morphisms} $\Hom_B^\bu(L^\cu,M^\cu)$ from $L$ to~$M$.
 The degree~$n$ component $\Hom_B^n(L,M)$ of the complex
$\Hom_B^\bu(L^\cu,M^\cu)$ is the group of all homogeneous $B$\+linear
maps $L\rarrow M$ of degree~$n$ (see~\cite[Sections~1.1
and~3.1]{Pkoszul} for the sign rule involved).
 The differential~$d$ on the complex $\Hom_B^\bu(L^\cu,M^\cu)$ is
defined by the usual rule $d(f)(l)=d_M(f(l))-(-1)^{|f|}f(d_L(l))$
for all $l\in L$.

 The rule above is well-known to define a differential with zero square
on the graded abelian group of homogeneous morphisms between two
\emph{DG\+modules}.
 It turns out that for two CDG\+modules over a CDG\+ring, the same
formula still defines a differential with zero square, as two
curvature-related terms in the computation of $d^2(f)$ cancel
each other.
 Consequently, there is the \emph{DG\+category of left CDG\+modules}
over $B^\cu=(B,d,h)$, which we denote simply by $B^\cu\modl$.

 Given a morphism of CDG\+rings $(f,a)\:(B,d_B,h_B)\rarrow (A,d_A,h_A)$
and a left CDG\+module $(M,d_M)$ over $(A,d_A,h_A)$, one can endow
the graded left $A$\+module $M$ with the induced structure of graded
left $B$\+module and a twisted differential~$d'_M$, making $(M,d'_M)$
a left CDG\+module over $(B,d_B,h_B)$.
 The twisted differential $d'_M\:M\rarrow M$ is defined by the formula
$d'_M(x)=d_M(x)+ax$ for all $x\in M$.
 Similarly, for a right CDG\+module $(N,d_N)$ over $(A,d_A,h_A)$,
the twisted differential~$d'_N$ defined by the rule
$d'_N(y)=d_N(y)-(-1)^{|y|}ya$ for all $y\in N^{|y|}$ makes $(N,d'_N)$
a right CDG\+module over $(B,d_B,h_B)$.

 In particular, the DG\+categories of CDG\+modules over two isomorphic
CDG\+rings are naturally isomorphic.
 Specializing to CDG\+isomorphisms (or in other words,
change-of-connection isomorphisms) of DG\+rings, we come to
the (perhaps not too widely known) observation that
\emph{a Maurer--Cartan twist of the differential on a DG\+ring
does not change the DG\+category of DG\+modules}.
 The reason is that the differentials on DG\+modules $(M,d_M)$ over
a DG\+ring $(B,d)$ can be twisted (by a Maurer--Cartan cochain
$a\in B^1$) alongside with a twist $d\leadsto d'=d+[a,{-}]$ of
the differential on~$B$ (according to the rule above,
$d'_M(x)=d_M(x)+ax$).

\subsection{CDG-coalgebras and CDG-comodules}
\label{cdg-coalgebras-subsecn}
 A \emph{CDG\+algebra} $(B,d,h)$ over a field~$k$ is a graded
(associative, unital) $k$\+algebra endowed with a CDG\+ring structure
with a $k$\+linear differential~$d$.
 The definitions of a \emph{CDG\+coalgebra} $(C,d,h)$ and
a \emph{CDG\+comodule} $(M,d_M)$ over it are obtained from those of
a CDG\+algebra and a CDG\+module by inverting the arrows.
 Let us spell out some details.

 Partly following the terminology of~\cite{KLN}, let us utilize the term
\emph{precomplex} for a graded vector space $K$ endowed with
a $k$\+linear differential $d\:K\rarrow K$ of degree~$1$
with \emph{not necessarily zero square}.
 A CDG\+coalgebra $(C,d,h)$ and a CDG\+comodule $(M,d_M)$ are, first
of all, precomplexes: both $d\:C\rarrow C$ and $d_M\:M\rarrow M$ are
homogeneous $k$\+linear maps of degree~$1$ with, generally speaking,
nonzero squares.

 Given two precomplexes $V^\cu=(V,d_V)$ and $W^\cu=(W,d_W)$, one defines
their \emph{tensor product} $V^\cu\ot_kW^\cu=(V\ot_k W,\>d)$ as follows.
 The graded vector space $V\ot_k W$ is simply the tensor product of
the graded vector spaces $V$ and $W$.
 The differential~$d$ on $V\ot_kW$ is given by the usual rule
$d(v\ot w)=d_V(v)\ot w+(-1)^{|v|}v\ot d_W(w)$ for $v\in V^{|v|}$
and $w\in W^{|w|}$.

 Now we can define the notions of \emph{coderivations} on coalgebras
and comodules.
 Let $C$ be a graded coalgebra over~$k$.
 An \emph{odd coderivation} (of degree~$1$) on $C$ is a homogeneous
$k$\+linear map $d\:C\rarrow C$ of degree~$1$ such that
the comultiplication map $\mu\:C\rarrow C\ot_k C$ is a morphism of
precomplexes (i.~e., $\mu$~commutes with the differentials).
 Here the differential~$d$ on $C\ot_k C$ is defined by the rule above.
 One can check that any odd coderivation~$d$ on $C$ is compatible with
the counit, in the sense that the counit $\epsilon\:C\rarrow k$ is
also a morphism of precomplexes (where the differential on~$k$ is zero).

 Let $(C,d)$ be a graded coalgebra endowed with an odd coderivation
of degree~$1$, and let $M$ be a graded left $C$\+comodule.
 Then an \emph{odd coderivation} on $M$ \emph{compatible with}
the coderivation~$d$ on $C$ is a $k$\+linear map $d_M\:M\rarrow M$
of degree~$1$ such that the left coaction map $\nu\:M\rarrow C\ot_k M$
is a morphism of precomplexes.
 Here, once again, the differential~$d$ on $C\ot_k M$ is defined
by the rule above in terms of the differentials~$d$ on $C$ and~$d_M$
on~$M$.

 In order to spell out the rules for the squares of the differentials
$d$ and~$d_M$ on a CDG\+coalgebra $C$ and a CDG\+comodule $M$, we need
to have a brief preliminary discussion of algebras dual to coalgebras
and their actions in comodules.
 It was already mentioned in
Section~\ref{local-finite-dimensionality-subsecn} that the dual vector
space to any coalgebra is an algebra.
 Similarly, the graded dual vector space $C^*$ to a graded coalgebra $C$
is a graded algebra.
 Furthermore, any graded left $C$\+comodule $M$ can be endowed with
a graded $C^*$\+module structure with the action map defined as
the composition
$$
 C^*\ot_k M\lrarrow C^*\ot_k C\ot_k M\lrarrow M
$$
of the map induced by the left coaction map $\nu\:M\rarrow C\ot_kM$
and the map induced by the pairing map $C^*\ot_k C\rarrow k$.

 Let us pause at this point, however, and observe that for any
algebra $A$ there is the opposite algebra~$A^\rop$.
 Which one of the two opposite multiplications on $C^*$ should one
choose and use as the canonical choice of the multiplication on
the dual vector space to a coalgebra?
 The traditional way of making this choice~\cite[Section~2.1]{Swe}
results in \emph{left} $C$\+comodules becoming \emph{right}
$C^*$\+modules and vice versa.
 Our usual preference is to choose the sides so that left $C$\+comodules
become left $C^*$\+modules and right $C$\+comodules become right
$C^*$\+modules (cf.~\cite[beginning of Section~1.4]{Prev}).

 With these observations in mind, let us introduce some notation.
 Given a (graded) coalgebra $C$ and two (homogeneous) $k$\+linear
functions $a$, $b\:C\rarrow k$ (of some degrees $|a|$, $|b|\in\boZ$),
we denote by $a*b\:C\rarrow k$ the product of $a$ and~$b$ in
the graded algebra~$C^*$.
 Given a graded left comodule $M$ over $C$, an element $x\in M$,
and a linear function $b\:C\rarrow k$, we denote by $b*x\in M$
the result of the left action of the element $b\in C^*$ on
the element~$x$ in the left $C^*$\+module~$M$.
 Similarly, for a graded right comodule $N$ over $C$, a homogeneous
element $y\in N$, and a homogeneous linear function $b\:C\rarrow k$,
we let $y*b\in N$ denote the result of the right action of~$b$ on~$y$.
 We refer to~\cite[Section~4.1]{Pkoszul} for the sign rules.

 Now we can dualize the definitions from
Section~\ref{cdg-rings-subsecn}.
 A \emph{curved DG\+coalgebra} (\emph{CDG\+coalgebra}) $C^\cu=(C,d,h)$
over a field~$k$ is a graded coalgebra $C=\bigoplus_{i\in\boZ}C^i$
endowed with the following data:
\begin{itemize}
\item $d\:C\rarrow C$ is an odd coderivation of degree~$1$ (so
the grading components of~$d$ are $d_i\:C^i\rarrow C^{i+1}$);
\item $h\:C\rarrow k$ is a homogeneous linear function of degree~$2$
(so the only grading component of~$h$ is $h\:C^{-2}\rarrow k$).
\end{itemize}
 The following axioms relating $d$ and~$h$ must be satisfied:
\begin{enumerate}
\renewcommand{\theenumi}{\roman{enumi}}
\item\label{cdg-coalgebra-square-of-differential} the square of
the differential~$d$ on $C$ is described by the formula $d^2(c)=h*c-c*h$
for all $c\in C$;
\item $h(d(c))=0$ for all $c\in C^{-3}$.
\end{enumerate}
 The linear function $h\:C^{-2}\rarrow k$ is called
the \emph{curvature linear function}.

 A DG\+coalgebra is the same thing as a CDG\+coalgebra with $h=0$.
 The category of DG\+coalgebras is a subcategory, but \emph{not}
a full subcategory in the category of CDG\+coalgebras, as one can
see from the following definition of a morphism of CDG\+coalgebras.

 Let $C^\cu=(C,d_C,h_C)$ and $D^\cu=(D,d_D,h_D)$ be two CDG\+coalgebras.
 A \emph{morphism of CDG\+coalgebras} $C^\cu\rarrow D^\cu$ is
a pair $(f,a)$, where
\begin{itemize}
\item $f\:C\rarrow D$ is a homomorphism of graded coalgebras;
\item $a\:C\rarrow k$ is a homogeneous linear function of degree~$1$
\end{itemize}
such that
\begin{enumerate}
\renewcommand{\theenumi}{\roman{enumi}}
\setcounter{enumi}{2}
\item $d_D(f(c))=f(d_C(c))+f(a*c)-(-1)^{|c|}f(c*a)$ for all
homogeneous elements $c\in C^{|c|}$;
\item $h_D(f(c))=h_C(c)+a(d_C(c))+a^{*2}(c)$ for all $c\in C$.
\end{enumerate}
 The linear function $a\:C^{-1}\rarrow k$ is called
the \emph{change-of-connection linear function}.

 The composition of morphisms and the identity morphisms in
the category of CDG\+coalgebras are defined in the way similar/dual
to the one for CDG\+rings (see~\cite[Section~4.1]{Pkoszul} for
the details).

 Let $C^\cu=(C,d,h)$ be a CDG\+coalgebra.
 A \emph{left CDG\+comodule} $M^\cu=(M,d_M)$ over $C^\cu$ is a graded
left $C$\+comodule endowed with
\begin{itemize}
\item an odd coderivation $d_M\:M\rarrow M$ of degree~$1$ compatible
with the coderivation~$d$ on~$C$
\end{itemize}
such that
\begin{enumerate}
\renewcommand{\theenumi}{\roman{enumi}}
\setcounter{enumi}{4}
\item\label{left-cdg-comodule-square-of-differential} the square of
the differential~$d_M$ on $M$ is described by the formula $d_M^2(x)=h*x$
for all $x\in M$.
\end{enumerate}
 Similarly, a \emph{right CDG\+comodule} $N^\cu=(N,d_N)$ over $C^\cu$
is a graded right $C$\+comodule endowed with
\begin{itemize}
\item an odd coderivation $d_N\:N\rarrow N$ of degree~$1$ compatible
with the coderivation~$d$ on~$C$
\end{itemize}
such that
\begin{enumerate}
\renewcommand{\theenumi}{\roman{enumi}}
\setcounter{enumi}{5}
\item\label{right-cdg-comodule-square-of-differential} the square of
the differential~$d_N$ on $N$ is given by the formula $d_N^2(y)=-y*h$
for all $y\in N$.
\end{enumerate}

 Similarly to the theory of CDG\+modules over CDG\+rings discussed
in Section~\ref{cdg-rings-subsecn}, left CDG\+comodules over
a CDG\+coalgebra $C^\cu=(C,d,h)$ form a DG\+category $C^\cu\comodl$.
 Any morphism of CDG\+coalgebras $(f,a)\:(C,d_C,h_C)\rarrow
(D,d_D,h_D)$ induces a DG\+functor $C^\cu\comodl\rarrow D^\cu\comodl$
assigning to a CDG\+comodule $(M,d_M)$ the CDG\+comodule $(M,d'_M)$,
with the graded $D$\+comodule structure on $M$ obtained from
the graded $C$\+comodule structure on $M$ by the corestriction
of scalars (as mentioned in
Section~\ref{brief-remarks-about-coalgebras-subsecn}) and
the twisted differential $d'_M$ given by the rule
$d'_M(x)=d_M(x)+a*x$.
 So an isomorphism of CDG\+coalgebras induces an isomorphism of
the DG\+categories of CDG\+comodules over them.

 Similarly to Section~\ref{cdg-rings-subsecn}, the equations for
the square of the differential
in~\eqref{cdg-coalgebra-square-of-differential},
\eqref{left-cdg-comodule-square-of-differential}
and~\eqref{right-cdg-comodule-square-of-differential} are all different.
 So a CDG\+coalgebra $C^\cu$ is naturally \emph{neither} a left,
\emph{nor} a right CDG\+comodule over itself; but it has a natural
structure of a \emph{CDG\+bicomodule}
over itself~\cite[Section~4.8]{Pkoszul}.

\subsection{Nonaugmented bar-construction for algebras and modules}
\label{nonaugmented-bar-for-algebras-and-modules-subsecn}
 Let $A$ be a nonzero associative $k$\+algebra; so $1\in A$ is
a nonzero element.
 Consider the $k$\+vector space $A_+=A/(k\cdot 1)$
(cf.\ Section~\ref{augmented-algebras-subsecn}).

 Choose a $k$\+linear map $v\:A\rarrow k$ such that $v(1)=1$.
 So $v$~is a $k$\+linear retraction of the $k$\+algebra $A$ (viewed
as a $k$\+vector space) onto its unit line $k=k\cdot 1\subset A$.
 Put $V=\ker(v)\subset A$.
 Then the composition of linear maps $V\rarrow A\rarrow A_+$ is
an isomorphism, so the vector space $A_+$ can be identified with
the subspace $V\subset A$, and $A=k\oplus V$ as a $k$\+vector space.

 The multiplication map $m\:A\ot_k A\rarrow A$ decomposes into
components according to the direct sum decomposition $A=k\oplus V$.
 Notice that the restrictions of~$m$ to $k\cdot 1\ot_kA$ and
$A\ot_k k\cdot1$ are prescribed by the condition that $1\in A$ is
a unit, so we do not need to keep track of these.
 The restriction of the map~$m$ to $V\ot_k V\subset A\ot_kA$
provides a linear map $V\ot_k V\rarrow A=k\oplus V$.
 Denote its components by $m_V\:V\ot_k V\rarrow V$ and
$m_k\:V\ot_kV\rarrow k$.

 Consider the tensor coalgebra $\udT(A_+)=\bigoplus_{n=0}^\infty
A_+^{\ot n}$, and endow it with the differential~$\d$,
$$
 k\overset0\llarrow A_+\overset\d\llarrow A_+\ot_k A_+
 \overset\d\llarrow A_+\ot_k A_+\ot_k A_+\llarrow\dotsb,
$$
given by the formulas $\d(a\ot b)=m_V(a\ot b)$, \ $\d(a\ot b\ot c)
=m_V(a\ot b)\ot c-a\ot m_V(b\ot c)$,~\dots,
\begin{multline*}
 \d(a_1\ot\dotsb\ot a_n)=m_V(a_1\ot a_2)\ot a_3\ot\dotsb\ot a_n-\dotsb
 \\ +(-1)^{i+1}a_1\ot\dotsb\ot a_{i-1}\ot m_V(a_i\ot a_{i+1})
 \ot a_{i+2}\ot\dotsb\ot a_n+\dotsb \\
 +(-1)^n a_1\ot\dotsb\ot a_{n-2}\ot m_V(a_{n-1}\ot a_n),
\end{multline*}
etc., for all $a$, $b$, $c$, $a_i\in A_+$, \ $n\ge1$, and
$1\le i\le n-1$.
 The identification $A_+\simeq V$ is presumed here.
 The leftmost differential $\d\:A_+\rarrow k$ is the zero map.

 Notice that the map $m_V\:A_+\ot A_+\rarrow A_+$ is \emph{not}
an associative multiplication.
 Accordingly, the tensor coalgebra $\udT(A_+)$ with
the differential~$\d$ is \emph{not} a complex: one has $\d^2\ne0$
(generally speaking).
 Still, essentially by construction, $\d$~is an odd coderivation
of the graded coalgebra $\udT(A_+)$, in the sense of
Section~\ref{cdg-coalgebras-subsecn} (to make it an odd coderivation
of degree~$1$, one has to change the sign of the grading on
the tensor coalgebra).

 Denote by $\Br(A)$ the tensor coalgebra $\udT(A_+)$ with the sign
of the grading changed, so $\Br^{-n}(A)=A_+^{\ot n}$.
 Furthermore, denote by $h\:\Br^{-2}(A)\rarrow k$ the linear map
$-m_k\:A_+\ot_kA_+\rarrow k$.
 Then $\Br^\cu_v(A)=(\Br(A),\d,h)$ is a CDG\+coalgebra over~$k$;
the linear function~$h$ is its curvature linear function.
 This is the nonaugmented version of the bar-construction for
an associative algebra.

 Now let $M$ be a left $A$\+module.
 Consider the cofree left comodule $\udT(A_+)\ot_k M$ over
the tensor coalgebra $\udT(A_+)$, and endow it with
the differential~$\d$,
$$
 M \overset\d\llarrow A_+\ot_k M \overset\d\llarrow
 A_+\ot_k A_+\ot_k M\llarrow\dotsb,
$$
given by the formulas $\d(a\ot x)=ax$, \ $\d(a\ot b\ot x)=
m_V(a\ot b)\ot\nobreak x-a\ot bx$,~\dots,
\begin{multline*}
 \d(a_1\ot\dotsb\ot a_n\ot x)=
 m_V(a_1\ot a_2)\ot a_3\ot\dotsb\ot a_n\ot x - \dotsb \\
 +(-1)^n a_1\ot a_2\ot\dotsb\ot a_{n-2}\ot m_V(a_{n-1}\ot a_n)
 \ot x \\ +(-1)^{n+1}a_1\ot\dotsb\ot a_{n-1}\ot a_nx,
\end{multline*}
etc., for all $a$, $b$, $a_i\in A_+$, \ $x\in M$, and $n\ge1$.
 As above, the identification $A_+\simeq V$ is presumed here
(explaining, in particular, the notation $ax$ for the action of
an element $a\in A_+\simeq V$ on an element $x\in M$).

 Once again, the graded vector space $\udT(A_+)\ot_k M$ with
the differential~$\d$ is \emph{not} a complex.
 Instead, the cofree graded comodule $\udT(A_+)\ot_k M$ over the graded
coalgebra $\udT(A_+)$, with the differential~$\d$ on the cofree graded
comodule, is a left CDG\+comodule over the CDG\+coalgebra
$\Br^\cu_v(A)=(\Br(A),\d,h)$.
 We denote this CDG\+comodule by $\Br^\cu_v(A,M)$.
 This is the nonaugmented version of the bar-construction for
a module over an associative algebra.

 Now let $v'\:A\rarrow k$ be another $k$\+linear map such that
$v'(1)=1$.
 Put $V'=\ker(v')\subset A$, and notice natural isomorphisms
$V'\simeq A_+\simeq V$.
 The difference $v-\nobreak v'\:A\rarrow k$ is a $k$\+linear map
taking~$1$ to~$0$; so it factorizes through the natural surjection
$A\rarrow A_+$, defining a $k$\+linear map $l\:A_+\rarrow k$.
 The vector subspace $V'\subset A$ can be then described as
$$
 V'=\{a'=a+l(a)\mid a\in V\}\,\subset\,A.
$$

 Following the construction above, there are the linear maps
$m_{V'}\:V'\ot_k V'\rarrow V'$ and $m_k'\:V'\ot_k V'\rarrow k$
corresponding to the choice of the retraction $v'\:A\rarrow k$
of the $k$\+algebra $A$ onto its unit line $k=k\cdot 1$.
 Furthermore, there is the related differential~$\d'$ on the tensor
coalgebra $\udT(A_+)$ and the related curvature linear function~$h'$,
forming together a CDG\+coalgebra $\Br^\cu_{v'}(A)=(\Br(A),\d',h')$.

 The linear function $l\:A_+\rarrow k$ can be viewed as
a change-of-connection linear function $l\:\Br^{-1}(A)\rarrow k$.
 Then one observes that the pair $(\id,l)$ is a change-of-connection
isomorphism of CDG\+coalgebras $\Br^\cu_v(A)\rarrow\Br^\cu_{v'}(A)$.

 For any left $A$\+module $M$, the differential~$\d'$ on
the bar-construction $\Br^\cu_{v'}(A,M)$ is obtained from
the differential~$\d$ on the bar-construction $\Br^\cu_v(A,M)$ by
twisting with the change-of-connection linear function~$l$.
 So the equivalence of DG\+categories $\Br^\cu_v(A)\comodl\simeq
\Br^\cu_{v'}(A)\comodl$ induced by the isomorphism of CDG\+coalgebras
$(\id,l)\:\Br^\cu_v(A)\rarrow\Br^\cu_{v'}(A)$ takes the CDG\+comodule
$\Br^\cu_v(A,M)$ over $\Br^\cu_v(A)$ to the CDG\+comodule
$\Br^\cu_{v'}(A,M)$ over $\Br^\cu_{v'}(A)$.

 Furthermore, given a homomorphism of associative algebras
$f\:A\rarrow B$, each of them endowed with a $k$\+linear retraction
onto its unit $v\:A\rarrow k$ and $v'\:B\rarrow k$, the difference
$v-v'\circ f\:A\rarrow k$ defines a linear function $l\:A_+\rarrow k$.
 At the same time, the map~$f$ induces a $k$\+linear map
$f_+\:A_+\rarrow B_+$, which in turn induces a homomorphism of
the tensor coalgebras $g\:\Br(A_+)\rarrow\Br(B_+)$.
 Then the pair $(g,l)$ defines a morphism of CDG\+coalgebras
$(g,l)\:\Br^\cu_v(A)\rarrow\Br^\cu_{v'}(B)$ assigned to a morphism
of associative algebras $A\rarrow B$.
 This makes the nonaugmented bar-construction functorial.

\subsection{Nonaugmented bar-construction for DG-algebras}
\label{nonaugmented-bar-for-dg-algebras-subsecn}
 Having spelled out the details of the nonaugmented bar-construction
in the simplest case of a $k$\+algebra $A$, we will now briefly sketch
the generalization to DG\+algebras~$A^\bu$.
 The following construction is a nonaugmented version of
Section~\ref{augmented-dg-algebras-subsecn}.

 Let $A^\bu=(A,d)$ be a DG\+algebra over a field~$k$.
 Assume that $A^\bu\ne0$, or equivalently, $1\in A^0$ is a nonzero
element.
 Notice that $k\cdot1\subset A^\bu$ is a subcomplex, since $d(1)=0$.
 Put $A^\bu_+=A^\bu/(k\cdot1)$; so $A^\bu_+$ is a complex of
$k$\+vector spaces, whose underlying graded vector space we denote
by~$A_+$.

 Choose a homogeneous $k$\+linear map $v\:A\rarrow k$ of degree~$0$
such that $v(1)=1$.
 So $v\:A\rarrow k$ is a homogeneous $k$\+linear retraction of
the graded $k$\+vector space $A$ onto the unit line $k=k\cdot1
\subset A$.
 Put $V=\ker(v)$; so $V$ is a homogeneous $k$\+vector subspace in $A$
such that $A=k\oplus V$.
 The composition $V\rarrow A\rarrow A_+$ is an isomorphism of
graded $k$\+vector spaces.

 Both the multiplication map $m\:A\ot_k A\rarrow A$ and the differential
$d\:A\rarrow A$ decompose into components according to the direct
sum decomposition $A=k\oplus V$.
 The restrictions of~$m$ to $k\cdot1\ot_k A$ and $A\ot_k k\cdot1$ are
prescribed by the condition that $1\in A$ is a unit, while
the restriction of~$d$ to $k\cdot1$ vanishes; so we do not need to keep
track of these.
 The restriction of the map~$m$ to the subspace $V\ot_kV\subset A\ot_kA$
provides a linear map $V\ot_kV\rarrow A=k\oplus V$, whose components
we denote by $m_V\:V\ot_kV\rarrow V$ and $m_k\:V\ot_kV\rarrow k$.
 The restriction of the map~$d$ to the subspace $V\subset A$ is
a linear map $V\rarrow A=k\oplus V$, whose components we denote by
$d_V\:V\rarrow V$ and $d_k\:V\rarrow k$.

 Denote by $\Br(A)$ the graded coalgebra $\udT(A_+[1])$.
 (Here, as usual, the cohomological degree shift~$[1]$ is responsible
for the construction of the total grading on the bigraded vector space
$\bigoplus_{n=0}^\infty A_+^{\ot n}$.)
 One observes that, for any graded $k$\+vector space $W$, odd
coderivations~$d$ of degree~$1$ on the tensor coalgebra $\udT(W)$ are
uniquely determined by their corestrictions to $W$, i.~e.,
the compositions $\udT(W)\overset d\rarrow\udT(W)\rarrow W$,
where $\udT(W)\rarrow W^{\ot1}=W$ is the direct summand projection.
 Moreover, an arbitrary homogeneous linear map $\udT(W)\rarrow W$
of degree~$1$ corresponds to some odd coderivation~$d$ on $\udT(W)$
in this way.

 The CDG\+coalgebra $\Br^\cu_v(A^\bu)=(\Br(A),d,h)$ is now
constructed as follows.
 The odd coderivation $d\:\Br(A)\rarrow\Br(A)$ of (total) degree~$1$
is defined by the condition that its composition with the projection
$\Br(A)\rarrow A_+$ has two possibly nonzero bigrading
components, given by the maps $m_V\:A_+\ot_k A_+\rarrow A_+$ and
$d_V\:A_+\rarrow A_+$.
 The curvature linear function $h\:\Br^{-2}(A)\rarrow k$ also has two
possibly nonzero bigrading components, namely, the maps
$m_k\:(A_+\ot_k A_+)^0\rarrow k$ and $d_k\:A_+^{-1}\rarrow k$.
 We refer to~\cite[Section~6.1]{Pkoszul} for the sign rules.

 Now let $v'\:A\rarrow k$ be another homogeneous $k$\+linear map
such that $v'(1)=1$.
 Then the same construction as above produces a CDG\+coalgebra
structure $\Br^\cu_{v'}(A^\bu)\allowbreak=(\Br(A),d',h')$ on
the graded coalgebra $\Br(A)$.
 Similarly to the discussion in
Section~\ref{nonaugmented-bar-for-algebras-and-modules-subsecn},
the CDG\+coalgebras $\Br^\cu_v(A^\bu)$ and $\Br^\cu_{v'}(A^\bu)$
are connected by a natural change-of-connection isomorphism
$(\id,l)\:\Br^\cu_v(A^\bu)\rarrow\Br^\cu_{v'}(A^\bu)$.
 Here $l\:\Br^{-1}(A)\rarrow k$ is a change-of-connection linear
function whose only nonzero bigrading component is the linear function
$l\:A_+^0\rarrow k$ measuring the difference between the retractions
$v$ and~$v'$.

 Recall the notation $k\alg_\dg$ introduced in
Section~\ref{duality-dg-algebras-dg-coalgebras-subsecn} for
the category of DG\+al\-gebras over~$k$.
 Denote by $k\alg_\dg^+\subset k\alg_\dg$ the full subcategory of
\emph{nonzero} DG\+algebras; and let $k\coalg_\cdg$ denote
the category of CDG\+coalgebras over~$k$.

\begin{prop} \label{nonaugmented-bar-construction-functoriality}
 The assignment of the curved DG\+coalgebra\/ $\Br^\cu_v(A^\bu)$ to
a DG\+algebra $A^\bu$ can be naturally extended to a functor from
the category of nonzero DG\+algebras to the category of
CDG\+coalgebras over~$k$,
$$
 \Br^\cu_v\:k\alg_\dg^+\lrarrow k\coalg_\cdg.
$$
\end{prop}

\begin{proof}
 In order to construct the functor $\Br^\cu_v$, one has to choose
for every nonzero DG\+algebra $A^\bu$ over $k$ a homogeneous
$k$\+linear retraction onto its unit $v_A\:A\rarrow k$.
 Such retractions can be chosen in an arbitrary way, and do not need
agree with each other in any sense.
 The resulting functor $\Br^\cu_v$, viewed up to a uniquely defined
isomorphism of functors, does not depend on the choice of
the retractions~$v_A$.

 Given a homomorphism of DG\+algebras $f\:A^\bu\rarrow B^\bu$, each
of them endowed with a homogeneous $k$\+linear retraction onto its unit
$v_A\:A\rarrow k$ and $v_B\:B\rarrow k$, one constructs the induced
morphism of CDG\+coalgebras $(g,l)\:\Br^\cu_{v_A}(A^\bu)\rarrow
\Br^\cu_{v_B}(B^\bu)$ in the way similar to the one in
Section~\ref{nonaugmented-bar-for-algebras-and-modules-subsecn}.
 The map~$f$ induces a homogeneous $k$\+linear map $f_+\:A_+
\rarrow B_+$ of degree~$0$, which in turn induces a graded coalgebra
homomorphism $g\:\Br(A)=\udT(A_+[1])\rarrow\udT(B_+[1])=\Br(B)$.
 The only nonzero bigrading component of the change-of-connection
linear function $l\:\Br^{-1}(A)\rarrow k$ is the linear function
$l\:A_+^0\rarrow k$ measuring the difference between the retractions
$v_A$ and $v_B\circ f\:A\rarrow k$.
\end{proof}

\subsection{Curved, noncoaugmented cobar-construction}
\label{curved-noncoaugmented-cobar-subsecn}
 This section is a generalization of
Section~\ref{coaugmented-dg-coalgebras-subsecn} to curved,
noncoaugmented DG\+coalgebras.
 The constructions below are dual to those of
Section~\ref{nonaugmented-bar-for-dg-algebras-subsecn}, up to a point.
 But then there are subtle differences between two dual pictures,
which we will explain.

 Let $C^\cu=(C,d_C,h_C)$ be a nonzero CDG\+coalgebra over a field~$k$.
 Then the counit $\epsilon\:C\rarrow k$ is a nonzero homogeneous
linear map of degree~$0$.
 Consider the graded vector subspace $C_+=\ker(\epsilon)\subset C$.

 Choose a homogeneous $k$\+linear map $w\:k\rarrow C$ of degree~$0$
such that the composition $k\overset w\rarrow C\overset\epsilon
\rarrow k$ is the identity map.
 So $w$~is a homogeneous $k$\+linear section of the counit
map~$\epsilon$.
 Put $W=\coker(w)$; then the composition $C_+\rarrow C\rarrow W$
is an isomorphism of graded $k$\+vector spaces and $C=k\oplus W$
as a graded $k$\+vector space.

 Both the comultiplication map $\mu\:C\rarrow C\ot_kC$ and
the differential $d_C\:C\rarrow C$ decompose into components according
to the direct sum decomposition $C=k\oplus W$.
 The projections of~$\mu$ onto $k\cdot1\ot_k C$ and $C\ot_k k\cdot1$
are prescribed by the counitality axiom, while the composition of~$d_C$
with the counit map vanishes; so we do not need to keep track of these.
 The composition of the map~$\mu$ with the projection $C\ot_k C\rarrow
W\ot_k W$ provides a $k$\+linear map $k\oplus W=C\rarrow W\ot_k W$,
whose components we denote by $\mu_W\:W\rarrow W\ot_k W$ and
$\mu_k\:k\rarrow W\ot_k W$.
 The composition of the map~$d_C$ with the projection $C\rarrow W$
provides a $k$\+linear map $k\oplus W=C\rarrow W$, whose components
we denote by $d_W\:W\rarrow W$ and $d_k\:k\rarrow W$.

 Denote by $\Cb(C)$ the free graded algebra $T(C_+[-1])$.
 One observes that, for any graded $k$\+vector space $V$, odd
derivations~$d$ of degree~$1$ on the tensor algebra $T(V)$ are
uniquely determined by their restructions to the subspace
(direct summand) of generators,  $V=V^{\ot1}\subset T(V)$.
 Moreover, any homogeneous linear map $V\rarrow T(V)$ of degree~$1$
corresponds to some odd derivation~$d$ on $T(V)$ in this way.

 We are going to construct a CDG\+algebra $\Cb^\cu_w(C^\cu)=
(\Cb(C),d_{\Cb},h_{\Cb})$, which can be called the curved,
noncoaugmented cobar-construction of a CDG\+coalgebra~$C^\cu$.
 Specifically, the odd derivation $d_{\Cb}\:\Cb(C)\rarrow\Cb(C)$ of
degree~$1$ is defined by the condition that its restriction to
the subspace of generators $C_+\subset\Cb(C)$ has three possibly
nonzero bigrading components, given by the maps $\mu_W\:C_+\rarrow
C_+\ot_k C_+$ and $d_W\:C_+\rarrow C_+$, \emph{and} also by
the curvature linear function $h_C\:C_+^{-2}=C^{-2}\rarrow k$.
 The curvature element $h_{\Cb}\in\Cb^2(C^\cu)$ has two possibly
nonzero bigrading components, namely, the elements $\mu_k(1)\in
(C_+\ot_k\nobreak C_+)^0$ and $d_k(1)\in C_+^1=C^1$.
 The sign rules can be found in~\cite[Section~6.1]{Pkoszul}.

\begin{rem}
 The first subtlety which deserves to be mentioned is that the previous
paragraph presumes that degree~$-2$ is different from degree~$0$.
 This is the case with the usual gradings by the group of integers,
which are generally presumed in this paper. 
 But sometimes one may be interested in $2$\+periodic complexes,
DG\+algebras, CDG\+coalgebras etc., graded by the group $\boZ/2\boZ$
(or by some other grading group, as in~\cite[Remark~1.1]{Pkoszul}
and~\cite[Section~1.1]{PP2}).
 When $2=0$ in the grading group~$\Gamma$, the equality $C_+^{-2}
=C^{-2}$ no longer holds, and the curvature linear function~$h_C$
also decomposes into two components, $h_W\:W\rarrow k$ and
$h_k\:k\rarrow k$.
 In this case, $h_W$~becomes a part of the differential~$d_{\Cb}$,
and $h_k$~becomes a third component of the curvature element~$h_{\Cb}$
in the cobar-construction (see~\cite[Section~6.1]{Pkoszul}).

 One can go further and consider the grading group $\Gamma=\{0\}$,
which makes sense over a ground field~$k$ of characteristic~$2$
(otherwise there are sign issues forcing the complexes to be at
least $\boZ/2\boZ$\+graded).
 In this setting, even change-of-connection elements/linear functions
can have nonzero unit/counit components (as $1=0$ in~$\Gamma$).
 We refer to~\cite[Section~6.1]{Pweak} for a discussion of the bar-
and cobar-constructions including a treatment of these unconventional
grading effects special to characteristic~$2$.
\end{rem}

 A \emph{coaugmentation}~$\gamma$ on a CDG\+coalgebra $(C,d,h)$
is a morphism of (graded, counital) coalgebras $\gamma\:k\rarrow C$
such that $(\gamma,0)\:(k,0,0)\rarrow(C,d,h)$ is a morphism of
CDG\+coalgebras.
 This definition includes the condition that the composition
$k\overset\gamma\rarrow C\overset d\rarrow C$ vanishes.
 (When $2=0$ in the grading group $\Gamma$, this definition also
includes the condition that the composition $k\overset\gamma\rarrow
C\overset h\rarrow k$ vanishes.)

 Given a coaugmented CDG\+coalgebra $(C^\cu,\gamma)$, one can take
the section $w=\gamma$ in the construction above and produce
a \emph{DG\+algebra} $\Cb^\bu_\gamma(C^\cu)=(\Cb(C),d_{\Cb})$.
 By the definition, the maps $\mu_k$ and~$d_k$ vanish in this case,
so $h_{\Cb}=0$.

 Otherwise, when no natural coaugmentation is available for $C^\cu$,
one has to choose an arbitrary section $w\:k\rarrow C$ to apply
the cobar-construction, and the issue of changing the section arises.
 Let $w'\:k\rarrow C$ be another homogeneous $k$\+linear map
of degree~$0$ such that $\epsilon\circ w'=\id_k$.
 Then the difference $b=w'(1)-w(1)$ is an element of the vector
subspace $C_+^0\subset C$.
 The vector space $C_+^0$ is one of the bigrading components of
the vector space $\Cb^1(C)$; so $b\in C_+^0\subset\Cb^1(C)$.
 The CDG\+algebras $\Cb^\cu_w(C^\cu)$ and $\Cb^\cu_{w'}(C^\cu)$ are
connected by a natural change-of-connection isomorphism
$(\id_{\Cb},b)\:\Cb^\cu_{w'}(C^\cu)\rarrow\Cb^\cu_w(C^\cu)$.

 Now suppose that we are given a change-of-connection isomorphism
of CDG\+coal\-gebras $(\id,a)\:(C,d,h)\rarrow(C,d',h')$, where
$a\:C^{-1}\rarrow k$ is a change-of-connection linear function
on~$C$.
 Let $w\:k\rarrow C$ be a homogeneous $k$\+linear section of the counit.
 Then the two cobar-constructions $\Cb^\cu_w(C,d,h)$ and
$\Cb^\cu_w(C,d',h')$ are naturally isomorphic as CDG\+algebras.

 To construct the latter isomorphism, one only needs to follow
the philosophy of nonaugmented Koszul duality, which tells that
change-of-connection isomorphisms on one side correspond to
changes of the chosen retraction of the unit or section of
the counit on the other side.
 The cobar-construction $\Cb^\cu_w(C,d',h')$ comes endowed with
a natural retraction onto the unit $v_{\Cb}$, namely, the direct
summand projection $\Cb(C)\rarrow C_+^{\ot0}=k$.
 (The map $v_{\Cb}$ is even an augmentation of the graded algebra
$\Cb(C)$, but it is \emph{not} an augmentation of the CDG\+algebra
$\Cb^\cu_w(C,d,h)$ unless $h_C=0$.)
 The desired isomorphism $\Cb^\cu_w(C,d,h)\simeq\Cb^\cu_w(C,d',h')$
is supposed \emph{not} to preserve these retractions onto the unit.

 Coming to the point, given two graded vector spaces $V$ and $W$,
any homogeneous $k$\+linear map $V\rarrow T(W)$ of degree~$0$ can be
uniquely extended to a graded $k$\+algebra homomorphism
$T(V)\rarrow T(W)$.
 Consider the graded $k$\+algebra automorphism $f_a\:\Cb(C)\rarrow
\Cb(C)$ defined by the rule $c\longmapsto c+a(c)$ for all
$c\in C_+=V=W$.
 Here $c+a(c)$ is an element of the graded vector space $V\oplus k
= V^{\ot1}\oplus V^{\ot0}\subset T(V)=\Cb(C)$.
 Then we have a natural isomorphism of CDG\+coalgebras
$(f_a,0)\:\Cb^\cu_w(C,d',h')\rarrow\Cb^\cu_w(C,d,h)$.

 Recall the notation $k\coalg_\cdg$ introduced in
Section~\ref{nonaugmented-bar-for-dg-algebras-subsecn} for
the category of CDG\+coalgebras over~$k$.
 Denote by $k\coalg_\cdg^+\subset k\coalg_\cdg$ the full subcategory
of nonzero CDG\+coalgebras, and by $k\coalg_\cdg^\coaug$
the category of coaugmented CDG\+coalgebras.
 Furthermore, let $k\alg_\cdg$ denote the category of CDG\+algebras
over~$k$.

\begin{prop} \label{curved-non-coaugmented-cobar-functoriality}
\textup{(a)} The assignment of the DG\+algebra $\Cb^\bu_\gamma(C^\cu)$
to a coaugmented CDG\+coalgebra $(C^\cu,\gamma)$ can be naturally
extended to a functor from the category of coaugmented CDG\+coalgebras
to the category of DG\+algebras over~$k$,
$$
 \Cb^\bu_\gamma\:k\coalg_\cdg^\coaug\lrarrow k\alg_\dg.
$$ \par
\textup{(b)} The assignment of the CDG\+algebra $\Cb^\cu_w(C^\cu)$
to a CDG\+coalgebra $C^\cu$ can be naturally extended to a functor
from the category of nonzero CDG\+coalgebras to the category of
CDG\+algebras over~$k$,
$$
 \Cb^\cu_w\:k\coalg_\cdg^+\lrarrow k\alg_\cdg.
$$
\end{prop}

\begin{proof}
 Part~(a): let $C^\cu=(C,d_C,h_C)$ and $D^\cu=(D,d_D,h_D)$ be two
CDG\+coalgebras endowed with coaugmentations $\gamma\:k\rarrow C$
and $\delta\:k\rarrow D$.
 Let $(f,a)\:(C,\gamma)\rarrow(D,\delta)$ be a morphism of coaugmented
CDG\+coalgebras, i.~e., a morphism of CDG\+coalgebras preserving
coaugmentations, in the sense that $f\circ\gamma=\delta$.
 (When $1=0$ in the grading group $\Gamma$, the condition that
$a\circ\gamma=0$ should be added.)

 The morphism of DG\+algebras $g\:\Cb^\bu_\gamma(C)\rarrow
\Cb^\bu_\delta(D)$ induced by the coaugmented CDG\+coalgebra
morphism~$(f,a)$ is constructed as follows.
 The map $f$ induces a homogeneous $k$\+linear map $f_+\:C_+\rarrow
D_+$ of degree~$0$.
 The graded $k$\+algebra homomorphism $g\:\Cb(C)=T(C_+[-1])
\rarrow T(D_+[-1])=\Cb(D)$ is defined by the rule
$c\longmapsto f_+(c)-a(c)$ for all $c\in C_+=V$.
 Here $f_+(c)-a(c)$ is an element of the graded vector space
$W\oplus k=W^{\ot1}\oplus W^{\ot0}\subset T(W)=\Cb(D)$,
where $W=D_+$.

 To construct the functor in part~(b), one has to choose for every
nonzero CDG\+coalgebra $C^\cu$ over~$k$ a homogeneous $k$\+linear
section of its counit $w_C\:k\rarrow C$.
 Such sections are chosen in an arbitrary way, and do not need to
agree with each other in any sense.
 The resulting functor $\Cb^\cu_w$, viewed up to a uniquely defined
isomorphism of functors, does not depend on the choice of
the sections~$w_C$.

 The construction of the morphism of CDG\+algebras $(g,b)\:
\Cb_{w_C}^\cu(C^\cu)\rarrow\Cb_{w_D}^\cu(D^\cu)$ assigned to a morphism
of CDG\+coalgebras $(f,a)\:C^\cu\rarrow D^\cu$ by the functor
$\Cb^\cu_w$ in part~(b) combines the construction from the proof of
part~(a) with the dual version of the construction from the proof of
Proposition~\ref{nonaugmented-bar-construction-functoriality}.
 We skip further details, which are rather straightforward:
the section/retraction and connection changes are taken care of as
described in the discussion preceding the proposition.
\end{proof}

\begin{rem} \label{no-isos-of-bar-induced-by-changes-of-connection}
 One conspicuous difference between the expositions in
Sections~\ref{nonaugmented-bar-for-dg-algebras-subsecn}
and~\ref{curved-noncoaugmented-cobar-subsecn} is that
the bar-construction was only applied to \emph{uncurved} DG\+algebras
$A^\bu=(A,d)$ in Section~\ref{nonaugmented-bar-for-dg-algebras-subsecn},
while the cobar-construction was defined for \emph{curved}
CDG\+coalgebras $C^\cu=(C,d,h)$ in
Section~\ref{curved-noncoaugmented-cobar-subsecn}.
 Let us explain the situation.

 Actually, one can define the bar-construction in the context of
curved DG\+algebras and assign a CDG\+coalgebra
$\Br^\cu_v(B^\cu)=(\Br(B),d_{\Br},h_{\Br})$ to any CDG\+algebra
$B^\cu=(B,d_B,h_B)$ with a chosen retraction onto the unit line
$v\:B\rarrow k$.
 When the CDG\+algebra $B^\cu$ is augmented and
$\beta\:B^\cu\rarrow k$ is the augmentation, the construction
produces a DG\+coalgebra $\Br^\bu_\beta(B^\cu)=(\Br(B),d_{\Br})$
(as $h_{\Br}=0$ in this case).

 The latter approach was suggested in the paper~\cite{Nic}.
 Subsequently it was discovered and reported in
the paper~\cite[Section~5.1]{KLN} that the approach does not work
as intended.

 In fact, there is a massive loss of information involved with
the passage from an (augmented or nonaugmented) CDG\+algebra
$B^\cu=(B,d_B,h_B)$ over a field~$k$ with a nonzero curvature
element $h_B\ne0$ to the bar-construction $\Br^\cu_v(B^\cu)$.
 In particular, \emph{all (C)DG\+comodules over the (C)DG\+coalgebra
$\Br^\cu_v(B^\cu)$ are contractible} when $h_B\ne0$.
 This vanishing phenomenon, originally observed by Kontsevich, was
recorded in the present author's memoir~\cite[Remark~7.3]{Pkoszul}.

 The Kontsevich vanishing can be avoided by working with so-called
\emph{weakly curved} DG\+algebras, which presumes a more complicated
setting with a \emph{topological local ring of coefficients} instead
of a ground field~$k$.
 Then the theory becomes well-behaved, as it was established
in the memoir~\cite{Pweak}, where the Koszul duality theory for weakly
curved DG\+algebras and weakly curved $\mathrm{A}_\infty$\+algebras
was developed.

 Returning to the construction of the CDG\+coalgebra $\Br^\cu_v(B^\cu)$
over a field~$k$, the reader can find some details of this
(well-defined, but not well-behaved) construction
in~\cite[Section~6.1]{Pkoszul}, with a further discussion
in~\cite[Remark~6.1]{Pkoszul}.
 One point is worth mentioning: even though one can construct
the CDG\+coalgebra $\Br^\cu_v(B^\cu)$ or even the DG\+coalgebra
$\Br^\bu_\beta(B^\cu)$, one \emph{cannot} assign
a (C)DG\+coalgebra isomorphism to a change-of-connection
isomorphism of CDG\+algebras $(\id_B,a)\:(B,d',h')\rarrow(B,d,h)$
with a change-of-connection element $0\ne a\in B^1$.

 The explanation is that the tensor coalgebra $\Br(B)=\udT(B_+)$ is
conilpotent, hence uniquely coaugmented; and while there are
coderivations of the tensor coalgebras $\udT(W)$ which do not
preserve the coaugmentations, there exist \emph{no} such endomorphisms
or homomorphisms of tensor coalgebras.
 A useful intuition may be provided by the construction of the dual
algebra $C^*$ to a coalgebra $C$, as mentioned in
Sections~\ref{local-finite-dimensionality-subsecn}
and~\ref{cdg-coalgebras-subsecn}.
 Suppose that $W$ is a finite-dimensional (ungraded) $k$\+vector
space, and consider the tensor coalgebra $C=\udT(W)$.
 Then the dual algebra $C^*$ is the algebra of noncommutative
\emph{formal Taylor power series} in a finite number of variables
over~$k$.
 On the other hand, $T(W)$ is the algebra of noncommutative
\emph{polynomials}.
 Now, variable changes like $z\longmapsto z+1$ are well-defined as
automorphisms of (commutative or noncommutative) polynomial rings,
but one cannot make such a variable change act on the ring of formal
power series in~$z$.
 Still, $\d/\d z_i$ are well-defined derivations of the formal power
series ring $k[[z_1,\dotsc,z_n]]$ (or its noncommutative version);
they just cannot be integrated/exponentiated to automorphisms.

 To sum up the discussion in this remark, one can say that
coderivations of the tensor coalgebras $\udT(W)$ which do not preserve
the coaugmentation may be well-defined, but they are \emph{not}
well-behaved in the context of differential graded homological algebra.
 This is one difference between the properties of the tensor algebras
and the tensor coalgebras; this is also an explanation of
the Kontsevich vanishing.

 The problem of nonfunctoriality of the bar-construction with respect
to change-of-connection morphisms of CDG\+algebras is resolved by
passing from the usual (conilpotent) to the ``extended''
(nonconilpotent) bar-construction of~\cite[Definition~2.5]{GL}.
 See the discussion in~\cite[Section~4]{GL}.
\end{rem}

\subsection{Duality between DG-algebras and CDG-coalgebras}
\label{duality-DG-algebras-curved-DG-coalgebras-subsecn}
 Let $C^\cu=(C,d,h)$ be a CDG\+coalgebra and $\gamma\:k\rarrow C^\cu$
be a coaugmentation of $C^\cu$ (as defined in
Section~\ref{curved-noncoaugmented-cobar-subsecn}).
 The CDG\+coalgebra $C^\cu$ is said to be \emph{conilpotent} if
the coaugmented graded coalgebra $(C,\gamma)$ is conilpotent.
 In other words, a CDG\+coalgebra $C^\cu$ is conilpotent if and only if
the graded coalgebra $C$ is conilpotent with the (unique)
coaugmentation~$\gamma$ \emph{and} $(\gamma,0)\:k\rarrow C^\cu$
is a morphism of CDG\+coalgebras.

 As above, we denote by $k\alg_\dg$ the category of DG\+algebras
over~$k$ and by $k\alg_\dg^+\subset k\alg_\dg$ the full subcategory
of nonzero DG\+algebras.
 Furthermore, $k\coalg_\cdg$ is the category of CDG\+coalgebras
over~$k$, and $k\coalg_\cdg^\coaug$ is the category of coaugmented
CDG\+coalgebras (with the CDG\+coalgebra morphisms $(f,a)\:C^\cu
\rarrow D^\cu$ forming a commutative triangle with the coaugmentations
$(\gamma,0)\:k\rarrow C^\cu$ and $(\delta,0)\:k\rarrow D^\cu$).
 Denote by $k\coalg_\cdg^\conilp\subset k\coalg_\cdg^\coaug$
the full subcategory of conilpotent CDG\+coalgebras.

 Somewhat similarly to the discussion in
Section~\ref{duality-dg-algebras-dg-coalgebras-subsecn},
the bar-construction $\Br^\cu_v(A^\bu)$ of any DG\+algebra $A^\bu$
is naturally a coaugmented CDG\+coalgebra (with the direct summand
inclusion $k=A_+^{\ot0}\rarrow\Br(A)$ providing the coaugmentation).
 Moreover, the coaugmented CDG\+coalgebra $\Br^\cu_v(A^\bu)$ is
conilpotent.

 So the bar-construction of
Proposition~\ref{nonaugmented-bar-construction-functoriality}
is actually a functor
$$
 \Br^\cu_v\:k\alg_\dg^+\lrarrow k\coalg_\cdg^\conilp.
$$
 The curved coaugmented cobar-construction of
Proposition~\ref{curved-non-coaugmented-cobar-functoriality}(a)
obviously produces nonzero DG\+algebras, so is a functor
$$
 \Cb^\bu_\gamma\:k\coalg_\cdg^\coaug\lrarrow k\alg_\dg^+.
$$

\begin{lem} \label{curved-nonaugmented-bar-cobar-adjunction}
 The restriction of the cobar-construction to conilpotent
CDG\+coalgebras,
$$
 \Cb^\bu_\gamma\:k\coalg_\cdg^\conilp\lrarrow k\alg_\dg^+,
$$
is naturally a left adjoint functor to the bar-construction\/
$\Br^\cu_v\:k\alg_\dg^+\rarrow k\coalg_\cdg^\conilp$.
\hbadness=1200
\end{lem}

 Lemma~\ref{curved-nonaugmented-bar-cobar-adjunction} is
a curved/nonaugmented version of
Lemma~\ref{augmented-bar-cobar-adjunction}.
 A very brief hint of proof will be given in
Section~\ref{twisting-cochains-curved-subsecn} below.

 One would like to define equivalence relations (that is, the classes
of morphisms to be inverted) in the categories of (nonzero)
DG\+algebras and conilpotent CDG\+coalgebras so that the adjoint
functors $\Br^\cu_v$ and $\Cb^\bu_\gamma$ become equivalences of
categories after these classes of morphisms are inverted.
 In the case of DG\+algebras, one can use the conventional
quasi-isomorphisms; but the notion of a quasi-isomorphism of
CDG\+coalgebras is \emph{undefined} because CDG\+coalgebras are
not complexes and do not have cohomology spaces.

 Similarly to Section~\ref{duality-dg-algebras-dg-coalgebras-subsecn},
the problem is resolved by considering \emph{filtered
quasi-isomorphisms}, which actually make sense for CDG\+coalgebras.
 An \emph{admissible filtration} on a coaugmented CDG\+coalgebra
$(C^\cu,\gamma)$ is defined as an exhaustive comultiplicative
increasing filtration by graded vector subspaces $F_nC^\cu\subset C^\cu$
such that $F_{-1}C^\cu=0$, \ $F_0C^\cu=\gamma(k)$, and
$d(F_nC^\cu)\subset F_nC^\cu$ for all $n\ge0$.

 In particular, any coaugmented CDG\+coalgebra having an admissible
filtration is conilpotent.
 Conversely, the canonical increasing filtration (as in
Section~\ref{conilpotent-coalgebras-subsecn}) on any conilpotent
CDG\+coalgebra is admissible.

 For any admissible filtration $F$ on a coaugmented CDG\+coalgebra
$(C^\cu,\gamma)$ one has $d^2(F_nC^\cu)\subset F_{n-1}C^\cu$ for
all $n\ge0$, since $d^2(c)=h*c-c*h$ for all $c\in C$ and
the curvature linear function $h\:C\rarrow k$ annihilates $F_0C^\cu$.
 Consequently, the induced differential~$d$ on the associated graded
vector space (in fact, the associated bigraded coalgebra)
$\gr^FC^\cu=\bigoplus_{n=0}^\infty F_nC^\cu/F_{n-1}C^\cu$ squares to
zero, so it makes $\gr^FC^\cu$ a complex, and in fact,
a (bigraded, uncurved) DG\+coalgebra.

 Let $(f,a)\:C^\cu\rarrow D^\cu$ be a morphism of conilpotent
CDG\+coalgebras.
 The morphism $(f,a)$ is said to be a \emph{filtered quasi-isomorphism}
if there exist admissible filtrations $F$ on both the CDG\+coalgebras
$C^\cu$ and $D^\cu$ such that $f(F_nC^\cu)\subset F_nD^\cu$ and
the induced map $F_nC^\cu/F_{n-1}C^\cu\rarrow F_nD^\cu/F_{n-1}D^\cu$
is a quasi-isomorphism of complexes of vector spaces for every $n\ge0$.
 Equivalently, this means that $\gr^Ff\:\gr^FC^\cu\rarrow\gr^FD^\cu$
is a quasi-isomorphism of (bigraded) DG\+coalgebras.

\begin{thm} \label{nonaugmented-algebras-curved-coalgebras-duality}
 Let\/ $\Quis$ be the class of all quasi-isomorphisms of (nonzero)
DG\+algebras and\/ $\FQuis$ be the class of all filtered
quasi-isomorphisms of conilpotent CDG\+coalgebras.
 Then the adjoint functors\/ $\Br^\cu_v$ and\/ $\Cb^\bu_\gamma$ induce
mutually inverse equivalences of categories
$$
 \Br^\cu_v\: k\alg_\dg^+[\Quis^{-1}]\,\simeq\,
 k\coalg_\cdg^\conilp[\FQuis^{-1}]\,:\!\Cb^\bu_\gamma.
$$
\end{thm}

\begin{proof}
 This is~\cite[Theorem~6.10(a)]{Pkoszul}.
\end{proof}

\begin{ex} \label{terminal-CDG-coalgebras-example}
 The following trivial example may be instructive.
 Notice that there are many nonzero DG\+algebras \emph{quasi-isomorphic}
to the zero DG\+algebra; they are all legitimate objects of
the category $k\alg_\dg^+$.
 One can easily see that all such DG\+algebras are isomorphic in
the category $k\alg_\dg^+[\Quis^{-1}]$.
 Typical representatives of this isomorphism class are the DG\+algebra
$A^\bu=(k[\varepsilon],d)$, with $\varepsilon\in A^{-1}$ and
$d(\varepsilon)=1$, and its quotient DG\+algebra $\overline A^\bu=
(k[\varepsilon]/(\varepsilon^2),d)$.
 For any (unital) DG\+algebra $X^\bu$ quasi-isomorphic to zero, there
exists a morphism of (unital) DG\+algebras $A^\bu\rarrow X^\bu$; so
$X^\bu$ is isomorphic to $A^\bu$ in $k\alg_\dg^+[\Quis^{-1}]$.

 One may wonder what the corresponding conilpotent CDG\+coalgebras are.
 Let $B^\cu=(k[h],0,h)$ be the CDG\+algebra over $k$ whose underlying
graded algebra is the algebra of polynomials in the curvature
element~$h$ (while the differential is $d=0$).
 As $B^\cu$ is a graded algebra with finite-dimensional components,
the graded dual vector space to $B^\cu$ has a natural structure of
CDG\+coalgebra, which we denote by $C^\cu$.
 So one has $C^n\simeq k$ whenever $n$~is an even nonpositive
integer, and $C^n=0$ otherwise; the curvature linear function
$h\:C^{-2}\rarrow k$ is an isomorphism, the differential on $C^\cu$
vanishes, and the comultiplication map $C^{-2n}\rarrow
C^{-2p}\ot_k C^{-2q}$ is an isomorphism for all $p+q=n$, \
$p$, $q\ge0$.
 One can easily see that $C^\cu=\Br^\cu_v(\overline A^\bu)$, where
$v\:\overline A\rarrow k$ is the unique homogeneous retraction onto
the unit line.
 (In particular, $C^\cu$ is obviously a conilpotent CDG\+coalgebra.)

 Furthermore, consider the CDG\+algebra $\overline B^\cu=
(k[h]/(h^2),0,h)$.
 The graded dual vector space $\overline C^\cu$ to $\overline B^\cu$
is a CDG\+coalgebra with $\overline C^n=k$ for $n=0$ and $n=-2$,
and $\overline C^n=0$ for all other~$n$.
 The curvature linear function $h\:\overline C^{-2}\rarrow k$ is
an isomorphism, while the differential on $\overline C^\cu$ vanishes.
 It is easy to see that $A^\bu=\Cb^\bu_\gamma(\overline C^\cu)$, where
$\gamma\:k\rarrow\overline C^\cu$ is the unique coaugmentation of
the conilpotent CDG\+coalgebra~$\overline C^\cu$.

 The CDG\+algebras $B^\cu$ and $\overline B^\cu$ were discussed
in~\cite[Section~2.2 and Proposition~3.2]{KLN} under the name of
``initial CDG\+algebras''.
 In fact, $B^\cu$ is the initial object of the category of
CDG\+algebras over~$k$ and \emph{strict} morphisms between them,
i.~e., morphisms $(f,0)$ with a vanishing change-of-connection
element $a=0$.
 By the same token, the CDG\+coalgebras $C^\cu$ and $\overline C^\cu$
might be called ``terminal CDG\+coalgebras'' (and indeed, the zero
DG\+algebra is the terminal object of the category $k\alg_\dg$).
\end{ex}

 Similarly to the discussion in
Section~\ref{quillen-equivalence-dg-algebras-dg-coalgebras-subsecn},
the equivalence of categories in
Theorem~\ref{nonaugmented-algebras-curved-coalgebras-duality} can be
expressed as a Quillen equivalence of model category structures.
 One additional complication as compared to
Section~\ref{quillen-equivalence-dg-algebras-dg-coalgebras-subsecn},
which arises here, is that the category $k\alg_\dg^+$ of nonzero
DG\+algebras has no terminal object (because the zero DG\+algebra, which
is the terminal object of $k\alg_\dg$, is excluded from $k\alg_\dg^+$).
 The category $k\coalg_\cdg^\conilp$ of conilpotent CDG\+coalgebras
has no terminal object, either (as only the category of conilpotent
CDG\+coalgebras and \emph{strict} morphisms between them has a terminal
object, mentioned in Example~\ref{terminal-CDG-coalgebras-example}
above).

 Still the definition of a model category (as in~\cite{Hov}) requires
existence of all limits and colimits.
 This problem is resolved by ``finalizing'' (formally adjoining terminal
objects to) both the categories $k\alg_\dg^+$ and
$k\coalg_\cdg^\conilp$ before defining model structures on them.
 We refer to~\cite[Section~9.2, Theorem~9.3(a), and end of
Section~9.3]{Pkoszul} for the details.

\subsection{Twisting cochains in the curved context}
\label{twisting-cochains-curved-subsecn}
 Let $V^\cu=(V,d_V)$ and $W^\cu=(W,d_W)$ be two precomplexes, i.~e.,
graded $k$\+vector spaces endowed with homogeneous $k$\+linear 
differentials $d_V\:V\rarrow V$ and $d_W\:W\rarrow W$ of degree~$1$
with possibly nonzero squares (as in
Section~\ref{cdg-coalgebras-subsecn}).
 Then the precomplex $\Hom^\cu_k(V^\cu,W^\cu)$ is defined as
the graded Hom space $\Hom_k(V,W)$ endowed with the differential
$d\:\Hom_k(V,W)\rarrow\Hom_k(V,W)$ of degree~$1$ given by the usual
formula $d(f)(x)=d_W(f(x))-(-1)^{|f|}f(d_V(x))$ for all
$f\in\Hom_k^{|f|}(V,W)$ and $x\in V^{|x|}$.
{\emergencystretch=0.5em\par}

 Let $C^\cu=(C,d_C,h_C)$ be a CDG\+coalgebra and $B^\cu=(B,d_B,h_B)$
be a CDG\+algebra over~$k$.
 Consider the graded Hom space $\Hom_k(C,B)$, and endow it with
a graded $k$\+algebra structure as constructed in
Section~\ref{hom-dg-algebra-and-twisting-cochains-subsecn}
and a differential~$d$ given by the usual rule as above.
 Furthermore, let $h\in\Hom^2_k(C,B)$ be the element given by
the formula $h(c)=\epsilon(c)h_B-h_C(c)\cdot1$ for all $c\in C$,
where $1\in B$ is the unit element and $\epsilon\:C\rarrow k$ is
the counit map.
 Then $\Hom^\cu_k(C^\cu,B^\cu)=(\Hom_k(C,B),d,h)$ is a CDG\+algebra
over~$k$ \,\cite[Section~6.2]{Pkoszul}.

 Let $E^\cu=(E,d,h)$ be a CDG\+ring over~$k$.
 Then an element $a\in E^1$ is said to be a \emph{Maurer--Cartan
element} if it satisfies the equation $a^2+d(a)+h=0$ in~$E^2$.
 Equivalently, $a\in E^1$ is a Maurer--Cartan element if and only
if there exists an odd derivation $d'\:E\rarrow E$ of degree~$1$
(specifically, $d'(e)=d(e)+[a,e]$ for all $e\in E$) such that
$(\id_E,a)\:(E,d',0)\rarrow(E,d,h)$ is an isomorphism of CDG\+rings.
 So Maurer--Cartan elements in a CDG\+ring parametrize its
change-of-connection isomorphisms with DG\+rings.

 By the definition, a \emph{twisting cochain}~$\tau$ for
a CDG\+coalgebra $C^\cu$ and a CDG\+algebra $B^\cu$ over~$k$ is
a Maurer--Cartan element in the CDG\+algebra $\Hom_k^\cu(C^\cu,B^\cu)$.
 So $\tau\:C^\cu\rarrow B^\cu$ is a homogeneous $k$\+linear map of
degree~$1$ satisfying the Maurer--Cartan equation.
 (There is a slight abuse of terminology involved here, as
by ``cochains'' one usually means elements in grading components
of a complex, while $\Hom_k^\cu(C^\cu,B^\cu)$ is not a complex.)

 This definition of a twisting cochain is a generalization of the one
in Section~\ref{hom-dg-algebra-and-twisting-cochains-subsecn}.
 Similarly to the discussion in
Section~\ref{hom-dg-algebra-and-twisting-cochains-subsecn},
various conditions of compatibility with (co)augmentations may be
imposed on twisting cochains.
 We will be particularly interested in the context where
$(C^\cu,\gamma)$ is a coaugmented (better yet, conilpotent)
CDG\+coalgebra, while $A^\bu$ is a DG\+algebra.
 In this case, twisting cochains $\tau\:C^\cu\rarrow A^\bu$ such that
$\tau\circ\gamma=0$ play an important role.

\begin{exs} \label{nonaugmented-bar-twisting-cochains-examples}
 (1)~This is a nonaugmented version of
Example~\ref{bar-construction-twisting-cochains-examples}(1).
 Let $A\ne0$ be an algebra over~$k$ and $C^\cu=\Br^\cu_v(A)$ be its
bar-construction, as in
Section~\ref{nonaugmented-bar-for-algebras-and-modules-subsecn}.
 Let $\gamma\:k\rarrow C^\cu$ be the natural coaugmentation of
$\Br^\cu_v(A)$ (as in
Section~\ref{duality-DG-algebras-curved-DG-coalgebras-subsecn}).
 Then the composition~$\tau$ of the direct summand projection
$\Br(A)\rarrow A_+^{\ot 1}=A_+$ and the inclusion $A_+\simeq
\ker(v)\rarrow A$ is a twisting cochain for the CDG\+coalgebra $C^\cu$
and the algebra $A$ (viewed as a (C)DG\+algebra in the obvious way).
 The equation of compatibility with the coaugmentation
$\tau\circ\gamma=0$ is satisfied in this example.

\smallskip
 (2)~This is a nonaugmented version of
Example~\ref{bar-construction-twisting-cochains-examples}(2).
 Let $A^\bu$ be a nonzero DG\+algebra and $C^\cu=\Br^\cu_v(A^\bu)$
be its bar-construction, as in
Section~\ref{nonaugmented-bar-for-dg-algebras-subsecn}.
 Let $\gamma\:k\rarrow C^\cu$ denote the natural coaugmentation of
$\Br^\cu_v(A^\bu)$.
 Then the composition $\tau\:C^\cu\rarrow A^\bu$ of the direct summand
projection $\Br(A)\rarrow A_{+}^{\ot1}=A_+$ and
the inclusion $A_+\simeq\ker(v)\rarrow A$ is a twisting cochain for
the CDG\+coalgebra $C^\cu$ and the DG\+algebra~$A^\bu$.
 The equation $\tau\circ\gamma=0$ is satisfied for this twisting
cochain.
\end{exs}

\begin{ex} \label{curved-cobar-construction-twisting-cochain-example}
 This is a curved version of
Example~\ref{cobar-construction-twisting-cochains-examples}(2).
 Let $(C^\cu,\gamma)$ be a coaugmented CDG\+coalgebra, and let
$A^\bu=\Cb^\bu_\gamma(C^\cu)$ be its cobar-construction, as in
Section~\ref{curved-noncoaugmented-cobar-subsecn}.
 Then the composition $\tau\:C^\cu\rarrow A^\bu$ of the natural
surjection $C^\cu\rarrow C^{\cu,+}=C^\cu/\gamma(k)$ and
the direct summand inclusion $C^{\cu,+}\rarrow\Cb_\gamma^\bu(C^\cu)$
is a twisting cochain for the CDG\+coalgebra $C^\cu$ and
the DG\+algebra~$A^\bu$.
 The equation $\tau\circ\gamma=0$ is satisfied for this twisting
cochain.
\end{ex}

\begin{ex} \label{curved-noncoaugmented-cobar-twisting-cochain-example}
 This is a noncoaugmented version of the previous example.
 Let $C^\cu\ne0$ be a CDG\+coalgebra, and let $B^\cu=\Cb^\cu_w(C^\cu)$
be its (noncoaugmented) cobar-construction, as in
Section~\ref{curved-noncoaugmented-cobar-subsecn}.
 Then the composition $\tau\:C^\cu\rarrow B^\cu$ of the surjection
$C\rarrow\coker(w)\simeq C_+$ and the direct summand inclusion
$C_+=C_+^{\ot1}\rarrow\Cb(C)$ is a twisting cochain for
the CDG\+coalgebra $C^\cu$ and the CDG\+algebra~$B^\cu$.
\end{ex}

\begin{proof}[Hint of proof of
Lemma~\ref{curved-nonaugmented-bar-cobar-adjunction}]
 Let $A^\bu$ be a nonzero DG\+algebra and $C^\cu$ be a conilpotent
CDG\+coalgebra over~$k$.
 Then both the DG\+algebra homomorphisms $\Cb^\bu_\gamma(C^\cu)
\allowbreak\rarrow A^\bu$ and the (coaugmented) CDG\+coalgebra
morphisms $C^\cu\rarrow\Br^\cu_v(A^\bu)$ correspond bijectively to
twisting cochains $\tau\:C^\cu\rarrow A^\bu$ satisfying the equation
of compatibility with the coaugmentation $\tau\circ\gamma=0$.
\emergencystretch=0em
\end{proof}

 Let $A^\bu$ be a nonzero DG\+algebra and $(C^\cu,\gamma)$ be
a conilpotent CDG\+coalgebra.
 Let $\tau\:C^\cu\rarrow A^\bu$ be a twisting cochain satisfying
the equation $\tau\circ\gamma=0$.
 Generalizing the definition from
Section~\ref{bar-cobar-and-acyclic-twisting-cochains-subsecn},
the twisting cochain~$\tau$ is said to be \emph{acyclic} if one of
(or equivalently, both) the related DG\+algebra homomorphism
$\Cb^\bu_\gamma(C^\cu)\rarrow A^\bu$ and the CDG\+coalgebra
morphism $C^\cu\rarrow\Br^\cu_v(A^\bu)$ become isomorphisms after
the quasi-isomorphisms of DG\+algebras, or respectively
the filtered quasi-isomorphisms of conilpotent CDG\+coalgebras,
are inverted, as in
Theorem~\ref{nonaugmented-algebras-curved-coalgebras-duality}.
 Simply put, the twisting cochain~$\tau$ is called acyclic if
the related homomorphism of DG\+algebras $\Cb^\bu_\gamma(C^\cu)
\rarrow A^\bu$ is a quasi-isomorphism.

 For example, the twisting cochain from
Example~\ref{curved-cobar-construction-twisting-cochain-example}
is acyclic, by the definition, whenever the coaugmented
CDG\+coalgebra $(C^\cu,\gamma)$ is conilpotent.
 It is a part of 
Theorem~\ref{nonaugmented-algebras-curved-coalgebras-duality}
that the twisting cochain from
Examples~\ref{nonaugmented-bar-twisting-cochains-examples} is
acyclic for any nonzero algebra $A$ or DG\+algebra~$A^\bu$.

\begin{ex} \label{curved-nonhomog-quadratic-twisting-cochain}
 The following example is a nonaugmented version of
Section~\ref{nonhomogeneos-quadratic-dual-to-augmented-subsecn}
and Example~\ref{nonhomogeneous-quadratic-twisting-cochain}.

 Let $(A,F)$ be a nonhomogeneous quadratic algebra, as defined
in Section~\ref{nonhomogeneos-quadratic-dual-to-augmented-subsecn},
and let $C=(\gr^FA)^?$ be the quadratic dual coalgebra to
the quadratic algebra $\gr^FA$.
 Let $v\:A\rarrow k$ be a $k$\+linear retraction onto the unit line.
 Then there is a natural isomorphism of vector spaces $\gr^F_1A
=F_1A/F_0A\simeq\ker(v)\cap F_1A\subset A$.
 So the vector space $V=\gr^F_1A$ can be viewed as a subspace in
$A_+=A/k\cdot1\simeq\ker(v)$, and the graded coalgebra $C$ can be
viewed as a subcoalgebra in $\udT(A_+)$,
$$
 C\subset\udT(V)\subset\udT(A_+)=\Br(A).
$$

 Similarly to
Section~\ref{nonhomogeneos-quadratic-dual-to-augmented-subsecn},
one observes that the subcoalgebra $C\subset\Br(A)$ is in fact
a CDG\+subcoalgebra in the CDG\+coalgebra $\Br^\cu_v(A)$ constructed
in Section~\ref{nonaugmented-bar-for-algebras-and-modules-subsecn},
\,$C^\cu\subset\Br^\cu_v(A)$.
 Quite simply, this means that the differential~$\d$ on
$\Br^\cu_v(A)$ preserves $C$, that is $\d(C)\subset C$.
 The curvature linear function $h\:C\rarrow k$ is produced as
the restriction of the curvature linear function $h_{\Br}\:\Br(A)
\rarrow k$ to the subcoalgebra $C\subset\Br(A)$.
 This is another way to spell out the nonhomogeneous quadratic
duality construction of~\cite[Proposition~2.2]{Pcurv}
or~\cite[Proposition~4.1 in Chapter~5]{PP}.

 We have constructed the CDG\+coalgebra $C^\cu=(C,\d,h)$ nonhomogeneous
quadratic dual to a nonhomogeneous quadratic algebra~$(A,F)$.
 Now the composition $C^\cu\rarrow\Br^\cu_v(A)\rarrow A$ of
the inclusion map $C^\cu\rarrow\Br^\cu_v(A)$ with the twisting
cochain $\Br^\cu_v(A)\rarrow A$ from
Example~\ref{nonaugmented-bar-twisting-cochains-examples}(1) is
a twisting cochain $\tau\:C^\cu\rarrow A$.
 Equivalently, the twisting cochain~$\tau$ can be constructed as
the composition of the direct summand projection $C\rarrow C^1=
F_1A/F_0A$ and the inclusion $F_1A/F_0A\simeq\ker(v)\cap F_1A
\rarrow A$.

 The direct summand inclusion $\gamma\:k=C_0\rarrow C$ is
a coaugmentation of the CDG\+coalgebra $C^\cu$, so $(C^\cu,\gamma)$
is a coaugmented, and in fact, conilpotent CDG\+coal\-gebra.
 The twisting cochain $\tau\:C^\cu\rarrow A$ satisfies the equation
$\tau\circ\gamma=0$.

 For any nonhomogeneous \emph{Koszul} algebra $(A,F)$, as defined
in Section~\ref{nonhomogeneos-quadratic-dual-to-augmented-subsecn}
(i.~e., a nonhomogeneous quadratic algebra for which the quadratic
graded algebra $\gr^FA$ is Koszul), the twisting cochain
$\tau\:C^\cu\rarrow A$ is acyclic.
 As in Section~\ref{bar-cobar-and-acyclic-twisting-cochains-subsecn},
this claim is a corollary of the proof of Poincar\'e--Birkhoff--Witt
theorem for nonhomogeneous Koszul algebras
in~\cite[Section~3.2\+-3.3]{Pcurv} or~\cite[Proposition~7.2(ii)
in Chapter~5]{PP}.

 A generalization of these constructions and the acyclicity claim
above to filtered \emph{DG\+algebras} $(A^\bu,F)$ can be found
in~\cite[Section~6.6]{Pkoszul}.
\end{ex}

\subsection{Derived Koszul duality on the comodule side}
\label{koszul-duality-comodule-side-subsecn}
 Let $B^\cu$ be a CDG\+algebra and $C^\cu$ be a CDG\+coalgebra
over~$k$, and let $\tau\:C^\cu\rarrow B^\cu$ be a twisting cochain.

 Given a left CDG\+comodule $N^\cu$ over $C^\cu$, we consider
the tensor product of graded vector spaces $B\ot_kN$ and endow
it with a differential~$d$ twisted with the twisting cochain~$\tau$
using the same formulas as in
Sections~\ref{twisted-differential-on-tensor-product-subsecn}\+-%
\ref{augmented-duality-comodule-side-subsecn}.
 Then $B^\cu\ot^\tau N^\cu=(B\ot_kN,\>d)$ is a left CDG\+module
over~$B^\cu$.

 Similarly, given a left CDG\+module $M^\cu$ over $B^\cu$, we consider
the tensor product of graded vector spaces $C\ot_kM$ and endow
it with a differential~$d$ twisted with the twisting cochain~$\tau$
using the same formulas as in
Sections~\ref{twisted-differential-on-tensor-product-subsecn}\+-%
\ref{augmented-duality-comodule-side-subsecn}.
 Then $C^\cu\ot^\tau M^\cu=(C\ot_kM,\>d)$ is a left CDG\+comodule
over~$C^\cu$.

 For example, let $A$ be a nonzero associative algebra and
$C^\cu=\Br^\cu_v(A)$ be its bar-construction, as in
Section~\ref{nonaugmented-bar-for-algebras-and-modules-subsecn}.
 Let $M$ be a left $A$\+module.
 Then the CDG\+module $\Br^\cu_v(A,M)$ over $\Br^\cu_v(A)$ constructed
in Section~\ref{nonaugmented-bar-for-algebras-and-modules-subsecn}
can be recovered as $\Br^\cu_v(A,M)=\Br^\cu_v(A)\ot^\tau M$, where
$\tau\:C^\cu\rarrow A$ is the twisting cochain from
Example~\ref{nonaugmented-bar-twisting-cochains-examples}(1).

 Recall the notation $B^\cu\modl$ for the DG\+category of left
CDG\+modules over $B^\cu$ and $C^\cu\comodl$ for the DG\+category
of left CDG\+comodules over $C^\cu$ (see
Sections~\ref{cdg-rings-subsecn}\+-\ref{cdg-coalgebras-subsecn}).
 Similarly to Section~\ref{augmented-duality-comodule-side-subsecn},
one observes that the DG\+functor
$$
 B^\cu\ot^\tau{-}\,:C^\cu\comodl\lrarrow B^\cu\modl
$$
is left adjoint to the DG\+functor
$$
 C^\cu\ot^\tau{-}\,:B^\cu\modl\lrarrow C^\cu\comodl.
$$

\begin{thm} \label{nonaugmented-acyclic-twisting-cochain-duality-thm}
 Let $A^\bu$ be a nonzero DG\+algebra and $(C^\cu,\gamma)$ be
a conilpotent CDG\+coalgebra over~$k$ (as defined in
Section~\ref{duality-DG-algebras-curved-DG-coalgebras-subsecn}).
 Let $\tau\:C^\cu\rarrow A^\bu$ be an acyclic twisting cochain
(as defined in Section~\ref{twisting-cochains-curved-subsecn});
this includes the condition that $\tau\circ\gamma=0$.
 Then the adjoint functors $M^\bu\longmapsto C^\cu\ot^\tau M^\bu$
and $N^\cu\longmapsto A^\bu\ot^\tau N^\cu$ induce a triangulated
equivalence between the conventional \emph{derived} category of
left DG\+modules over $A^\bu$ and the \emph{coderived} category of
left CDG\+comodules over~$C^\cu$,
$$
 \sD(A^\bu\modl)\simeq\sD^\co(C^\cu\comodl).
$$
\end{thm}

\begin{proof}
 This is~\cite[Theorem~6.5(a)]{Pkoszul}.
 The particular case corresponding to the twisting cochain from
Examples~\ref{nonaugmented-bar-twisting-cochains-examples} can be
found in~\cite[Theorem~6.3(a)]{Pkoszul}, while
the case of the twisting cochain from
Example~\ref{curved-cobar-construction-twisting-cochain-example}
(for a conilpotent CDG\+coalgebra~$C^\cu$)
is considered in~\cite[Theorem~6.4(a)]{Pkoszul}.
 For the definition of the coderived category,
see Section~\ref{coderived-cdg-comodules-subsecn} below.
\end{proof}

 The triangulated equivalence of
Theorem~\ref{nonaugmented-acyclic-twisting-cochain-duality-thm} takes
the free left DG\+module $A^\bu$ over $A^\bu$ to the left
CDG\+comodule~$k$ over $C^\cu$ (with the $C$\+comodule structure
on~$k$ defined in terms of the coaugmentation~$\gamma$).
 Notice that, \emph{unlike} in
Theorems~\ref{augmented-acyclic-twisting-cochain-duality-thm}\+-%
\ref{augmented-nonconilpotent-duality-thm}, there is no natural
structure of a DG\+module over $A^\bu$ on the one-dimensional
vector space~$k$, as the DG\+algebra $A^\bu$ is not augmented.
 This fact is not unrelated to the fact that there is no natural
structure of a left CDG\+comodule over $C^\cu$ on the cofree graded
left $C$\+comodule $C$ (for the reason explained at the end of
Section~\ref{cdg-coalgebras-subsecn}).

\begin{thm} \label{nonaugmented-nonconilpotent-duality-thm}
 Let $C^\cu$ be a nonzero CDG\+coalgebra, and let $\tau\:C^\cu\rarrow
\Cb^\cu_w(C^\cu)=B^\cu$ be the twisting cochain from
Example~\ref{curved-noncoaugmented-cobar-twisting-cochain-example}.
 Then the adjoint functors $M^\cu\longmapsto C^\cu\ot^\tau M^\cu$
and $N^\cu\longmapsto B^\cu\ot^\tau N^\cu$ induce a triangulated
equivalence between the \emph{absolute derived} category of left
CDG\+modules over $B^\cu$ and the \emph{coderived} category of
left CDG\+comodules over~$C^\cu$,
$$
 \sD^\abs(B^\cu\modl)\simeq\sD^\co(C^\cu\comodl).
$$
\end{thm}

\begin{proof}
 This is~\cite[Theorem~6.7(a)]{Pkoszul}.
 The definitions of the coderived and absolute derived categories will
be explained in Sections~\ref{co-contra-derived-cdg-modules-subsecn}\+-%
\ref{coderived-cdg-comodules-subsecn} below.
 The absolute derived category $\sD^\abs(B^\cu\modl)$ coincides with
the coderived category $\sD^\co(B^\cu\modl)$
by~\cite[Theorem~3.6(a)]{Pkoszul}; see
Theorem~\ref{fin-homol-dim-derived-cdg-modules}(a) below.
\end{proof}

 For extended versions of
Theorems~\ref{nonaugmented-acyclic-twisting-cochain-duality-thm}
and~\ref{nonaugmented-nonconilpotent-duality-thm}
(``Koszul triality''), see Section~\ref{koszul-triality-subsecn}.
 Some additional comments on the proofs of the theorems will be
offered in Section~\ref{comments-on-proof-subsecn}.

 A version of nonconilpotent Koszul duality
(Theorem~\ref{nonaugmented-nonconilpotent-duality-thm}) with
the roles of the algebras and coalgebras switched (starting from
a DG\+algebra $A^\bu$ and considering the cofree nonconilpotent
coalgebra spanned by~$A_+$, or rather, the vector space dual
topological algebra) was suggested in the paper~\cite{GL}.
 For a version of derived nonhomogeneous Koszul duality for
\emph{DG\+categories}, see~\cite{HL}.

\Section{Derived Categories of the Second Kind}

 ``Derived categories of the second kind'' is a common name for
the \emph{coderived}, \emph{contraderived}, and \emph{absolute derived}
categories, while the conventional \emph{derived} category is called
the ``derived category of the first kind''.
 The \emph{classical homological algebra} can be described as
the domain where there is \emph{no difference} between derived
categories of the first and the second kind.
 Derived categories of the second kind play a crucial role in all
formulations of derived Koszul duality outside of the contexts
in which assumptions of boundedness of the (internal or cohomological)
grading ensure applicability of the classical homological algebra.

\subsection{Importance of derived categories of the second kind}
 The earliest versions of derived Koszul duality, establishing
the very existence of the phenomenon, were formulated for complexes
(DG\+algebras, etc.)\ up to conventional quasi-isomorphism.
 This includes Quillen's equivalence between the categories of
negatively cohomologically graded Lie DG\+algebras and
connected, simply connected, negatively cohomologically graded
cocommutative DG\+coalgebras over a field of characteristic
zero~\cite[Theorem~I]{Quil}.
 In this result, the boundedness conditions on the cohomological
grading allow to avoid the problem demonstrated in
Example~\ref{bar-of-algebras-quasi-isomorphism} above
(cf.\ our Theorems~\ref{augmented-algebras-coalgebras-duality}
and~\ref{nonaugmented-algebras-curved-coalgebras-duality}, where
the cohomological gradings are completely unbounded, but
the notion of a \emph{filtered quasi-isomorphism} is used).
 See, e.~g., \cite[Section~3]{Pqf} for noncommutative versions of
Quillen's theory.

 The most important area of derived Koszul duality results for
which the conventional quasi-isomorphisms are sufficient is
the \emph{homogeneous} Koszul duality.
 This means the study of complexes of modules, DG\+modules etc.\
endowed with an additional \emph{internal} grading which is
assumed to be positive or negative on the rings and bounded below
or bounded above on the modules.
 In some classical contexts, the modules are assumed to be finitely
generated, which serves to ensure that their internal gradings
are bounded.
 This theory goes back to the paper~\cite{BGG} (see
also~\cite[Appendix~A]{OSS}), where the \emph{derived
$S$\+$\Lambda$ duality} (the equivalence between the bounded
derived categories of finitely generated graded modules over
the symmetric algebra of a finite-dimensional vector space and
the exterior algebra of the dual vector space) was introduced.
 A standard reasonably modern source on derived homogeneous
Koszul duality is~\cite[Section~2.12]{BGS}; another exposition
can be found in~\cite[Appendix~A]{Pkoszul}.

 To give a simple explanation of the role of positive or bounded
internal grading, notice that $\Tor^A_0(k,M)\ne0$ for any bounded
below graded module $M\ne0$ over a positively graded algebra
$A=k\oplus A_1\oplus A_2\oplus\dotsb$.
 This serves to exclude counterexamples such as in
Examples~\ref{bar-of-modules-acyclic}.
 Say, one can take $A=k[x]$ to be the polynomial (symmetric) algebra in
one variable~$x$, internally graded so that the generator~$x$ sits in
the internal grading~$1$, and endowed with the augmentation
$\alpha\:A\rarrow k$ induced by the internal grading (so $\alpha(x)=0$).
 Then all the counterexamples of modules $M$ in
Example~\ref{bar-of-modules-acyclic}(2) are either ungraded, or
their grading is unbounded.

 When the conventional quasi-isomorphism is \emph{not} suitable as
the equivalence relation on one of the sides (which is usually, but
not always, the coalgebra or comodule/contramodule side) of Koszul
duality, another equivalence relation on the respective side has to be
defined instead.
 When the equivalence relation on the other (usually, the algebra or
module side) of the duality is still the conventional quasi-isomorphism,
a possibility of quick hack presents itself: call a morphism on
the coalgebra side a ``weak equivalence'' if the Koszul duality functor
transforms it into a quasi-isomorphism on the algebra side.
 In this approach, the Koszul duality becomes an equivalence of
two categories one of which is constructed partially in terms of
the other one.
 Such formulations of Koszul duality can be found in
the paper~\cite[Theorem~3.2]{Hin2},
the dissertation~\cite[Th\'eor\`eme~2.2.2.2]{Lef}, and
the note~\cite[Section~4]{Kel}, as well as in
the preprint~\cite[Section~7.2]{BD2}.

 The latter reference concerns the \emph{$D$\+$\Omega$ duality},
that is, Koszul duality between the rings of differential operators
and the de~Rham (C)DG\+rings of differential forms (which is 
a thematic example of \emph{relative} nonhomogeneous Koszul
duality~\cite{Prel}.)
 For another early approach to the $D$\+$\Omega$ duality,
see~\cite{Kap}.
 Such approaches were tried in the absence of understanding of
the existence of and the role of the constructions of
the \emph{derived categories of the second kind}.

\subsection{History of derived categories of the second kind, I}
\label{second-kind-reminiscences-subsecn}
 Let me start with some personal reminiscences.
 A shorter historical account of the same events can be found
in~\cite[Remark~9.2]{PS4}.

 The problem of constructing derived nonhomogeneous Koszul duality
(for nonhomogeneous Koszul algebras, as in 
Section~\ref{nonhomogeneos-quadratic-dual-to-augmented-subsecn})
was suggested to me by Misha Finkelberg sometime around~1991--92,
in a handwritten letter sent to Moscow (where I~lived) from
Massachusetts (where he was doing his Ph.~D. studies at Harvard).
 Alongside with A.~Vaintrob, Misha was my main informal de~facto
teacher and advisor in the second half of '80s and early '90s.

 Misha wrote that the problem of derived nonhomogeneous Koszul duality
(i.~e., constructing a derived equivalence between complexes of
modules over a nonhomogeneous Koszul algebra and DG\+modules over
the nonhomogeneous quadratic dual DG\+algebra) would be a natural
extension of my work on nonhomogeneous quadratic
duality~\cite{Pcurv,PP}.
 (The paper~\cite{Pcurv} was only submitted in November~1992, but very
early short versions of what became the book~\cite{PP} circulated
since 1991, if not 1990.)
 The case of a finite-dimensional Lie algebra and its
Chevalley--Eilenberg DG\+algebra (as in
Example~\ref{chevalley-eilenberg}) was perceived as a thematic
example, of course; and the idea to proceed further to
$D$\+$\Omega$ duality was suggested in Misha's letter.

 Soon, that is also around~1992, I~made an intensive attempt to attack
and conquer the problem, which resulted in a tentative understanding
that the derived category $\sD(\g\modl)$ was equivalent to the homotopy
category of DG\+modules over $\Lambda^\bu(\g^*)$ with injective (or
equivalently, projective) underlying graded $\Lambda(\g^*)$\+modules.
 From the contemporary point of view (say, after Becker's
paper~\cite{Bec}), that was about the best answer one could hope for.
 Back in the '90s, I~did not see it that way.

 Following my thinking of the early '90s, derived nonhomogeneous Koszul
duality was supposed to be an equivalence between \emph{Verdier
quotient categories} of the homotopy categories of DG\+modules, not
their \emph{subcategories}.
 I~wanted a \emph{derived} Koszul duality, not a \emph{homotopy} one!
(in the sense of the homotopy categories of complexes or DG\+modules,
that is, closed morphisms up to cochain homotopy).
 Specifically, whatever the construction of a category of DG\+modules
over $\Lambda^\bu(\g^*)$ equivalent to $\sD(\g\modl$) would be,
there was supposed to be a way to assign an object of this category
to an arbitrary DG\+module over $\Lambda^\bu(\g^*)$.
 I~did not see a natural way to assign a graded-injective DG\+module
to an arbitrary one.
 Once again, Becker's~\cite[Proposition~1.3.6]{Bec} did not exist
back then.

 By late 90's, I~perceived derived nonhomogeneous Koszul duality as
an important unsolved, perhaps unsolvable, problem obstructing
development of the theory of DG\+modules generally.
 In the meantime, I~moved from Moscow to the U.S., first as a visitor,
then as a graduate student, and finally as a postdoc.
 My access to the literature improved, and I~was able to observe that
the theory was developing nevertheless.
 On top of Spaltenstein's paper~\cite{Spal}, which was known to me
back in early '90s, there appeared an important paper of
Keller~\cite{Kel0}.
 I~had also noticed several papers by Neeman (e.~g.,~\cite{BN,Neem1}),
from which I~learned the concepts of homotopy (co)limits and
the telescope construction in triangulated categories, as well
as compactly generated triangulated categories and other things.

 In the 1998--99 academic year I~was an NSF postdoc at the Institute
for Advanced Study.
 My postdoc advisors there were P.~Deligne and V.~Voevodsky.
 In March~99 my term as a postdoc was coming to an end, and Voevodsky
suggested that we should meet for a discussion.
 At the meeting, Voevodsky returned to the topic of DG\+algebra and
DG\+module theory (which he occasionally mentioned to me during our
conversations over the years, starting from~1994).
 His idea was to develop a DG version of Tannaka theory for
conservative functors from an (enhanced) triangulated category to
the (derived) category of vector spaces, with applications to his
derived categories of motives and motivic realization functors
in mind.

 Discussing the problems involved, we arrived to the topic of
a Koszul duality between ungraded modules over the symmetric and
exterior algebras; and I~mentioned that it was an unsolved problem,
as the na\"\i ve approach with conventional derived categories
on both sides of the would-be equivalence did not work.
 Upon hearing my arguments, Voevodsky immediately replied that this
problem can be solved, and must be solved.
 I~went home and solved it.
 It was immediately obvious to me that an important discovery was made.
 By mid-April~99, I~had both the definition of what is now called
the coderived category
(Definitions~\ref{coderived-cdg-modules-definition}
and~\ref{coderived-cdg-comodules-definition} below), and the proofs
of derived nonhomogeneous Koszul duality results such as
Theorems~\ref{complexes-of-modules-augmented-bar-construction-thm},
\ref{complexes-of-modules-quadratic-linear-thm},
\ref{complexes-of-comodules-coaugmented-cobar-thm},
and~\ref{complexes-of-comodules-conilpotent-cobar-thm}.

 Having made my discovery, I~went to the IAS library to look for
relevant prior literature.
 The most important thing I~found was the remarkable series of
papers by Eilenberg and Moore~\cite{EM0,EM1,EM2}, with the final
piece of the series authored by Husemoller, Moore, and
Stasheff~\cite{HMS}.
 In particular, the 1962 paper~\cite{EM0} clarified the issue of
``divergent spectral sequences'', which was how the difficulties
in Koszul duality were spoken of in my Moscow circles in
the early~'90s (see Remarks~\ref{convergent-spectral-sequence-remark}
and~\ref{divergent-spectral-sequence-remark} below for
a discussion).

 The 1974 paper~\cite{HMS} introduced the distinction between what
it called the \emph{differential derived functors of the first} and
\emph{the second kind}.
 I~immediately realized that it was relevant to my exotic derived
categories of DG\+modules and DG\+comodules; so I~seized on this
terminology and started to call my categories ``derived categories
of the second kind'' (as opposed to the conventional derived category,
which was the ``derived category of the first kind'').

 Only after my definitions of derived categories of the second kind
were already invented, I~discovered Hinich's papers~\cite{Hin1,Hin2}.
 My recollection is that I~may have also reinvented Hinich's
definition of a filtered quasi-isomorphism of conilpotent
DG\+coalgebras before reading it in~\cite{Hin2}.

 That \emph{direct sums of injective modules} and \emph{direct products
of projective modules} play an important role in my constructions of
derived categories of the second kind for DG\+modules, I~also
well realized in Spring~1999.
 In this connection, I~found references such as the papers~\cite{Cha}
and~\cite{Bas} (explaining when the products of projective modules are
projective), which were eventually cited in the memoir~\cite{Pkoszul}.
 But the more advanced conditions ($*$) and~(${*}{*}$)
from~\cite[Sections~3.7 and~3.8]{Pkoszul}
(see Theorem~\ref{co-contra-derived-cdg-modules} below)
were only invented in~2009.

 In fact, the condition that the direct sums of injective modules are
injective characterized Noetherian rings, which in my thinking of
the time was general enough for an interesting theory (cf.~\cite{Kra}).
 But the rings over which products of projective modules are projective
turned out to be more rare (as per~\cite{Bas,Cha}), essentially only
Artinian rings and their immediate generalizations.
 This created a misconception that the contraderived categories were
rarely well-behaved outside of the realms of contramodules over
coalgebras over fields and modules over Artinian rings, which haunted
me for the years to come (until Spring~2012,
when~\cite[Appendix~B]{Pweak} was largely written and the notion of
\emph{contraherent cosheaves} was discovered).

 1998--99 was a year of Geometric Representation Theory at the IAS.
 My postdoc was not a part of that Special Program, but many
mathematicians I~knew, including Finkelberg and other people from
Moscow, were Members of the IAS for that academic year as a part of
that program.
 In the Spring~99 we had an informal seminar where I~presented my new
results on Koszul duality in the form of a series of talks.
 Subsequently I~spoke about these results at various seminars
(including H.-J.~Baues' Oberseminar Topologie at MPIM-Bonn
in July~2001).

 The first preprint versions of what became my memoir~\cite{Pkoszul}
only appeared in Spring~2009.
 There were several reasons for such a protracted delay, the most
important of them being that the task of typing an original research
exposition of the size and complexity of~\cite{Pkoszul} was way
above my abilities in the first half of~'00s.
 My early research work consisted of very short papers; writing
longer texts required a different technique.
 An acute problem of insufficient mathematical writing skills, which
tormented me since Fall~1995, was only resolved by about Winter~2006/07.
 But there was also a terminological problem.

 Naming things is important.
 The expression ``derived category of the second kind'' was both too
long and too imprecise, as there were at least two important
``derived categories of the second kind'' known to me in Spring~1999
already.
 Now they are called the \emph{coderived} and \emph{contraderived}
categories.
 I~was unable to invent such a nomenclature myself.
 What I~had in Spring~99 was not a terminology for the distinction
between the two dual concepts of a coderived and a contraderived
category, but only a \emph{notation}, and not a very good one at that.
 The coderived category was denoted by $\sD'$, and the contraderived
category by~$\sD''$ (one can still see traces of this notation in
the old letters~\cite{Plet} and the recent paper~\cite{Pps}).
 A work such as~\cite{Pkoszul} could not be written in a readable form
with such clumsy terminological tools.

 The terminology ``coderived category'' first appeared in
Keller's note~\cite{Kel}.
 One of the mathematicians in Moscow directed my attention to this
note sometime in mid-'00s.
 Upon seeing that the terminological problem had been solved, I~knew
that the time had come for me to write up my results.
 I~started with a treatise on semi-infinite homological
algebra~\cite{Psemi}, which was much heavier with lots of details that
I~could not keep all in my mind, but could only work out in writing;
so I~had more motivation to type that work.
 A very elaborate terminological system featuring the prefixes ``co-''
and ``contra-'' (and also ``semi-'') was developed for the purposes of
the book~\cite{Psemi} based on the idea I~borrowed from Keller's note.

 I~typed the bulk of the material in~\cite{Pkoszul} while being
a visitor at the IHES in March--April~2009.
 I~contacted Bernhard Keller by e-mail from Bures and went to visit
him in his office in Paris, and subsequently gave a talk at
the Algebra Seminar in the Institut Henri Poincar\'e.
 The title of the talk was ``Koszul Triality''.
 This talk is mentioned in the acknowledgement to~\cite{KLN}.

\subsection{Philosophy of derived categories of the second kind, I}
\label{philosophy-of-second-kind-I-subsecn}
 A complex (e.~g., a complex of modules or a DG\+module) $K^\bu$
can be thought of in two ways.
 One can view $K^\bu$ as a deformation of its \emph{cohomology module}
$H^*(K^\bu)$, or as a deformation of its \emph{underlying graded
module} $K$ (say, endowed with the zero differential).
 In the notation of~\cite{Pkoszul} (going back
to~\cite[Remark~I.5.1]{HMS}), the important operation of forgetting
the differential is denoted by the upper index~$\#$; so one can
write $K=K^{\bu}{}^\#$.

 These two ways of looking at a complex are reflected in the classical
concept of \emph{two spectral sequences of
hypercohomology}~\cite[Section~XVII.3]{CaE}.
 The old word ``hypercohomology'' means the cohomology groups one
obtains by applying a derived functor to a complex (rather than just
to a module).
 Let $F$ be a left exact functor and $K^\bu$ be a complex of modules,
bounded below, to which we apply the right derived functor $\boR^*F$.
 Then there are two spectral sequences
\begin{align*}
 {}^{I}\!E_2^{p,q}=\boR^pF(H^q(K^\bu))
 &\Longrightarrow\boR^{p+q}F(K^\bu), \\
 {}^{II}\!E_1^{p,q}=\boR^qF(K^p)
 &\Longrightarrow\boR^{p+q}F(K^\bu),
\end{align*}
the former of which starts from the derived functor $\boR^*F$ applied
to the \emph{cohomology} of the complex $K^\bu$, while the latter one
starts from the derived functor $\boR^*F$ applied to the \emph{terms}
of the complex~$K^\bu$.
 So the spectral sequence ${}^{I}\!E$ expresses the point of view on
$K^\bu$ as a ``deformation of its cohomology'', while the spectral
sequence ${}^{II}\!E$ is an expression of the vision of $K^\bu$ as
a ``deformation of its terms''.
 The condition that $K^\bu$ is a bounded below complex ensures that
both the spectral sequences converge to the same limit (viz.,
the hypercohomology groups $\boR^*F(K^\bu)$).

\begin{rem} \label{convergent-spectral-sequence-remark}
 Let us illustrate by examples the role of convergent and ``divergent''
spectral sequences in Koszul duality, as mentioned in
Section~\ref{second-kind-reminiscences-subsecn}.
 We start with a convergent example; a ``divergent'' one will be
presented below in Remark~\ref{divergent-spectral-sequence-remark}.

 Let $A^\bu$ be a DG\+algebra, $(C^\bu,\gamma)$ be a conilpotent
DG\+coalgebra, and $\tau\:C^\bu\rarrow A^\bu$ be twisting cochain
satisfying the equation $\tau\circ\gamma=0$ (cf.\
Theorem~\ref{augmented-acyclic-twisting-cochain-duality-thm}).
 Let us show that the functor $M^\bu\longmapsto C^\bu\ot^\tau M^\bu$
takes acyclic DG\+modules over $A^\bu$ to acyclic DG\+comodules
over~$C^\bu$.

 Indeed, let $F$ be the canonical increasing filtration on
the conilpotent DG\+coalgebra $C^\bu$ (as in
Sections~\ref{conilpotent-coalgebras-subsecn}
and~\ref{conilpotent-dg-coalgebras-subsecn}).
 Denote also by $F$ the induced increasing filtration on the tensor
product $C\ot_kM$; so $F_n(C\ot_kM)=(F_nC)\ot_kM$ for all $n\ge0$.
 Then $F$ is a filtration of $C^\bu\ot^\tau M^\bu$ by subcomplexes
(in fact, by DG\+subcomodules).
 Moreover, the differential on the successive quotient complexes
$F_n(C^\bu\ot^\tau M^\bu)/F_{n-1}(C^\bu\ot^\tau M^\bu)$ is simply
the tensor product differential on the tensor product of complexes
$(F_nC^\bu/F_{n-1}C^\bu)\ot_kM^\bu$, essentially due to
the assumption that $\tau\circ\gamma=0$.
 If the complex $M^\bu$ is acyclic, it follows that the complex
$F_n(C^\bu\ot^\tau M^\bu)/F_{n-1}(C^\bu\ot^\tau M^\bu)$ is acyclic
as well for every $n\ge0$.

 Now the increasing filtration $F$ on the complex $C^\bu\ot^\tau M^\bu$
is exhaustive; so the related spectral sequence converges (this is
an easy version of~\cite[Theorem~7.4]{EM0}), and it follows that
the complex $C^\bu\ot^\tau M^\bu$ is acyclic.
 (Dually one can show, using a similar complete decreasing filtration,
that the functor $M^\bu\longmapsto\Hom^\tau(C^\bu,M^\bu)$ from
Sections~\ref{twisted-differential-on-graded-Hom-subsecn}\+-%
\ref{koszul-duality-contramodule-side-subsecn} below takes acyclic
DG\+modules over $A^\bu$ to acyclic DG\+contramodules over~$C^\bu$;
cf.\ Theorem~\ref{contra-side-conilpotent-duality-thm}.)
 This argument can be found in~\cite[proof of Theorem~6.4]{Pkoszul},
where it is used to prove coacyclicity rather than acyclicity.
\end{rem}

 Having dispelled the mystery of ``divergent spectral sequences'',
the authors of~\cite{EM0} embarked upon a program of applying derived
functors (such as Ext and Tor) to DG\+modules, and particularly of
defining the derived functor Cotor (of the functor of cotensor product
of comodules) and applying it to DG\+comodules~\cite{EM2}.
 The program culminated in the paper~\cite{HMS}, where is was pointed
out that the two hypercohomology spectral sequences for DG\+modules
(as well as, generally speaking, for unbounded complexes) converge, in
one sense or another, to \emph{two different limits}.
 These two limits were called the \emph{differential derived functors
of the first} and \emph{second kind} in~\cite{HMS}.
 The only difference between their constructions consisted in using two
different ways of totalize a given bicomplex: taking either direct sums,
or direct products along the diagonals~\cite[Definition~I.4.1]{HMS}.
 In retrospective, one can say that \emph{classical homological algebra}
(as exemplified by~\cite{CaE} but not~\cite{EM0}) consisted in
studying complexes under suitable boundedness or finite homological
dimension assumptions guaranteeing that the two kinds of differential
derived functors agree.

\subsection{Philosophy of derived categories of the second kind, II}
\label{philosophy-of-second-kind-II-subsecn}
 No derived categories were mentioned in the papers~\cite{EM0,EM2,HMS},
but in the subsequent decades people became interested in derived
categories (including derived categories of DG\+modules).
 An important property of a well-behaved derived category is that it
can be defined in two or more equivalent ways: either as the category
of complexes up to an \emph{equivalence relation} (such as
the quasi-isomorphism), or as a category of \emph{adjusted} complexes
(such as complexes of projective or injective objects), which can be
used as \emph{resolutions} for constructing derived functors.
 From the derived category perspective, \emph{classical homological
algebra} is the realm where the equivalence relation on complexes is
the conventional quasi-isomorphism, and the adjustedness conditions
are the usual termwise conditions: the resolutions are the complexes
of injective modules, or the complexes of projective modules, etc.

 In particular, for any abelian category $\sA$ with enough injective
objects, the bounded below derived category $\sD^+(\sA)$ is equivalent
to the homotopy category of bounded below complexes of injective
objects in~$\sA$.
 Dually, for any abelian category $\sB$ with enough projectives objects,
the bounded above derived category $\sD^-(\sB)$ is equivalent to
the homotopy category of bounded above complexes of projective objects
in~$\sB$.
 These are the classical homological algebra settings.

 Here is a thematic counterexample showing that the assertions in
the previous paragraph are \emph{not} true for unbounded complexes.
 Let $\Lambda=k[\varepsilon]/(\varepsilon^2)$ be the exterior algebra
in one variable (or in other words, the ring of dual numbers).
 Consider the unbounded complex of free $\Lambda$\+modules with
one generator
\begin{equation} \label{unbounded-acyclic-standard-example-eqn}
 \dotsb\lrarrow\Lambda\overset\varepsilon\lrarrow\Lambda\overset
 \varepsilon\lrarrow\Lambda\lrarrow\dotsb,
\end{equation}
where all the differentials are the operators of multiplication
with~$\varepsilon$.
 Then~\eqref{unbounded-acyclic-standard-example-eqn} is an unbounded
complex of projective (also, injective) $\Lambda$\+modules which is
acyclic, but not contractible.
 So~\eqref{unbounded-acyclic-standard-example-eqn} is a zero object
in the derived category $\sD(\Lambda\modl)$, but a nonzero object in
the homogopy category of complexes of projective (or injective)
objects in $\Lambda\modl$.
 The point is that~\eqref{unbounded-acyclic-standard-example-eqn} is
a quite \emph{nonzero} object in the coderived category
$\sD^\co(\Lambda\modl)$.

 Denote by $C=\Lambda^*$ the dual vector space to $\Lambda$ with
its natural coalgebra structure.
 Then the triangulated equivalence of
Theorem~\ref{complexes-of-comodules-conilpotent-cobar-thm}
essentially (up to obvious grading effects) assigns the acyclic
complex~\eqref{unbounded-acyclic-standard-example-eqn}, viewed
as a complex of $C$\+comodules (as per the construction in
Section~\ref{cdg-coalgebras-subsecn}), to the $k[x]$\+module
$k[x,x^{-1}]$ from Example~\ref{bar-of-modules-acyclic}(2).

\begin{rem} \label{divergent-spectral-sequence-remark}
 Now we can present an example of ``divergent spectral sequence''
in Koszul duality promised in
Remark~\ref{convergent-spectral-sequence-remark}.
 Let $(C,\gamma)$ be a (say, conilpotent) coaugmented coalgebra
over~$k$.
 Similarly to
Remark~\ref{cobar-not-preserves-quasi-isomorphisms-remarks},
we observe that the functor $N^\bu\longmapsto\Cb^\bu_\gamma(C,N^\bu)$
from Theorem~\ref{complexes-of-comodules-conilpotent-cobar-thm}
\emph{need not} take acyclic complexes of comodules over $C$ to acyclic
DG\+modules over the DG\+algebra $\Cb^\bu_\gamma(C)$.

 Indeed, consider the coalgebra $C=\Lambda^*$ as above, and
the complex of
$C$\+comod\-ules~\eqref{unbounded-acyclic-standard-example-eqn}.
 Then the DG\+algebra $A^\bu=\Cb^\bu_\gamma(C)$ is, in fact,
the polynomial ring $A=k[x]$, endowed with the obvious grading
and the zero differential.
 The functor $M^\bu\longmapsto C\ot^\tau M^\bu$ takes the graded
$A$\+module $k[x,x^{-1}]$ (viewed as a DG\+module with zero
differential) to the acyclic complex of
$C$\+comodules~\eqref{unbounded-acyclic-standard-example-eqn}.

 It is a part of
Theorem~\ref{complexes-of-comodules-conilpotent-cobar-thm} that
the adjunction morphism $A^\bu\ot^\tau C^\bu\ot^\tau M^\bu\rarrow
M^\bu$ is a quasi-isomorphism of DG\+modules for any DG\+module
$M^\bu$ over~$A^\bu$.
 It follows that the functor $\Cb^\bu_\gamma(C,{-})=A^\bu\ot^\tau{-}$
takes the acyclic complex of
$C$\+comodules~\eqref{unbounded-acyclic-standard-example-eqn} to
a nonacyclic DG\+module over $A^\bu$ quasi-isomorphic to $k[x,x^{-1}]$.

 How can it happen?
 Consider the decreasing filtration $F$ on the graded vector
space $\Cb_\gamma(C)=\bigoplus_{n=0}^\infty C^+{}^{\ot n}$ defined
by the obvious rule $F^n\Cb_\gamma(C)=\bigoplus_{i=n}^\infty
C^+{}^{\ot i}$.
 Given a complex of $C$\+comodules $N^\bu$, let $F$ be the induced
filtration on the graded vector space $\Cb_\gamma(C,N)=
\Cb_\gamma(C)\ot_k N$; so $F^n\Cb_\gamma(C,N)=F^n\Cb_\gamma(C)\ot_kN$.
 Then $F^n\Cb^\bu_\gamma(C,N^\bu)$ is a subcomplex in
$\Cb^\bu_\gamma(C,N^\bu)$ for every $n\ge0$.
 Furthermore, the quotient complex $F^n\Cb^\bu_\gamma(C,N^\bu)/
F^{n+1}\Cb^\bu_\gamma(C,N^\bu)$ is the tensor product
$C^+{}^{\ot n}\ot_k N^\bu$ with the differential induced by
the differential on~$N^\bu$.
 So all the quotient complexes $F^n\Cb^\bu_\gamma(C,N^\bu)/
F^{n+1}\Cb^\bu_\gamma(C,N^\bu)$ are acyclic whenever
a complex of $C$\+comodules $N^\bu$ is acyclic.
 Still, the complex $\Cb^\bu_\gamma(C,N^\bu)$ need \emph{not} be
acyclic.
 How can it happen?

 This is what people na\"\i vely call a ``divergent spectral sequence''.
 In fact, the spectral sequence associated with the filtration $F$ on
the cobar-complex $\Cb^\bu_\gamma(C,N^\bu)$ does not diverge at all;
it just does not converge to the cohomology of the complex
$\Cb^\bu_\gamma(C,N^\bu)$.
 It is not supposed to.
 The point is that the decreasing filtration $F$ on the complex
$\Cb^\bu_\gamma(C,N^\bu)$ is \emph{not complete}.
 Simply put, if one wants the spectral sequence of a decreasing
filtration on the totalization of a bicomplex to converge to
the cohomology of the totalization, one needs to totalize the bicomplex
by taking \emph{direct products} along the diagonals and
not the direct sums.
 The spectral sequence associated with the filtration $F$ on
the cobar-complex $\Cb^\bu_\gamma(C,N^\bu)$ converges (at least in
the weak sense of~\cite[Theorem~7.4]{EM0}) to the cohomology of
the \emph{completion} of the complex $\Cb^\bu_\gamma(C,N^\bu)$,
that is, to the cohomology of the \emph{direct product totalization}
of the bicomplex.

 Let us denote by $\ovCb^\bu(C,N^\bu)$ the completed cobar-complex,
i.~e., the direct product totalization of the cobar-bicomplex of $C$
with the coefficients in~$N^\bu$.
 Then, by~\cite[Theorem~7.4]{EM0}, the complex $\ovCb^\bu(C,N^\bu)$ is
indeed acyclic for any acyclic complex of $C$\+comodules $N^\bu$
(cf.\ Remark~\ref{inverting-the-arrows-remark}).
 The completed cobar-construction is an interesting functor on its own;
a discussion of it can be found, e.~g.,
in~\cite[Section~8 of Chapter~5]{PP}.
 But it is \emph{not} the adjoint functor to the functor
$M^\bu\longmapsto C\ot^\tau M^\bu$ in
Theorem~\ref{complexes-of-comodules-conilpotent-cobar-thm}
(cf. Section~\ref{augmented-duality-comodule-side-subsecn}).

 As suggested in the first paragraph of the introduction
to~\cite{Pkoszul} (see~\cite[Section~0.1]{Pkoszul}), \emph{if one's
spectral sequence diverges, one should either replace the complex
with its completion, or choose a different filtration}.
 In the context of Koszul duality, it is better to choose
a different filtration.
 That is what one accomplishes by replacing acyclic complexes or
DG\+comodules with coacyclic ones.
\end{rem}

 The reader can find a further introductory discussion of
the complex~\eqref{unbounded-acyclic-standard-example-eqn} and its place
in Koszul duality in the book~\cite[Prologue]{Prel}, and an introductory
discussion of derived categories of the second kind largely centered
around the example~\eqref{unbounded-acyclic-standard-example-eqn}
in the paper~\cite[Section~5]{PSch}.
 This material goes back to~\cite[Example~3.3]{Pkoszul}, where another
counterexample can be also found: a totally finite-dimensional
(in fact, two-dimensional) DG\+module $M^\bu$ over a (two-dimensional)
DG\+algebra $B^\bu$ such that the underlying graded $B$\+module $M$ is
projective and injective, while the DG\+module $M^\bu$ is acyclic
but not contractible.

 The triangulated equivalence of
Theorem~\ref{augmented-acyclic-twisting-cochain-duality-thm}
assigns the acyclic DG\+module $M^\bu$ (viewed as a DG\+comodule
over the DG\+coalgebra $C^\bu=B^\bu{}^*$) to the $k[x]$\+module~$k_1$
($a=1$) from Example~\ref{bar-of-modules-acyclic}(2).
 Both the acyclic complex~\eqref{unbounded-acyclic-standard-example-eqn}
and the acyclic DG\+mod\-ule $M^\bu$ can be also viewed as assigned to
the respective modules over the abelian Lie algebra $\g=k$ by
the triangulated equivalence~\eqref{fin-dim-lie-algebra-koszul-duality}
from Example~\ref{chevalley-eilenberg}.

\begin{ex} \label{semisimple-chevalley-eilenberg}
 To give a further series of examples, let us continue the discussion
of Lie algebras from Example~\ref{chevalley-eilenberg}.
 Let $\g$~be a finite-dimensional semisimple Lie algebra over
a field~$k$ of characteristic~$0$, and let $\Lambda^\bu(\g^*)$ be
its cohomological Chevalley--Eilenberg complex.
 To any complex of modules $M^\bu$ over~$\g$, the Koszul
duality~\eqref{fin-dim-lie-algebra-koszul-duality} assigns
the cohomological Chevalley--Eilenberg complex (or rather, the total
complex of the bicomplex) $\Lambda^\bu(\g^*)\ot_k M^\bu$, with
the differential including summands induced by the Lie bracket
on~$\g$, the action of~$\g$ in $M$, and the differential on~$M$.
 Here the complex $\Lambda^\bu(\g^*)\ot_k M^\bu$ is viewed a
a DG\+module over the DG\+algebra~$\Lambda^\bu(\g^*)$.
 
 In particular, for any $\g$\+module $M$, the Chevalley--Eilenberg
complex $(\Lambda^\bu(\g^*)\ot_kM,\>d)$ computes the Lie algebra
cohomology $H^*(\g,M)$ with the coefficients in~$M$.
 Now one has $H^*(\g,M)=0$ for any nontrivial finite-dimensional
irreducible $\g$\+module~$M$.
 So $(\Lambda^\bu(\g^*)\ot_kM,\>d)$ is another example of
a finite-dimensional acyclic, noncontractible DG\+module whose
underlying graded module is both projective and injective.

 Let us explain in yet another way why a triangulated equivalence
such as~\eqref{fin-dim-lie-algebra-koszul-duality} \emph{cannot} hold
with the conventional derived category $\sD(\Lambda^\bu(\g^*)\modl)$
used instead of the coderived category.
 Any quasi-isomorphism of DG\+algebras $f\:A^\bu\rarrow B^\bu$ induces
a triangulated equivalence of the derived categories of DG\+modules,
$\sD(A^\bu\modl)\simeq\sD(B^\bu\modl)$ \,\cite[Example~6.1]{Kel0},
\cite[Theorem~1.7]{Pkoszul}.
 It is well-known that, for a semisimple Lie algebra~$\g$ in
$\operatorname{char}k=0$, the DG\+algebra $\Lambda^\bu(\g^*)$ is
formal (i.~e., quasi-isomorphic to its cohomology algebra);
indeed, the subalgebra of $\g$\+invariant elements in
$\Lambda^\bu(\g^*)$ is a DG\+subalgebra with zero differential
mapping quasi-isomorphically onto the cohomology.
 Moreover, the cohomology algebra $H^*(\g)=H^*(\g,k)$ is not very
informative: it is an exterior (free graded commutative) algebra
with generators in certain odd degrees.
 There is \emph{no} hope of recovering the derived category of
$\g$\+modules $\sD(\g\modl)$ from the quasi-isomorphism class of
the DG\+algebra~$\Lambda^\bu(\g^*)$.
\end{ex}

 So phenomena appearing to be specific to unbounded complexes also
manifest themselves in totally bounded, finite-dimensional DG\+modules.
 The explanation is that, if one wants to stay within the realm of
classical homological algebra, then one has to restrict oneself to
suitably bounded DG\+modules over DG\+rings $A^\bu$ belonging to
one of the following two classes: either $A^\bu$ is nonpositively
cohomologically graded~\cite[Theorem~3.4.1]{Pkoszul}, or $A^\bu$
is connected, simply connected, and positively cohomologically
graded~\cite[Theorem~3.4.2]{Pkoszul}.
 The DG\+algebra $B^\bu$ from~\cite[Example~3.3]{Pkoszul} (discussed
above), as well as, more generally, the DG\+algebra
$\Lambda^\bu(\g^*)$ from Examples~\ref{chevalley-eilenberg}
and~\ref{semisimple-chevalley-eilenberg}, belong to neither
of the two classes.

\subsection{Philosophy of derived categories of the second kind, III}
\label{philosophy-of-second-kind-III-subsecn}
 Thus the choice between derived categories of the first and the second
kind can be presented as consisting in \emph{having to decide what to do
with the complex~\eqref{unbounded-acyclic-standard-example-eqn}}.
 Is it to be considered as a trivial object and not acceptable as
a resolution (for computing derived functors with it), or as
a nontrivial object and acceptable as a resolution?
 We can now summarize the philosophy of two kinds of derived categories
as follows.

 In derived categories of the first kind (the conventional derived
categories):
\begin{itemize}
\item a complex or a DG\+module is viewed as a deformation of its
cohomology;
\item the conventional quasi-isomorphism (that is, the property of
a morphism to induce an isomorphism on the cohomology objects) is
the equivalence relation on complexes or DG\+modules in the derived
category;
\item strong and complicated conditions need to be imposed on
resolutions: the adjusted complexes or DG\+modules are known as
the \emph{homotopy projectives}, \emph{homotopy flats},
\emph{homotopy injectives} etc.; in slightly different terminological
systems they are called \emph{$K$\+projectives}, \emph{$K$\+flats}, 
\emph{$K$\+injectives} etc., or \emph{DG\+projectives},
\emph{DG\+flats}, \emph{DG\+injectives}, etc.\
(see~\cite[Remark~6.4]{PS4} for the terminological discussion).
\end{itemize}

 It is important here that one \emph{cannot} tell whether a complex or
DG\+module is homotopy adjusted (i.~e., can be used as a resolution
for derived categories or derived functors of the first kind) by
looking only on the terms of the complex or on the underlying graded
module of a DG\+module.
 The property of a complex/DG\+module to be homotopy adjusted
\emph{depends} on the differential.
 The classical works where the theory of derived categories of
the first kind was developed are~\cite{Spal,Kel0,Hin1}.

 In derived categories of the second kind (the coderived,
contraderived, and absolute derived categories):
\begin{itemize}
\item a complex or a (C)DG\+module is viewed as a deformation of its
underlying graded module, or of itself endowed with the zero
differential;
\item strong and complicated equivalence relations are imposed on
complexes or DG\+modules in the derived categories; the related
classes of trivial objects are called the \emph{coacyclic},
\emph{contraacyclic}, or \emph{absolutely acyclic} complexes or
(C)DG\+modules; for conilpotent (C)DG\+coalgebras, the equivalence
relation of \emph{filtered quasi-isomorphism} is used;
\item all the conventional termise injective or termwise projective
complexes, or (C)DG\+modules with injective/projective underlying
graded modules are adjusted (i.~e., can be used as resolutions);
more precisely:
\begin{itemize}
\item in the coderived categories, all complexes of injective
objects or graded-injective (C)DG\+modules are acceptable as
resolutions;
\item in the contraderived categories, all complexes of projective
objects or graded-projective (C)DG\+modules are acceptable as
resolutions.
\end{itemize}
\end{itemize}

 It is important here that one \emph{cannot} tell whether a morphism
is a weak equivalence for a derived category of the second kind,
or whether a complex or DG\+module is coacyclic/contraacyclic/etc.\
by looking only on its underlying complex of abelian groups.
 The property of a complex/DG\+module to be co- or contraacyclic
\emph{depends} on the module structure.
 The classical works where the theory of derived categories of
the second kind was developed are~%
\cite{Hin2,Lef,Kel,Jor,Kra,IK,Neem2,KLN,Psemi,Pkoszul,EP,Bec,Neem3,Sto}.

 The terminology \emph{coderived category} reflects an understanding
that the coderived categories are most suitable for \emph{comodules}.
 This was perhaps an intended meaning of the term in Keller's
note~\cite{Kel}, where it was first introduced (in fact, in
the original definition in~\cite{Kel}, the coderived category was
defined for DG\+comodules \emph{only}; the definition was not
applicable to modules or DG\+modules).

 Dually, the \emph{contraderived category} is most suitable for
\emph{contramodules}, while the conventional derived category works
best for complexes of modules or DG\+modules.
 This philosophy was developed and played a central role in the present
author's monograph on semi-infinite homological algebra~\cite{Psemi}.
 The technical aspects of the assertion ``coderived categories are
best behaved for comodules, contraderived categories for contramodules,
and conventional derived categories for modules'' are discussed in
detail in the memoir~\cite[Sections~1\+-4]{Pkoszul} (cf.\
Sections~\ref{co-contra-derived-cdg-modules-subsecn}\+-%
\ref{coderived-cdg-comodules-subsecn}
and~\ref{contraderived-cdg-contramodules-subsecn} below).

 But there are also other considerations.
 In particular, \emph{the conventional derived category makes no sense
for curved structures}.
 There is a partial exception to this rule in the (complicated,
technical, and special) \emph{weakly curved} case, as developed
in~\cite{Pweak}.
 But generally speaking, CDG\+modules (as well as CDG\+rings,
CDG\+comodules etc.)\ have \emph{no} cohomology groups or modules,
so the notion of a conventional quasi-iso\-mor\-phism is
\emph{undefined} for them (as we mentioned already in
Section~\ref{posing-the-problem-nonaugmented-duality}).
 Thus, for CDG\+modules one has to consider derived categories of
the second kind.

 This presents no problem for Koszul duality.
 But, say, derived categories of \emph{matrix factorizations} can be
only defined as derived categories of the second kind~\cite{Or,PP2,EP}.
 So sometimes one is forced to work with the coderived categories of
modules, technically complicated as they might be.
 Conversely, differential derived functors of the first kind for
DG\+comodules appear in the context of the Eilenberg--Moore
spectral sequence~\cite{EM2}; see~\cite[Section~0.2.10]{Psemi}
for a discussion.

\subsection{Coderived and contraderived categories of CDG-modules}
\label{co-contra-derived-cdg-modules-subsecn}
 To any bicomplex one can assign its total complex.
 In fact, there are several ways to do it: at least, one has to choose
between taking infinite direct sums or infinite products along
the diagonals.
 There are further possibilities, which may be called
the \emph{Laurent totalizations} (as in ``the Laurent formal
power series''): one can take direct sums in one direction and
direct products in the other one~\cite[Definition~I.3.4]{HMS}.
 For bicomplexes with only a finite number of nonzero terms on every
diagonal, there is only one way of producing the total complex.

 Here we are interested in a special class of \emph{bicomplexes with
three rows}: namely, the \emph{short exact sequences of complexes}.
 It is straightforward to see that the total complex of any short exact
sequence of complexes (say, in an abelian category, or even in
an exact category) is acyclic.

 The concept of totalization can be extended to DG\+categories, where
one can speak of the (product or coproduct) totalization of a complex
of objects and \emph{closed morphisms of degree\/~$0$} between
them~\cite[Section~1.2]{Pkoszul}, \cite[Section~1.3]{Pedg}.
 In particular, one can consider complexes (e.~g., short exact
sequences) of DG\+modules or CDG\+mod\-ules, and their totalizations.

 Specifically, let $B^\bu=(B,d)$ be a DG\+ring.
 A \emph{short exact sequence of DG\+modules} $0\rarrow K^\bu\rarrow
L^\bu\rarrow M^\bu\rarrow0$ over $B^\bu$ consists of three (say, left)
DG\+modules $K^\bu$, $L^\bu$, $M^\bu$ and two morphisms
$f\in\Hom^0_B(K,L)$ and $g\in\Hom^0_B(L,M)$ such that $d(f)=0=d(g)$ in
the respective complexes $\Hom^\bu_B(K,L)$ and $\Hom^\bu_B(L,M)$, and
$0\rarrow K\overset f\rarrow L\overset g\rarrow M\rarrow0$ is
a short exact sequence of graded $B$\+modules.
 The \emph{totalization} $\Tot(K^\bu\to L^\bu\to M^\bu)$ of $0\rarrow
K^\bu\rarrow L^\bu\rarrow M^\bu\rarrow0$ is constructed by passing to
the direct sums of the (at most three) grading components along every
diagonal, and setting the total differential to be the sum of $d_K$,
$d_L$, $d_M$, $f$, and~$g$ with suitable signs.
 Then $\Tot(K^\bu\to L^\bu\to M^\bu)$ is again a (left)
DG\+module over~$B^\bu$.

 The definitions of a \emph{short exact sequence of CDG\+modules}
$0\rarrow K^\cu\rarrow L^\cu\rarrow M^\cu\rarrow0$ over a CDG\+ring
$B^\cu$ and its \emph{total CDG\+module} $\Tot(K^\cu\to L^\cu\to M^\cu)$
over $B^\cu$ are spelled out literally the same (recall that the Hom
of two CDG\+modules over $B^\cu$ is a complex, as explained in
Section~\ref{cdg-rings-subsecn}).
 The totalization can be also constructed as an iterated cone: take
the cone $\cone(f)$ of the closed morphism $f\:K^\cu\rarrow L^\cu$ in
the DG\+category of CDG\+modules $B^\cu\modl$, and put
$\Tot(K^\cu\to L^\cu\to M^\cu)=\cone(\cone(f)\to M^\cu)$.
 (We are only speaking of the totalization as defined up to some
cohomological grading shift, and refrain from specifying any preferred
choice of such grading shift, as it is not important for our purposes.)

 Let $\Hot(B^\cu\modl)$ denote the homotopy category of left
CDG\+modules over a CDG\+ring $B^\cu$, i.~e., the category whose
objects are the CDG\+modules over $B^\cu$ and morphisms are
closed morphisms of degree~$0$ up to cochain homotopy.
 As the DG\+category $B^\cu\modl$ has shifts and cones, its homotopy
category $\Hot(B^\cu\modl)$ is a triangulated category~\cite{BonKap}.

\begin{defn}[{\cite[Section~3.3]{Pkoszul}}]
\label{absolute-derived-cdg-modules-definition}
 Let $B^\cu$ be a CDG\+ring.
 A left CDG\+module over $B^\cu$ is said to be \emph{absolutely acyclic}
if it belongs to the minimal thick subcategory of $\Hot(B^\cu\modl)$
containing the totalizations of short exact sequences of
left CDG\+modules over~$B^\cu$.
 The full subcategory of absolutely acyclic CDG\+modules is denoted
by $\Ac^\abs(B^\cu\modl)\subset\Hot(B^\cu\modl)$.
 The triangulated Verdier quotient category
$$
 \sD^\abs(B^\cu\modl)=\Hot(B^\cu\modl)/\Ac^\abs(B^\cu\modl)
$$
is called the \emph{absolute derived category} of left CDG\+modules
over~$B^\cu$.
\end{defn}

\begin{defn}[{\cite[Section~2.1]{Psemi}, \cite[Section~3.3]{Pkoszul}}]
\label{coderived-cdg-modules-definition}
 A left CDG\+module over $B^\cu$ is said to be \emph{coacyclic}
if it belongs to the minimal triangulated subcategory
of $\Hot(B^\cu\modl)$ containing the totalizations of short exact
sequences of left CDG\+modules over~$B^\cu$ \emph{and closed under
infinite direct sums}.
 The thick subcategory of coacyclic CDG\+modules is denoted
by $\Ac^\co(B^\cu\modl)\subset\Hot(B^\cu\modl)$.
 The triangulated Verdier quotient category
$$
 \sD^\co(B^\cu\modl)=\Hot(B^\cu\modl)/\Ac^\co(B^\cu\modl)
$$
is called the \emph{coderived category} of left CDG\+modules
over~$B^\cu$.
\end{defn}

\begin{defn}[{\cite[Section~4.1]{Psemi}, \cite[Section~3.3]{Pkoszul}}]
\label{contraderived-cdg-modules-definition}
 A left CDG\+module over $B^\cu$ is said to be \emph{contraacyclic}
if it belongs to the minimal triangulated subcategory
of $\Hot(B^\cu\modl)$ containing the totalizations of short exact
sequences of left CDG\+modules over~$B^\cu$ \emph{and closed under
infinite products}.
 The thick subcategory of contraacyclic CDG\+modules is denoted
by $\Ac^\ctr(B^\cu\modl)\subset\Hot(B^\cu\modl)$.
 The triangulated Verdier quotient category
$$
 \sD^\ctr(B^\cu\modl)=\Hot(B^\cu\modl)/\Ac^\ctr(B^\cu\modl)
$$
is called the \emph{contraderived category} of left CDG\+modules
over~$B^\cu$.
\end{defn}

\begin{ex}[{\cite[Examples~3.3]{Pkoszul}}]
 The doubly unbounded acyclic complex of $\Lambda$\+modules in
the formula~\eqref{unbounded-acyclic-standard-example-eqn}
from Section~\ref{philosophy-of-second-kind-II-subsecn}
is \emph{neither} coacyclic, \emph{nor} contraacyclic.
 Its (bounded above) acyclic subcomplex of canonical truncation
is \emph{contraacyclic}, but \emph{not} coacyclic.
 Dually, the (bounded below) acyclic quotient complex of canonical
truncation of~\eqref{unbounded-acyclic-standard-example-eqn} is
\emph{coacyclic}, but \emph{not} contraacyclic.
 Consequently, \emph{neither} one of the three complexes is
absolutely acyclic.
 We refer to~\cite[beginning of Section~5]{PSch} for a detailed
discussion.
\end{ex}

 The maximal natural generality for 
Definitions~\ref{absolute-derived-cdg-modules-definition}\+-%
\ref{contraderived-cdg-modules-definition} is that of \emph{exact
DG\+cat\-e\-gories}~\cite[Section~5.1]{Pedg}.
 This concept, suggested in~\cite[Section~3.2 and
Remarks~3.5\+-3.7]{Pkoszul}, was worked out in the preprint~\cite{Pedg}.

 The following two theorems are the main results of the basic theory
of coderived, contraderived, and absolute derived categories of
CDG\+modules.

\begin{thm} \label{fin-homol-dim-derived-cdg-modules}
\textup{(a)} Let $B^\cu=(B,d,h)$ be a CDG\+ring whose underlying graded
ring $B$ has finite left global dimension (as a graded ring; i.~e.,
the abelian category of graded left $B$\+modules has finite homological
dimension).
 Then the three classes of coacyclic, contraacyclic, and absolutely
acyclic CDG\+modules over $B^\cu$ coincide,
$$
 \Ac^\co(B^\cu\modl)=\Ac^\abs(B^\cu\modl)=\Ac^\ctr(B^\cu\modl),
$$
and accordingly, the three derived categories of the second kind
coincide,
$$
 \sD^\co(B^\cu\modl)=\sD^\abs(B^\cu\modl)=\sD^\ctr(B^\cu\modl).
$$ \par
\textup{(b)} Let $A^\bu=(A,d)$ be a DG\+ring whose underlying graded
ring $A$ has finite left global dimension (as a graded ring).
 Assume additionally that \emph{either} $A^n=0$ for all $n>0$,
\emph{or otherwise} $A^n=0$ for all $n<0$, the ring $A^0$ is
(classically) semisimple, and $A^1=0$.
 Then the three classes of coacyclic, contraacyclic, and absolutely
acyclic DG\+modules over $A^\bu$ coincide with the class of all
acyclic DG\+modules,
$$
 \Ac(A^\bu\modl)=\Ac^\co(A^\bu\modl)=
 \Ac^\abs(A^\bu\modl)=\Ac^\ctr(A^\bu\modl).
$$
 Accordingly, all the four derived categories coincide,
$$
 \sD(A^\bu\modl)=\sD^\co(A^\bu\modl)=
 \sD^\abs(A^\bu\modl)=\sD^\ctr(A^\bu\modl).
$$
\end{thm}

\begin{proof}
 Part~(a) is~\cite[Theorem~3.6(a)]{Pkoszul}.
 For an alternative proof based on the concept of a \emph{cotorsion
pair}, see~\cite[Corollaries~4.8 and~4.15]{Pctrl}.
 For a generalization to exact DG\+categories (covering also
Theorems~\ref{fin-homol-dim-derived-cdg-comodules}
and~\ref{fin-homol-dim-derived-cdg-contramodules} below),
see~\cite[Theorem~5.6]{Pedg}, or for a much wider generalization
with weaker assumptions, \cite[Theorem~8.9]{Pedg}.

 Part~(b) for complexes of modules over a ring of finite global
dimension is a particular case of~\cite[Remark~2.1]{Psemi}.
 In full generality, part~(b) is provable by comparing part~(a)
with~\cite[Theorem~3.4.1(d)]{Pkoszul} (in the first case) or
with~\cite[Theorem~3.4.2(d)]{Pkoszul} (in the second case).
\end{proof}

 For a further discussion of assumptions under which a result like
Theorem~\ref{fin-homol-dim-derived-cdg-modules}(b) holds, explore
the references in~\cite[last paragraph of Section~3.6]{Pkoszul}.
 In particular, by~\cite[Theorem~9.4]{Pkoszul}, one has
$\Ac(A^\bu\modl)=\Ac^\co(A^\bu\modl)=\Ac^\abs(A^\bu\modl)=
\Ac^\ctr(A^\bu\modl)$ for any \emph{cofibrant} DG\+algebra $A^\bu$
over a commutative ring of finite global dimension.

 Let $B$ be a graded ring.
 The following conditions on~$B$ are relevant for the theory of derived
categories of the second kind~\cite[Sections~3.7 and~3.8]{Pkoszul}:
\begin{itemize}
\item[($*$)] any countable direct sum of injective graded left
$B$\+modules has finite injective dimension (as a graded $B$\+module);
\item[(${*}{*}$)] any countable product of projective graded left
$B$\+modules has finite projective dimension (as a graded $B$\+module).
\end{itemize}
 Notice that any graded ring $B$ of finite left global dimension
(as in Theorem~\ref{fin-homol-dim-derived-cdg-modules}(a))
satisfies both ($*$) and~(${*}{*}$).

 One denotes by $\Hot(B^\cu\modl_\inj)\subset\Hot(B^\cu\modl)$ the full
triangulated subcategory in the homotopy category formed by all
the CDG\+modules $J^\cu=(J,d_J)$ over $B^\cu$ whose underlying graded
left $B$\+modules $J$ are injective (as graded left $B$\+modules).
 Such CDG\+modules $J^\cu$ over $B^\cu$ are called
\emph{graded-injective}.

 Dually, one denotes by $\Hot(B^\cu\modl_\proj)\subset\Hot(B^\cu\modl)$
the full triangulated subcategory in the homotopy category formed by all
the CDG\+modules $P^\cu=(P,d_P)$ over $B^\cu$ whose underlying graded
left $B$\+modules $P$ are projective.
 Such CDG\+modules $P^\cu$ over $B^\cu$ are called
\emph{graded-projective}.
 
\begin{thm} \label{co-contra-derived-cdg-modules}
\textup{(a)} Let $B^\cu=(B,d,h)$ be a CDG\+ring whose underlying
graded ring $B$ satisfies condition~\textup{($*$)}.
 Then the composition\/ $\Hot(B^\cu\modl_\inj)\rarrow\Hot(B^\cu\modl)
\rarrow\sD^\co(B^\cu\modl)$ of the triangulated inclusion functor\/
$\Hot(B^\cu\modl_\inj)\rarrow\Hot(B^\cu\modl)$ and the Verdier quotient
functor\/ $\Hot(B^\cu\modl)\rarrow\sD^\co(B^\cu\modl)$ is
an equivalence of triangulated categories,
$$
 \Hot(B^\cu\modl_\inj)\simeq\sD^\co(B^\cu\modl).
$$ \par
\textup{(b)} Let $B^\cu=(B,d,h)$ be a CDG\+ring whose underlying
graded ring $B$ satisfies condition~\textup{(${*}{*}$)}.
 Then the composition\/ $\Hot(B^\cu\modl_\proj)\rarrow\Hot(B^\cu\modl)
\rarrow\sD^\ctr(B^\cu\modl)$ of the triangulated inclusion functor\/
$\Hot(B^\cu\modl_\proj)\rarrow\Hot(B^\cu\modl)$ and the Verdier quotient
functor\/ $\Hot(B^\cu\modl)\rarrow\sD^\ctr(B^\cu\modl)$ is
an equivalence of triangulated categories, \hfuzz=5pt
$$
 \Hot(B^\cu\modl_\proj)\simeq\sD^\ctr(B^\cu\modl).
$$
\end{thm}

\begin{proof}
 Part~(a) is~\cite[Theorem~3.7]{Pkoszul}, and part~(b)
is~\cite[Theorem~3.8]{Pkoszul}.
 For an alternative proof based on the notion of a \emph{cotorsion
pair}, see~\cite[Corollary~4.18]{Pctrl} for part~(a)
and~\cite[Corollary~4.9]{Pctrl} for part~(b).
 For a generalization to exact DG\+categories (covering also
Theorems~\ref{coderived-cdg-comodules-thm}
and~\ref{contraderived-cdg-contramodules-thm} below),
see~\cite[Theorem~5.10]{Pedg}.
\end{proof}

\subsection{Coderived category of CDG-comodules}
\label{coderived-cdg-comodules-subsecn}
 This section is a comodule version of the previous one.
 In the spirit of the discussion in the end of
Section~\ref{philosophy-of-second-kind-III-subsecn}, we will see
(in Theorem~\ref{coderived-cdg-comodules-thm}) that
the coderived categories of comodules are somewhat better behaved
than the coderived categories of modules.
 On the other hand, it makes \emph{no} sense to consider
``contraderived categories of comodules'', as the functors of
infinite products are usually \emph{not} exact in comodule categories.

 Let $C^\cu=(C,d,h)$ be a CDG\+coalgebra over~$k$.
 Similarly to Section~\ref{co-contra-derived-cdg-modules-subsecn},
one can speak of short exact sequences $0\rarrow K^\cu\rarrow L^\cu
\rarrow M^\cu\rarrow0$ of left CDG\+comodules over $C^\cu$ and their
totalizations (total CDG\+comodules) $\Tot(K^\cu\to L^\cu\to M^\cu)$,
which are again left CDG\+comodules over~$C^\cu$.

 Let $\Hot(C^\cu\comodl)$ denote the homotopy category of left
CDG\+comodules over a CDG\+coalgebra~$C^\cu$.
 Similarly to Section~\ref{co-contra-derived-cdg-modules-subsecn},
the DG\+category $C^\cu\comodl$ has shifts and cones, so its
homotopy category $\Hot(C^\cu\comodl)$ is triangulated.

\begin{defn}[{\cite[Section~4.2]{Pkoszul}}]
\label{absolute-derived-cdg-comodules-definition}
 Let $C^\cu$ be a CDG\+coalgebra.
 A left CDG\+comod\-ule over $C^\cu$ is said to be \emph{absolutely
acyclic} if it belongs to the minimal thick subcategory of
$\Hot(C^\cu\comodl)$ containing the totalizations of short exact
sequences of left CDG\+comodules over~$C^\cu$.
 The full subcategory of absolutely acyclic CDG\+comodules is denoted
by $\Ac^\abs(C^\cu\comodl)\subset\Hot(C^\cu\comodl)$.
 The triangulated Verdier quotient category
$$
 \sD^\abs(C^\cu\comodl)=\Hot(C^\cu\comodl)/\Ac^\abs(C^\cu\comodl)
$$
is called the \emph{absolute derived category} of left CDG\+comodules
over~$C^\cu$.
\end{defn}

\begin{defn}[{\cite[Section~2.1]{Psemi}, \cite[Section~4.2]{Pkoszul}}]
\label{coderived-cdg-comodules-definition}
 A left CDG\+comodule over $C^\cu$ is said to be \emph{coacyclic}
if it belongs to the minimal triangulated subcategory
of $\Hot(C^\cu\comodl)$ containing the totalizations of short exact
sequences of left CDG\+comodules over~$C^\cu$ \emph{and closed under
infinite direct sums}.
 The thick subcategory of coacyclic CDG\+comodules is denoted
by $\Ac^\co(C^\cu\comodl)\subset\Hot(C^\cu\comodl)$.
 The triangulated Verdier quotient category
$$
 \sD^\co(C^\cu\comodl)=\Hot(C^\cu\comodl)/\Ac^\co(C^\cu\comodl)
$$
is called the \emph{coderived category} of left CDG\+comodules
over~$C^\cu$.
\end{defn}

 In the context of the next theorem, one can keep in mind that,
\emph{unlike} for rings, the left and right global dimensions agree
for any (ungraded or graded) coalgebra $C$ over a field~$k$
(see~\cite[beginning of Section~4.5]{Pkoszul}).

\begin{thm} \label{fin-homol-dim-derived-cdg-comodules}
 Let $C^\cu=(C,d,h)$ be a CDG\+coalgebra over~$k$ whose
underlying graded coalgebra $C$ has finite left global dimension
(as a graded coalgebra; i.~e., the abelian category of graded left
$C$\+comodules has finite homological dimension).
 Then the two classes of coacyclic and absolutely acyclic
CDG\+comodules over $C^\cu$ coincide,
$$
 \Ac^\co(C^\cu\comodl)=\Ac^\abs(C^\cu\modl),
$$
and accordingly, the two derived categories of the second kind coincide,
$$
 \sD^\co(C^\cu\comodl)=\sD^\abs(C^\cu\comodl).
$$
\end{thm}

\begin{proof}
 This is~\cite[Theorem~4.5(a)]{Pkoszul}.
 For a generalization to exact DG\+categories,
see~\cite[Theorem~5.6(a) or Theorem~8.9(a)]{Pedg}.
\end{proof}

 It would be interesting to obtain a DG\+comodule version of
Theorem~\ref{fin-homol-dim-derived-cdg-modules}(b).
 The nonexistence of a meaningful notion of a ``contraacyclic
DG\+comodule'' stands in the way of applying an argument similar to
the proof of Theorem~\ref{fin-homol-dim-derived-cdg-modules}(b) above
(cf.~\cite[Section~4.3]{Pkoszul}).
 The case of complexes of comodules over a coalgebra of finite global
dimension is covered by~\cite[Remark~2.1]{Psemi}.

 Notice that the comodule version of condition~($*$) from
Section~\ref{co-contra-derived-cdg-modules-subsecn} holds for
\emph{any} (graded) coalgebra $C$ over a field~$k$, because the class
of all injective $C$\+comodules is closed under infinite direct sums.
 Indeed, the injective $C$\+comodules are the direct summands of
the cofree ones (cf.\ Section~\ref{coalgebras-and-comodules-subsecn});
and cofree left $C$\+comodules have the form $C\ot_k V$, where $V$
ranges over (graded) $k$\+vector spaces; so cofree $C$\+comodules
obviously form a class closed under infinite direct sums.

 Similarly to Section~\ref{co-contra-derived-cdg-modules-subsecn},
we denote by $\Hot(C^\cu\comodl_\inj)\subset\Hot(C^\cu\comodl)$
the full triangulated subcategory in the homotopy category formed by all
the CDG\+comodules whose underlying graded $C$\+comodules are injective.
 Such CDG\+comodules are called \emph{graded-injective}.

\begin{thm} \label{coderived-cdg-comodules-thm}
 Let $C^\cu=(C,d,h)$ be a CDG\+coalgebra over~$k$. 
 Then the composition\/ $\Hot(C^\cu\comodl_\inj)\rarrow
\Hot(C^\cu\comodl)\rarrow\sD^\co(C^\cu\comodl)$ of the triangulated
inclusion functor\/ $\Hot(C^\cu\comodl_\inj)\rarrow\Hot(C^\cu\comodl)$
and the Verdier quotient functor\/ $\Hot(C^\cu\comodl)\rarrow
\sD^\co(C^\cu\comodl)$ is an equivalence of triangulated categories,
$$
 \Hot(C^\cu\comodl_\inj)\simeq\sD^\co(C^\cu\comodl).
$$
\end{thm}

\begin{proof}
 This is~\cite[Theorem~4.4(c)]{Pkoszul}.
 For a generalization to exact DG\+categories,
see~\cite[Theorem~5.10(a)]{Pedg}.
\end{proof}

\begin{rem}
 For any DG\+ring $A^\bu=(A,d)$, the derived category of DG\+modules
$\sD(A^\bu\modl)$ is compactly generated (in the sense
of~\cite{Neem0,Neem1}).
 In fact, it follows immediately from the definitions that the free
DG\+module $A^\bu$ over $A^\bu$ is a single compact generator of
$\sD(A^\bu\modl)$ \,\cite[Section~3.5]{Kel2}.
 On the coalgebra side of Koszul duality, the related observation is
that, for any CDG\+coalgebra $C^\cu=(C,d,h)$ over~$k$, the coderived
category of CDG\+comodules $\sD^\co(C^\cu\comodl)$ is compactly
generated.
 The CDG\+comodules whose underlying graded $k$\+vector spaces are
(totally) finite-dimensional form a set of compact generators of
$\sD^\co(C^\cu\comodl)$ \,\cite[Section~5.5]{Pkoszul}.
 The latter result has many analogous versions and far-reaching 
generalizations, including the case of the coderived category of
CDG\+modules over a graded Noetherian
CDG\+ring~\cite[Section~3.11]{Pkoszul} (see
also~\cite[Proposition~1.5(d)]{EP}), the Becker coderived category
of a locally coherent abelian category~\cite[Corollary~6.13]{Sto},
the coderived category of a locally Noetherian abelian
DG\+category~\cite[Theorem~9.23]{Pedg}, and the coderived category
of a locally coherent abelian DG\+category under a certain assumption
of ``finite fp\+projective dimension''~\cite[Theorem~9.39]{Pedg}.
\end{rem}

\subsection{History, II, and conclusion}
 Let me try to explain my original (end of March~1999) motivation
for introducing Definitions~\ref{coderived-cdg-modules-definition}
and~\ref{coderived-cdg-comodules-definition}.
 The following remark, purporting to serve as the explanation, is
written in the notation of twisting cochains, because it is very
convenient; though back in Spring~99 I~was not familiar with
twisting cochains.
 Rather, I~was thinking in terms of Koszul duality functors arising
in the context of bar- and cobar-constructions, such as in
Theorems~\ref{complexes-of-modules-augmented-bar-construction-thm},
\ref{complexes-of-comodules-coaugmented-cobar-thm},
and~\ref{complexes-of-comodules-conilpotent-cobar-thm}.

\begin{rem} \label{motivation-coacyclics-to-contractibles-remark}
 The trouble with Koszul duality functors is that they can take
acyclic complexes or DG\+comodules to nonacyclic ones, as illustrated
by the example in Remark~\ref{divergent-spectral-sequence-remark}.
 By contrast, one can notice that the Koszul duality functors
mentioned above in this survey always take coacyclic objects to
\emph{contractible} ones!

 In full generality, let $B^\cu$ be a CDG\+algebra, $C^\cu$ be
a CDG\+coalgebra, and $\tau\:C^\cu\rarrow B^\cu$ be a twisting cochain.
 Then the CDG\+module $B^\cu\ot^\tau N^\cu$ over $B^\cu$ is
contractible for any coacyclic CDG\+comodule $N^\cu$ over~$C^\cu$.
 Similarly, the CDG\+comodule $C^\cu\ot^\tau M^\cu$ over $C^\cu$ is
contractible for any coacyclic CDG\+module $M^\cu$ over~$B^\cu$.

 Indeed, let us consider, e.~g., a short exact sequence $0\rarrow
K^\cu\rarrow L^\cu\rarrow M^\cu\rarrow0$ of CDG\+comodules
over~$C^\cu$.
 Notice that, for any CDG\+comodule $N^\cu$ over $C^\cu$,
the underlying graded $B$\+module structure of the CDG\+module
$B^\cu\ot^\tau N^\cu$ does not depend either on the differential
or on the graded $C$\+comodule structure on $N^\cu$, but only on
the underlying graded vector space of~$N^\cu$.
 But as a short exact sequence of graded vector spaces, the sequence
$0\rarrow K^\cu\rarrow L^\cu\rarrow M^\cu\rarrow0$ is split.
 Consequenly, the induced short sequence of CDG\+modules
$0\rarrow B^\cu\ot^\tau K^\cu\rarrow B^\cu\ot^\tau L^\cu\rarrow
B^\cu\ot^\tau M^\cu\rarrow0$ over $B^\cu$ is not only exact, but its
underlying short exact sequence of graded $B$\+modules is even
\emph{split exact}.

 Now one can easily see that, for any short exact sequence of
CDG\+modules over $B^\cu$ that is split exact as a short exact
sequence of graded $B$\+modules, the corresponding total CDG\+module
over $B^\cu$ is contractible.
 Obviously, one has $B^\cu\ot^\tau\Tot(K^\cu\to L^\cu\to M^\cu)
=\Tot(B^\cu\ot^\tau K^\cu\to B^\cu\ot^\tau L^\cu\to
B^\cu\ot^\tau M^\cu)$.
 Thus the CDG\+module $B^\cu\ot^\tau\Tot(K^\cu\to L^\cu\to M^\cu)$
over $B^\cu$ is contractible. 
\end{rem}

 While my work on derived categories of the second kind remained
unwritten until 2007 or~09, other people were approaching
the same or almost the same concepts from various angles.
 Filtered quasi-isomorphisms of DG\+coalgebras were introduced
by Hinich~\cite{Hin2} back in 1998, and the same approach was
extended to DG\+comodules by K.~Lef\`evre-Hasegawa~\cite{Lef}
in~2003.
 In the same year, the term \emph{coderived category}, together
with its first definition for DG\+comodules, appeared in
Keller's note~\cite{Kel}.

 The study of the homotopy category of unbounded complexes of
projective modules was initiated by J\o rgensen~\cite{Jor} in~2003,
and for the homotopy category of unbounded complexes of injective
objects in a locally Noetherian Grothendieck abelian category this
was done by Krause~\cite{Kra} in~2004.
 Later this line of research was taken up and developed by Neeman
in~\cite{Neem2,Neem3} and St\!'ov\'\i\v cek in~\cite{Sto}.

 In the meantime, essentially the same constructions
as in my Definitions~\ref{absolute-derived-cdg-modules-definition}\+-%
\ref{contraderived-cdg-modules-definition}, of what they called
``derived categories'' in quotes, were arrived at by
Keller, Lowen, and Nicol\'as~\cite{KLN} in~2006, but their results
remained unavailable to the public until after my seminar talk
in Paris in April~2009.
 The first \texttt{arXiv} version of my research monograph~\cite{Psemi},
where Definitions~\ref{coderived-cdg-modules-definition}\+-%
\ref{contraderived-cdg-modules-definition} were spelled out in
the context of exact categories, appeared in~2007; and the first
\texttt{arXiv} version of the memoir~\cite{Pkoszul}, where these
definitions were presented and studied for CDG\+modules, CDG\+comodules,
and CDG\+contramodules, became available in~2009.

 To summarize, the theory of derived categories of the second kind
identifies three constructions, showing that they produce one and
the same triangulated category.
 Here $C^\cu=(C,d,h)$ is a conilpotent CDG\+coalgebra over a field~$k$:
\begin{enumerate}
\item the coderived category of left CDG\+comodules over $C^\cu$
defined as in the note~\cite[Section~4]{Kel}, using
the cobar-construction in order to pass to DG\+modules over
the DG\+algebra $\Cb^\bu_\gamma(C^\cu)$;
\item the homotopy category of graded-injective left CDG\+comodules
over $C^\cu$ (that is, CDG\+comodules with injective underlying graded
$C$\+comodules), which is the CDG\+comodule\ version of the homotopy
category of complexes of injective modules as
in~\cite[Section~2]{Kra};
\item the coderived category of left CDG\+comodules over $C^\cu$
as per Definition~\ref{coderived-cdg-comodules-definition}, which
is the comodule version of~\cite[Section~3.1\,(A1\+-A2)]{KLN}.
\end{enumerate}

 The equivalence of~(1) and~(3) is a corollary of
Theorem~\ref{augmented-acyclic-twisting-cochain-duality-thm}
(for DG\+coalgebras)
or Theorem~\ref{nonaugmented-acyclic-twisting-cochain-duality-thm}
(in full generality).
 The equivalence of~(2) and~(3) does not depend on the conilpotency
assumption and holds for any CDG\+coalgebra $C^\cu$ over~$k$;
this is Theorem~\ref{coderived-cdg-comodules-thm}.

\subsection{Becker's derived categories of the second kind}
 From our point of view, the most important development in the theory
of derived categories of the second kind after~\cite{Psemi,Pkoszul} was
Becker's paper~\cite{Bec}.
 The specific result which is presumed here,
\cite[Proposition~1.3.6]{Bec}, is stated in the language of abelian
model structures; so we will translate it below into the perhaps more
familiar language of triangulated categories.
 This is an extension of the approach of~\cite{Jor,Kra,Neem2,Neem3}
from complexes of modules into the CDG\+module realm.

 After some hesitation, we dared to put in writing
in~\cite[Remark~9.2]{PS4} the terminology of \emph{derived categories
of the second kind in the sense of Positselski} vs.\
\emph{derived categories of the second kind in the sense of Becker}.
 The definitions in
Sections~\ref{co-contra-derived-cdg-modules-subsecn}\+-%
\ref{coderived-cdg-comodules-subsecn} are those of derived categories
of the second kind in my sense.

 We are following the expositions in~\cite[Sections~7 and~9]{PS4} (in
the case of abelian categories) and~\cite[Sections~4.2\+-4.3]{Pctrl}
(in the case case of CDG\+modules).
 Let $B^\cu$ be a CDG\+ring.
 A left CDG\+module $X^\cu$ over $B^\cu$ is said to be
\emph{contraacyclic in the sense of Becker} if the complex of abelian
groups $\Hom_B^\bu(P^\cu,X^\cu)$ is acyclic for any graded-projective
left CDG\+module $P^\cu$ over~$B^\cu$.
 The thick subcategory of Becker-con\-traacyclic CDG\+modules in
the homotopy category is denoted by $\Ac^\bctr(B^\cu\modl)\subset
\Hot(B^\cu\modl)$; and the related triangulated Verdier quotient
category
$$
 \sD^\bctr(B^\cu\modl)=\Hot(B^\cu\modl)/\Ac^\bctr(B^\cu\modl)
$$
is called the \emph{contraderived category of left CDG\+modules
over $B^\cu$ in the sense of Becker}.

 Dually, a left CDG\+module $Y^\cu$ over $B^\cu$ is said to be
\emph{coacyclic in the sense of Becker} if the complex of abelian
groups $\Hom_B^\bu(Y^\cu,J^\cu)$ is acyclic for any graded-injective
left CDG\+module $J^\cu$ over~$B^\cu$.
 The thick subcategory of Becker-coacyclic CDG\+modules in
the homotopy category is denoted by $\Ac^\bco(B^\cu\modl)\subset
\Hot(B^\cu\modl)$; and the related triangulated Verdier quotient
category
$$
 \sD^\bco(B^\cu\modl)=\Hot(B^\cu\modl)/\Ac^\bco(B^\cu\modl)
$$
is called the \emph{coderived category of left CDG\+modules over
$B^\cu$ in the sense of Becker}.

 The result of~\cite[Theorem~3.5]{Pkoszul} tells that
\emph{co/contraacyclicity in the sense of Positselski implies
co/contraacyclicity in the sense of Becker}.
 Thus Becker's co/contraderived categories are ``nonstrictly smaller''
than mine.

 Becker's result~\cite[Proposition~1.3.6]{Bec} tells (or rather,
implies) that the compositions of triangulated functors
$$
 \Hot(B^\cu\modl_\proj)\lrarrow\Hot(B^\cu\modl)\lrarrow
 \sD^\bctr(B^\cu\modl)
$$
and
$$
 \Hot(B^\cu\modl_\inj)\lrarrow\Hot(B^\cu\modl)\lrarrow
 \sD^\bco(B^\cu\modl)
$$
are triangulated equivalences
$$
 \sD^\bctr(B^\cu\modl)\simeq\Hot(B^\cu\modl_\proj)
 \quad\text{and}\quad
 \sD^\bco(B^\cu\modl)\simeq\Hot(B^\cu\modl_\inj)
$$
for any CDG\+ring~$B^\cu$.
 Thus one can say that ``the Becker analogue of
Theorem~\ref{co-contra-derived-cdg-modules} holds for any CDG\+ring''.

 On the other hand, with~\cite[Proposition~1.3.6]{Bec} in mind, one
can interpret Theorem~\ref{co-contra-derived-cdg-modules} as telling
that \emph{under the assumption of condition~\textup{($*$)}
or~\textup{(${*}{*}$)}, Becker's and Positselski's derived categories
of the second kind agree}.
 It is an \emph{open problem} whether they agree for an arbitrary
CDG\+ring; in fact, this is not known even for complexes of modules
over a ring (see~\cite[Examples~2.5(3) and~2.6(3)]{Pps} for
a discussion).

 Positselski's and Becker's derived categories of the second kind
are known to agree for CDG\+comodules or CDG\+contramodules over
a CDG\+coalgebra over a field (see
Theorem~\ref{coderived-cdg-comodules-thm} above
and Theorem~\ref{contraderived-cdg-contramodules-thm} below).
 It is in this sense that one says that ``the coderived categories
are (known to be) better behaved for comodules, and the contraderived
categories for contramodules''.

 The definitions in~\cite{KLN,Psemi,Pkoszul,PP2,EP,Pps,PSch,Prel,Pedg}
are those of co/contraderived categories in the sense of
Positselski.
 The definitions in~\cite{Jor,Kra,Neem2,Bec,Neem3,Sto,Pctrl,PS4}
are those of co/contraderived categories in the sense of Becker.

\Section{Contramodules over Coalgebras}

 Contramodules are dual analogues of comodules.
 To any coalgebra (over a field), one can assign the (abelian)
categories of left and right contramodules over it, alongside with
the comodule categories.
 The structure theory of contramodules over a coalgebra over a field
is more complicated, but not too much more complicated, than
the structure theory of comodules.
 The historical obscurity/neglect of contramodules is responsible for
the popular misconception that injective objects are much more common
than projective objects in ``naturally appearing'' abelian categories.
 We suggest the survey paper~\cite{Prev} as the standard reference
source on contramodules.

\subsection{The basics of contramodules}
\label{basics-of-contramodules-subsecn}
 The definitions of coalgebras and comodules (as in
Section~\ref{coalgebras-and-comodules-subsecn}) are obtained by
writing down the definitions of algebras and modules in the tensor
notation and inverting the arrows.
 In order to arrive to the definition of a contramodule, it remains
to notice that there are \emph{two} equivalent ways to express
the definition of a module in a tensor (or Hom) notation.

 An algebra over a field~$k$ is a vector space $A$ endowed with
linear maps $m\:A\ot_kA\rarrow A$ (the multiplication map) and
$e\:k\rarrow A$ (the unit map).
 A left $A$\+module $M$ is a $k$\+vector space endowed with a linear
map $n\:A\ot_kM\rarrow M$ (the action map).
 The usual associativity and unitality axioms need to be imposed.
 But there is a different way to spell out the definition of a module.

 A left $A$\+module $M$ is the same thing as a $k$\+vector space
endowed with a $k$\+linear map $p\:M\rarrow\Hom_k(A,M)$.
 This is just an expression of the tensor-Hom adjunction: the maps $n$
and~$p$ are connected by the rule $p(x)(a)=n(a\ot x)\in M$ for all
$x\in M$ and $a\in A$.
 We leave it to the reader to write down the associativity and
unitality axioms for an $A$\+module $M$ in terms of the map~$p$.

 The definition of a \emph{contramodule} over a coalgebra $C$ over
a field~$k$ is obtained by inverting the arrows in the definition of
a module given in terms of the map~$p$.
 We refer to~\cite[Section~2.6]{BBW} for a relevant discussion of
monad-comonad and comonad-monad adjoint pairs.

 Specifically, let $C$ be a (coassociative and counital) coalgebra
over a field~$k$.
 A \emph{left $C$\+contramodule} $P$ is a $k$\+vector space endowed
with a $k$\+linear map of \emph{left $C$\+contraaction}
$$
 \pi\:\Hom_k(C,P)\lrarrow P
$$
satisfying the following \emph{contraassociativity} and
\emph{contraunitality} axioms.
 Firstly, the two compositions
$$
 \Hom_k(C,\Hom_k(C,P))\simeq\Hom_k(C\ot_kC,\>P)\rightrightarrows
 \Hom_k(C,P)\rarrow P
$$
must be equal to each other, $\pi\circ\Hom_k(C,\pi)=
\pi\circ\Hom_k(\mu,P)$.
 Secondly, the composition
$$
 P\rarrow\Hom_k(C,P)\rarrow P
$$
must be equal to the identity map, $\pi\circ\Hom_k(\epsilon,P)=\id_P$
\,\cite[Section~1.1]{Prev}.
 Here, as in Section~\ref{coalgebras-and-comodules-subsecn},
\,$\mu\:C\rarrow C\ot_kC$ and $\epsilon\:C\rarrow k$ are
the comultiplication and counit maps of the coalgebra~$C$.

 In the definition of a \emph{left} contramodule, the identification
$\Hom_k(C,\Hom_k(C,P))\simeq\Hom_k(C\ot_kC,\>P)$ is obtained as
a particular case of the adjunction isomorphism $\Hom_k(U,\Hom_k(V,W))
\simeq\Hom_k(V\ot_kU,\>W)$, where $U$, $V$, and $W$ are arbitrary
$k$\+vector spaces.
 In the definition of a \emph{right} contramodule, the isomorphism
$\Hom_k(V,\Hom_k(U,W))\simeq\Hom_k(V\ot_kU,\>W)$ is presumed.

 For any right $C$\+comodule $N$ and any $k$\+vector space $V$,
the vector space $\Hom_k(N,V)$ has a natural structure of left
$C$\+contramodule (see~\cite[Section~1.2]{Prev} for the details).
 The left $C$\+contramodule $\Hom_k(C,V)$ is called the \emph{free}
left $C$\+contramodule generated by the vector space~$V$.

 For any left $C$\+contramodule $Q$, the $k$\+vector space of all
left $C$\+contramodule maps $\Hom_k(C,V)\rarrow Q$ is naturally
isomorphic to the $k$\+vector space of all $k$\+linear maps
$V\rarrow Q$,
$$
 \Hom^C(\Hom_k(C,V),Q)\simeq\Hom_k(V,Q).
$$
 Hence the ``free contramodule'' terminology.

\subsection{Duality-analogy of comodules and contramodules}
\label{duality-analogy-subsecn}
 The duality-analogy (or ``covariant duality'') between comodules and
contramodules is a remarkable yet unfamiliar phenomenon.
 It can be expressed by saying that the categories of comodules and
contramodules look as though they were opposite categories---up to
a point.
 In fact, \emph{no} category of contramodules is ever the opposite
category to a category of comodules; for example, over the coalgebra
$C=k$, both the categories of $C$\+comodules and $C$\+contramodules
coincide with the category of $k$\+vector spaces (which is certainly
\emph{not} equivalent to its opposite category).
 But the \emph{analogy} is striking.

 Let $C$ be a coalgebra over~$k$.
 Then \emph{both} the category of left $C$\+comodules $C\comodl$ and
the category of left $C$\+contramodules $C\contra$ are abelian.

 The abelian category $C\comodl$ has enough injective objects;
in fact, the injective comodules are precisely the direct summands of
the cofree comodules (as defined in
Section~\ref{coalgebras-and-comodules-subsecn}).
 The abelian category $C\contra$ has enough projective objects;
in fact, the projective contramodules are precisely the direct summands
of the free contramodules (as defined in
Section~\ref{basics-of-contramodules-subsecn})
\,\cite[Section~1.2]{Prev}.

 The forgetful functor $C\comodl\rarrow k\vect$ from the category of
$C$\+comodules to the category of $k$\+vector spaces is exact and
preserves infinite coproducts (but \emph{not} infinite products).
 Consequently, the coproduct functors in $C\comodl$ are exact
(moreover, so are the functors of filtered colimit).

 The forgetful functor $C\contra\rarrow k\vect$ is exact and preserves
infinite products (but \emph{not} infinite coproducts).
 Consequently, the product functors in $C\contra$ are exact (but
the filtered limits are \emph{not} exact in $C\contra$, of course,
as they are not exact already in $k\vect$).

 A wide class of abelian categories to which the categories of comodules
belong is called the class of \emph{Grothendieck categories}.
 A wide class of abelian categories to which the categories of
contramodules belong is called the class of \emph{locally presentable
abelian categories with enough projective objects}.
 The main reference source on the latter class of abelian categories
is the preprint~\cite{Pper} (see also~\cite[Section~6]{PS1}).
 The duality-analogy (or ``covariant duality'') between these two
classes of abelian categories is emphasized in the paper~\cite{PS4}
(cf.~\cite{PS1,PS2}).

\subsection{Comodules and contramodules over power series in
one variable}  \label{power-series-in-one-variable-subsecn}
 The aim of this section is to introduce the intuition of contramodules
as \emph{modules with infinite summation operations}.

 The dual vector space to an infinite-dimensional (discrete) vector
space comes endowed with a natural topology; such topological vector
spaces are known as \emph{linearly compact} or \emph{pseudocompact}
(or \emph{pro-finite-dimensional}).
 In particular, the dual vector spaces to infinite-dimensional
coalgebras are topological algebras.
 In fact, the category of coalgebras is anti-equivalent to the category
of linearly compact topological algebras (see the discussion
in~\cite[first paragraph of Section~1.3]{Prev}). 

 So coalgebras can be identified by the names of their dual topological
algebras.
 In particular, let $kz^*$ be a one-dimensional vector space with
the basis vector~$z^*$.
 Consider the tensor coalgebra $C=\udT(kz^*)$, as defined in
Section~\ref{coalgebra-structure-cobar-subsecn}.
 Then the dual topological algebra to $C$ is the algebra $C^*=k[[z]]$
of formal Taylor power series in the variable~$z$ (cf.\
Remark~\ref{no-isos-of-bar-induced-by-changes-of-connection}).

 Let $C$ be a coalgebra over~$k$.
 As explained in Section~\ref{cdg-coalgebras-subsecn}, any comodule
over a coalgebra $C$ is a module over the algebra~$C^*$.
 Similarly, any contramodule over $C$ is also a module over~$C^*$.
 The composition
$$
 C^*\ot_kP\lrarrow\Hom_k(C,P)\lrarrow P
$$
of the natural injective map $C^*\ot_k P\rarrow\Hom_k(C,P)$ with
the contraaction map $\pi\:\Hom_k(C,P)\rarrow P$ defines an action
of $C^*$ in~$P$.

 We would like to describe comodules and contramodules over
the coalgebra $C=\udT(kz^*)$.
 Namely, a $C$\+comodule $M$ is the same thing as a $k[z]$\+module with
a \emph{locally nilpotent} action of the operator~$z$.
 This means that for every element $m\in M$ there exists an integer
$n\ge1$ such that $z^nm=0$ in~$M$.
 The coaction map $M\rarrow C\ot_kM$ is given in terms of
the $k[z]$\+module structure on $M$ by the formula
$$
 m\longmapsto\sum\nolimits_{n=0}^\infty z^*{}^{\ot n}\ot z^nm
 \qquad\text{for all $m\in M$},
$$
and the condition of local nilpotence of the action of~$z$ in $M$
comes from the condition that the sum in the right-hand side must be
finite (if it is to define an element of the tensor product $C\ot_kM$).

 A $C$\+contramodule $P$ is the same thing as a $k$\+vector space with
the following \emph{$z$\+power infinite summation operation}.
 To every sequence of elements $p_0$, $p_1$, $p_2$~\dots~$\in P$,
an element denoted formally by $\sum_{n=0}^\infty z^np_n\in P$ is
assigned.
 The following axioms must be satisfied:
$$
 \sum\nolimits_{n=0}^\infty z^n(ap_n+bq_n)=
 a\sum\nolimits_{n=0}^\infty z^np_n +
 b\sum\nolimits_{n=0}^\infty z^nq_n
 \quad\text{for all $p_n$, $q_n\in P$, \ $a$, $b\in k$}
$$
(\emph{linearity}),
$$
 \sum\nolimits_{n=0}^\infty z^np_n=p_0
 \quad\text{if $p_1=p_2=\dotsb=0$ in $P$}
$$
(\emph{contraunitality}), and
$$
 \sum\nolimits_{i=0}^\infty z^i\sum\nolimits_{j=0}^\infty z^jp_{ij}
 = \sum\nolimits_{n=0}^\infty z^n\sum\nolimits_{i,j\ge0}^{i+j=n}p_{ij}
$$
for all $p_{ij}\in P$ (\emph{contraassociativity}).
 In the latter equation, the first three summation signs denote
the $z$\+power infinite summation operation, while the forth one means
a finite sum of elements in~$P$ \,\cite[Section~1.3]{Prev}.

 The intuition of contramodules as modules with infinite summation
operations is axiomatized in the concepts of contramodules over
a commutative ring with a finitely generated
ideal~\cite[Sections~3\+-4]{Pcta}, \cite[Section~1]{Pdc},
contramodules over topological rings and topological associative
algebras~\cite[Sections~2.1 and~2.3]{Prev}, \cite[Section~6]{PS1},
contramodules over topological Lie algebras~\cite[Sections~1.7
and~2.4]{Prev}, etc.

 One special property of the coalgebra $C=\udT(kz^*)$ is that
a $C$\+contramodule structure can be uniquely recovered from
its underlying $k[[z]]$\+module structure.
 In fact, the forgetful functor to the category of modules over
the polynomial algebra $C\contra\rarrow k[z]\modl$ is already fully
faithful~\cite[Remark~A.1.1]{Psemi}, \cite[Theorem~3.3]{Pcta},
\cite[Theorem~B.1.1]{Pweak}.
 This is \emph{not} true for an arbitrary coalgebra $C$, of course;
still this is true for any \emph{finitely cogenerated} conilpotent
coalgebra~$C$ \,\cite[Theorem~2.1]{Psm}.
 See~\cite[Section~3.8]{Prev} for a discussion of far-reaching
generalizations.

\subsection{Nonseparated contramodules}
\label{nonseparated-contramodules-subsecn}
 The duality-analogy between comodules and contramodules has some
limitations, though.
 The examples of \emph{nonseparated contramodules} demonstrate
such limitations.

 Let $f\:C\rarrow D$ be a homomorphism of coalgebras.
 Then, as mentioned in
Section~\ref{brief-remarks-about-coalgebras-subsecn}, every
$C$\+comodule acquires a $D$\+comodule structure.
 Similarly, every $C$\+contramodule acquires a $D$\+contramodule
structure.
 The resulting functor $C\contra\rarrow D\contra$ is called
the \emph{contrarestriction of scalars} with respect to~$f$.

 In particular, let $E\subset C$ be a subcoalgebra.
 Then the corestriction and contrarestriction of scalars are
fully faithful functors
\begin{align*}
 E\comodl &\lrarrow C\comodl, \\
 E\contra &\lrarrow C\contra.
\end{align*}

 The functor $E\comodl\rarrow C\comodl$ has a right adjoint functor
assigning to every left $C$\+comodule $M$ its maximal subcomodule
${}_EM\subset M$ whose $C$\+comodule structure comes from
an $E$\+comodule structure.
 The $E$\+comodule ${}_EM$ can be computed as the kernel of
the composition of maps $M\rarrow C\ot_k M\rarrow C/E\ot_kM$.

 The functor $E\contra\rarrow C\contra$ has a left adjoint functor
assigning to every left $C$\+contramodule $P$ its maximal
quotient contramodule ${}^E\!P$ whose $C$\+contramodule structure
comes from an $E$\+contramodule structure.
 The $E$\+contramodule ${}^E\!P$ can be computed as the cokernel of
the composition of maps $\Hom_k(C/E,P)\rarrow\Hom_k(C,P)\rarrow P$.

 According to Lemma~\ref{coalgebras-comodules-loc-fin-dim-lemma}, any
$C$\+comodule $M$ is the union of its finite-dimensional subcomodules,
and each of these is a comodule over a finite-dimensional subcoalgebra
in~$C$.
 Therefore, one has
$$
 M=\bigcup\nolimits_{E\subset C} {}_EM,
$$
where the union is taken over all the finite-dimensional subcoalgebras
$E\subset C$.

 The dual assertion for contramodules in \emph{not} true.
 Specifically, for any $C$\+con\-tramodule $P$ one can consider
the natural map
\begin{equation} \label{map-to-inverse-limit-over-E}
 P\lrarrow\varprojlim\nolimits_{E\subset C} {}^E\!P,
\end{equation}
where the inverse limit is taken over all the finite-dimensional
subcoalgebras $E\subset C$.
 The map~\eqref{map-to-inverse-limit-over-E} is \emph{not} injective
in general.

 In fact, there exists an infinite-dimensional coalgebra $C$ and
a \emph{two-dimensional} $C$\+contramodule $P$ for which
$\varprojlim_{E\subset C} {}^E\!P$ is a one-dimensional
$C$\+contramodule.
 So the $C$\+contramodule structure on $P$ does \emph{not} come from
a contramodule structure over any finite-dimensional subcoalgebra
in $C$, and the intersection of the kernels of the maps $P\rarrow
{}^E\!P$ is a one-dimensional subcontramodule in the two-dimensional
contramodule~$P$ \,\cite[Section~A.1.2]{Psemi}.

 Let us now consider the example of the coalgebra $C=\udT(kz^*)$ from
Section~\ref{power-series-in-one-variable-subsecn}.
 Then the map~\eqref{map-to-inverse-limit-over-E} is simply the
natural map from $P$ to the $z$\+adic completion of~$P$,
\begin{equation} \label{map-to-z-adic-completion}
 P\lrarrow\varprojlim\nolimits_{n\ge1} P/z^nP.
\end{equation}
 The map~\eqref{map-to-z-adic-completion} is known to be
\emph{surjective} for any $C$\+contramodule~$P$
\,\cite[Lemma~A.2.3]{Psemi}, \cite[Theorems~3.3 and~5.6]{Pcta},
\cite[Section~1]{Pdc}.
 But there exist $C$\+contramodules $P$ for which 
the map~\eqref{map-to-z-adic-completion} is \emph{not injective}.
 In other words, \emph{any $C$\+contramodule is $z$\+adically
complete, but it need not be $z$\+adically separated}.
 The now-classical counterexample can be found
in~\cite[Example~2.5]{Sim}, \cite[Section~A.1.1]{Psemi},
\cite[Example~3.20]{Yek}, \cite[Example~2.7(1)]{Pcta},
\cite[Section~1.5]{Prev}.

 In fact, the mentioned counterexample from~\cite{Sim,Psemi,Yek}
is a $C$\+contramodule $P$ with the following property.
 There exists a sequence of elements $p_0$, $p_1$, $p_2$,~\dots~$\in P$
such that $z^np_n=0$ for every $n\ge0$, but
$\sum_{n=0}^\infty z^np_n\ne0$ in~$P$.
 Here we use the $z$\+power infinite summation operation notation
from Section~\ref{power-series-in-one-variable-subsecn}.
 So \emph{the contramodule infinite summation operations cannot be
interpreted as any kind of limits of finite partial
sums}~\cite[Preface]{Psemi},
\cite[Section~0.2 of the Introduction]{Prev}.

\subsection{Contramodule Nakayama lemma and irreducible contramodules}
 As explained in Section~\ref{nonseparated-contramodules-subsecn},
the map~\eqref{map-to-inverse-limit-over-E} need not be injective.
 However, this map \emph{never} vanishes (for a nonzero
$C$\+contramodule~$P$).

 Dropping the contraunitality axiom from the definition of
a contramodule in Section~\ref{basics-of-contramodules-subsecn}, one
obtains the definition of a contramodule over a noncounital
coalgebra~$D$.
 The following result is called the \emph{contramodule Nakayama lemma}
(for contramodules over coalgebras over a field).

\begin{lem} \label{nakayama-lemma}
\textup{(a)} Let $D$ be a conilpotent noncounital coalgebra
(as defined in Section~\ref{conilpotent-coalgebras-subsecn}) and $P$
be a nonzero $D$\+contramodule.
 Then the contraaction map\/ $\pi\:\Hom_k(D,P)\rarrow P$ is \emph{not}
surjective. \par
\textup{(b)} Let $C$ be a (counital) coalgebra and $P$ be a nonzero
left $C$\+contramodule.
 Then there exists a finite-dimensional subcoalgebra $E\subset C$
such that ${}^E\!P\ne0$.
\end{lem}

\begin{proof}
 Part~(a) is~\cite[Lemma~A.2.1]{Psemi} (see~\cite[Lemma~1.3.1]{Pweak},
\cite[Lemmas~2.1 and~3.22]{Prev} and the references therein
for generalizations).
 To deduce part~(b) from part~(a), one has to use the structure theory
of coalgebras, or more specifically, the fact that the maximal
cosemisimple subcoalgebra $C^\ss\subset C$ is a direct sum of
finite-dimensional (cosimple) coalgebras and the quotient coalgebra
without counit $D=C/C^\ss$ is conilpotent
(see Section~\ref{cosemisimple-coalgebras-and-coradical-subsecn})
together with a description of contramodules over an infinite direct
sum of coalgebras~\cite[Lemma~A.2.2]{Psemi}.
\end{proof}

 It follows from Lemma~\ref{nakayama-lemma}(b) that any 
\emph{irreducible} $C$\+contramodule is finite-dimen\-sional (moreover,
it is a contramodule over a finite-dimensional subcoalgebra in~$C$).

 In fact, while there is, of course, no way to define a contramodule
structure on an arbitrary $C$\+comodule, there does exist a natural
way to define a left $C$\+contramodule structure on any
\emph{finite-dimensional} left
$C$\+comodule~\cite[Section~A.1.2]{Psemi}.
 Indeed, the functor $N\longmapsto N^*$ is an anti-equivalence between
the categories of finite-dimensional right and left $C$\+comodules,
and on the other hand, the dual vector space to a right $C$\+comodule
is a left $C$\+contramodule (as mentioned in
Section~\ref{basics-of-contramodules-subsecn}).
 This produces a fully faithful covariant functor from the category of
finite-dimensional left $C$\+comodules to the category of
finite-dimensional left $C$\+contramodules; a finite-dimensional left
$C$\+contramodule $P$ belongs to the essential image of this functor
if and only if its contramodule structure comes from
an $E$\+contramodule structure for some finite-dimensional subcoalgebra
$E\subset C$ (that is, $P={}^E\!P$).

 This covariant fully faithful functor restricts to a bijection between
the isomorphism classes of irreducible left $C$\+comodules and
irreducible left $C$\+contramodules.
 Both the sets of isomorphism classes of irreducibles are also naturally
bijective to the set of all cosimple subcoalgebras in $C$, of course.

\subsection{Contratensor product} \label{contratensor-product-subsecn}
 The formalisms of tensor product and Hom-type operations play a key
role in the book~\cite{Psemi} and the memoir~\cite{Pweak}.
 In the ring and module theory, the formalism of tensor/Hom operations
has only two such operations, viz., the tensor product and Hom of
(bi)modules.
 In the context of coalgebras, contramodules, etc., the tensor/Hom
operations are more numerous.

 In particular, the definition of the \emph{cotensor product} of 
comodules goes back, at least, to the paper~\cite{EM2}.
 On the other hand, the construction of the \emph{contratensor product}
of a comodule and a contramodule seems to have first appeared
in~\cite{Plet,Psemi}.

 Let $C$ be a coalgebra, $N$ be a right $C$\+comodule, and $P$ be a left
$C$\+contramodule.
 The \emph{contratensor product} $N\ocn_CP$ is a $k$\+vector space
constructed as the cokernel of the difference of two natural maps
$$
 N\ot_k\Hom_k(C,P)\rightrightarrows N\ot_k P.
$$
 Here the first map $N\ot_k\Hom_k(C,P)\rarrow N\ot_kP$ is induced by
the contraaction map $\pi\:\Hom_k(C,P)\rarrow P$, while the second
map is the composition $N\ot_k\Hom_k(C,P)\rarrow N\ot_kC\ot_k\Hom_k(C,P)
\rarrow N\ot_k P$ of the map induced by the coaction map
$\nu\:N\rarrow N\ot_kC$ and the map induced by the natural evaluation
map $C\ot_k\Hom_k(C,P)\rarrow P$ \,\cite[Section~3.1]{Prev}.

 For any right $C$\+comodule $N$ and any $k$\+vector space~$V$,
there is a natural isomorphism of vector spaces
$$
 N\ocn_C\Hom_k(C,V)\simeq N\ot_k V.
$$
 For any right $C$\+comodule $N$, any left $C$\+contramodule $P$, and
any vector space~$V$, there is a natural isomorphism of vector spaces
$$
 \Hom^C(P,\Hom_k(N,V))\simeq\Hom_k(N\ocn_CP,\>V),
$$
where, as in Section~\ref{basics-of-contramodules-subsecn},
\,$\Hom^C$ denotes the space of morphisms in the category of
left $C$\+contramodules.

 For a discussion of a generalization of the construction of
contratensor product to topological rings, see~\cite[Section~7.2]{PS1}
(some further discussion and references can be found
in~\cite[end of Section~3.3]{Prev}).

\subsection{Underived co-contra correspondence}
\label{underived-co-contra-subsecn}
 Let $C$ be a (coassociative, counital) coalgebra over a field~$k$.
 As mentioned in Section~\ref{duality-analogy-subsecn}, there are
enough injective objects in the abelian category of left $C$\+comodules,
and these injectives are precisely the direct summands of the cofree
left $C$\+comodules $C\ot_kV$ (where $V$ ranges over the $k$\+vector
spaces).
 Similarly, there are enough projective objects in the abelian
category of left $C$\+contramodules, and these projectives are precisely
the direct summands of the free left $C$\+contramodules $\Hom_k(C,V)$.

 It turns out that there is a natural equivalence between the additive
categories of injective left $C$\+comodules and projective left
$C$\+contramodules,
\begin{equation} \label{underived-co-contra-equivalence}
 C\comodl_\inj\simeq C\contra_\proj.
\end{equation}

 The simplest way to construct the category
equivalence~\eqref{underived-co-contra-equivalence} is to define it on
cofree comodules and free contramodules by the rule
$C\ot_k V\longleftrightarrow\Hom_k(C,V)$.
 Then one has to compute that the groups of morphisms agree,
\begin{multline*}
 \Hom_C(C\ot_k U,\>C\ot_kV)\simeq\Hom_k(C\ot_kU,\>V) \\ \simeq
 \Hom_k(U,\Hom_k(C,V))\simeq\Hom^C(\Hom_k(C,U),\,\Hom_k(C,V))
\end{multline*}
for any $k$\+vector spaces $U$ and $V$, in view of the descriptions
of morphisms into a cofree comodule and from a free contramodule
in Sections~\ref{coalgebras-and-comodules-subsecn}
and~\ref{basics-of-contramodules-subsecn}.
 A category-theoretic version of this argument (an equivalence of
Kleisli categories for adjoint comonad-monad pairs) is discussed
in~\cite[Section~2.6(2)]{BBW}.

 Another way to obtain the additive category
equivalence~\eqref{underived-co-contra-equivalence} is to construct
a pair of adjoint functors
\begin{equation} \label{underived-co-contra-adjoint-pair}
 \Hom_C(C,{-})\:C\comodl \leftrightarrows C\contra
 \,:\!C\ocn_C{-}
\end{equation}
and check that it restricts to an equivalence between the full
subcategories of injective comodules and projective contramodules.

 Here the right adjoint functor $\Hom_C(C,{-})$ is the Hom functor
in the comodule category $C\comodl$.
 Taking the Hom eats up the left $C$\+comodule structure on $C$, but
the right $C$\+comodule structure on $C$ stays and induces a left
$C$\+contramodule structure on the Hom space, essentially as
explained in Section~\ref{basics-of-contramodules-subsecn}.
 The left adjoint functor $C\ocn_C{-}$ is the functor of contratensor
product.
 Taking the contratensor product consumes the right $C$\+comodule
structure on $C$, but the left $C$\+comodule structure on $C$ remains
and induces a left $C$\+comodule structure on the contratensor
product (cf.~\cite[Sections~1.2 and~3.1]{Prev}).

 For much more advanced discussions of the philosophy of
comodule-contramodule correspondence, see the introductions to
the papers~\cite{Pmgm,Pps}.
 One important early work related to the co-contra correspondence
is~\cite[Section~4]{IK}.
 From the contemporary point of view, the co-contra correspondence
can be thought of as a particular case of the \emph{tilting-cotilting
correspondence}~\cite{PS1,PS2}.
 A discussion of underived versions of the co-contra correspondence
containing various generalizations of the results above in this section
can be found in~\cite[Sections~3.4\+-3.6]{Prev}.

\subsection{History of contramodules}
 The notion of a contramodule (over a coalgebra over a commutative
ring) was invented by Eilenberg and Moore and defined \emph{on par}
with comodules in the 1965 memoir~\cite[Section~III.5]{EM1}.
 Besides~\cite{EM1}, two other papers on contramodules were published
in 1965 and 1970, among them the rather remarkable paper~\cite{Bar}.
 Then contramodules were all but forgotten for 30 to 40 years.

 In the meantime, what came to be known as ``Ext-$p$-complete'' or
``weakly-$l$-complete'' abelian groups~\cite{BouKan,Jan} or (in a later
terminology) ``cohomologically $I$\+adically complete'' or
``derived $I$\+adically complete modules'' (for a finitely generated
ideal $I$ in a commutative ring)~\cite{PSY} were defined and studied
by authors who remained apparently completely unaware of
the connection with Eilenberg and Moore's contramodules.
 We refer to the presentation~\cite{Psli3} for a detailed discussion
of this part of the history.

 As described in Section~\ref{second-kind-reminiscences-subsecn},
in Spring~1999 the present author went to the IAS library to look
for prior literature relevant to my March--April~99 discovery of
derived categories of the second kind and their role in derived
nonhomogeneous Koszul duality.
 Among other things, I~found the definition of a contramodule in
a paper copy of the memoir~\cite{EM1} which the IAS library had.

 In retrospect, I~could have realized already then that contramodules
form an ideal context for my definition of the contraderived category
(which I~already had in Spring~99).
 But I~didn't realize that, though I~was a bit unhappy about
the definition of the contraderived category (called ``the derived
category~$\sD''$\,'' at the time) having too little use.
 The definition of a contramodule felt strange.
 As many people had before and would later, I~remained unimpressed by
the definition of a contramodule in~1999, though I~kept a recollection
of it at the bottom of my memory.

 My view of contramodules changed in Summer~2000 when I~was working
on the first (2000) part of the ``Summer letters on semi-infinite
homological algebra''~\cite{Plet}.
 It was then that I~recalled the definition of a contramodule and
realized that contramodules are suited for use as the coefficients
for semi-infinite \emph{cohomology} (as opposed to semi-infinite
\emph{homology}) theories.
 The concept of comodule-contramodule correspondence was also arrived
at in~2000--02 and reflected in~\cite{Plet}.
 What is now called the contraderived category of contramodules
(together with its mixed counterpart, ``the semicontraderived category
of semicontramodules'') was first considered in~\cite{Plet}.

 The realization that derived nonhomogeneous Koszul duality has
a contramodule side, and should be formulated in the form of
a ``triality'' with the comodule side, the contramodule side,
and the co-contra correspondence forming three sides of a triangular
diagram of triangulated category equivalences, came to me by
mid-'00s.

 The letters~\cite{Plet} were eventually posted to the internet, where
some people could read them, including in particular T.~Brzezi\'nski
(as transliterated Russian is somewhat similar to Polish).
 This influenced his thinking, as exemplified by publications
such as~\cite{BBW}.
 My own account of contramodule theory first appeared in the August
2007 preprint version of the book~\cite{Psemi} and subsequenly in
the memoir~\cite{Pkoszul}.

 The September~2007 presentation of Brzezi\'nski on
contramodules~\cite{Bsli} featured the following statistics of
MathSciNet search hits:
\begin{itemize}
\item comodules = 797;
\item contramodules = 3.
\end{itemize}
 As I~am typing these lines (on July~19, 2022), the current statistics
of MathSciNet search hits (``Anywhere=(${-}$)'') is
\begin{itemize}
\item comodules = 1457;
\item contramodules = 30.
\end{itemize}
 From our point of view, contramodules (over coalgebras and corings)
should be treated on par and in parallel with comodules
in most contexts.

\Section{CDG-Contramodules}

 The contramodule side of the derived Koszul duality complements
the comodule side.
 Together with the derived comodule-contramodule correspondence,
they form a commutative triangle diagram of triangulated category
equivalences called the \emph{Koszul triality}.

\subsection{Graded contramodules}
 Let $C=\bigoplus_{n\in\boZ}C^n$ be a graded coalgebra.
 The definition of a \emph{graded $C$\+contramodule} is easily
formulated \emph{by analogy} with the definition of an ungraded
contramodule over an ungraded coalgebra in
Section~\ref{basics-of-contramodules-subsecn}.
 All one needs to do is to replace the usual ungraded notions of
the tensor product and Hom spaces with the ones intrinsic to the world
of graded vector spaces.

 So a \emph{graded left $C$\+contramodule} is a graded $k$\+vector
space $P$ endowed with a \emph{left contraaction map}
$\pi\:\Hom_k(C,P)\rarrow P$, where $\Hom_k(C,P)$ denotes the graded
Hom space (of homogeneous $k$\+linear maps $C\rarrow P$ of various
degrees $n\in\boZ$) and $\pi$~is a morphism of graded vector spaces
(i.~e., a homogeneous linear map of degree~$0$).
 The same contraassociativity and contraunitality axioms as in
Section~\ref{basics-of-contramodules-subsecn} are imposed.

 However, one can easily get confused trying to define what it means
to \emph{specify a grading} on a given ungraded $C$\+contramodule.
 The problem is that the conventional thinking of a grading as a direct
sum decomposition that can be either ignored or taken into account at
one's will is insufficient.
 Rather, one needs to think of the \emph{category} of graded vector
spaces and the forgetful functor from it to the category of
ungraded ones.
 The point is that there is \emph{more than one} such forgetful functor.
 \emph{Two} of them are important for the contramodule theory.

 A suitable formalism was suggested in~\cite[Section~11.1.1]{Psemi} and
then again in~\cite[Section~2.1]{Prel}.
 To a graded vector space $V$ one can assign \emph{two} ungraded
vector spaces, viz., $\Sigma V=\bigoplus_{n\in\boZ} V^n$ and
$\Pi V=\prod_{n\in\boZ} V^n$.
 Further possibilities include the \emph{Laurent summations}
$\bigoplus_{n<0}V^n\oplus\prod_{n\ge0} V^n$ and
$\prod_{n\le0}V^n\oplus\bigoplus_{n>0}V^n$ (cf.\ the beginning of
Section~\ref{co-contra-derived-cdg-modules-subsecn}), but for our
purposes we do not need these options.

 Given a graded module (or even a nonpositively or nonnegatively graded
ring), one can \emph{choose} between producing its underlying ungraded
module (respectively, ring) by taking the direct sum \emph{or}
the direct product of the grading components, that is the forgetful
functor $\Sigma$ or~$\Pi$.
 For coalgebras, comodules, and contramodules, there is no choice.
 The underlying ungraded coalgebra or comodule is constructed by
applying the direct sum forgetful functor~$\Sigma$.
 The underlying ungraded contramodule is obtained by applying
the direct product forgetful functor~$\Pi$.

 The reason for such preferences in respect to the functors of
forgetting the grading lies in the natural isomorphisms
\begin{gather*}
 \Sigma(U\ot_kV)\simeq \Sigma U\ot_k \Sigma V, \\
 \Pi\Hom_k(V,W)\simeq \Hom_k(\Sigma V,\Pi W),
\end{gather*}
which hold for all graded vector spaces $U$, $V$, and~$W$.
 Here the tensor product/Hom spaces in the left-hand sides of
the formulas are taken in the realm of graded vector spaces, while
in the right-hand sides of the formulas the tensor product/Hom
functors are applied to ungraded vector spaces.

 So, if $C$ is a graded coalgebra, $M$ is a graded $C$\+comodule, and
$P$ is a graded $C$\+contramodule, then $\Sigma M$ is an ungraded
$\Sigma C$\+comodule, while $\Pi P$ is an ungraded
$\Sigma C$\+contramodule.
 Thus specifying a grading on an ungraded $C$\+contramodule $Q$ means
decomposing $Q$ into a \emph{direct product} of its grading
components~\cite[Remark~2.2]{Pkoszul}.
 We refer to~\cite[Section and Remark~2.2]{Pkoszul} for a discussion
of the related sign rules.

\subsection{Definition of CDG-contramodules}
\label{cdg-contramodules-subsecn}
 Given a DG\+coalgebra $C^\bu$, one can spell out the notion of
a (\emph{left}) \emph{DG\+contramodule} $P^\bu$ over $C^\bu$ similarly
to the definition of a DG\+comodule in
Section~\ref{dg-coalgebras-subsecn}.
 All one needs to do is to transfer the definition of a contramodule
from Section~\ref{basics-of-contramodules-subsecn} to the world of
complexes of vector spaces.
 We skip the details, which can be found in~\cite[Section~2.3]{Pkoszul},
and pass to the more general case of CDG\+contramodules.

 Continuing the exposition from Section~\ref{cdg-coalgebras-subsecn},
we use the language of precomplexes of vector spaces.
 Given two precomplexes $V^\cu=(V,d_V)$ and $W^\cu=(W,d_W)$,
the \emph{Hom precomplex} $\Hom^\cu_k(V^\cu,W^\cu)$ is defined as
explained in Section~\ref{twisting-cochains-curved-subsecn}.

 Now we can define the notion of a \emph{contraderivation} of
a contramodule.
 Let $C$ be a graded coalgebra endowed with an odd coderivation
$d\:C\rarrow C$ of degree~$1$, as defined in
Section~\ref{cdg-coalgebras-subsecn}; and let $P$ be a graded
left $C$\+contramodule.
 Then an \emph{odd contraderivation} on $P$ \emph{compatible with}
the coderivation~$d$ on $C$ is a $k$\+linear map $d_P\:P\rarrow P$
of degree~$1$ such that the left contraaction map
$\pi\:\Hom_k(C,P)\rarrow P$ is a morphism of precomplexes.
 Here the differential~$d$ on $\Hom_k(C,P)$ is defined by the rule
from Section~\ref{twisting-cochains-curved-subsecn} in terms of
the differentials~$d$ on $C$ and~$d_P$ on~$P$.

 Furthermore, as explained in
Section~\ref{power-series-in-one-variable-subsecn}, any contramodule
$P$ over a coalgebra $C$ is naturally a module over the algebra~$C^*$.
 This construction has an obvious graded version: any graded
contramodule $P$ over a graded coalgebra $C$ is naturally endowed
with a graded module structure over the graded dual vector space $C^*$
to $C$, with the natural structure of graded algebra on~$C^*$.
 Similarly to the discussion in Section~\ref{cdg-coalgebras-subsecn},
one has to choose between two opposite ways of defining
the multiplication on~$C^*$.
 We prefer to choose the sides so that left $C$\+comodules become
left $C^*$\+modules; then left $C$\+contramodules also become left
$C^*$\+modules.

 Given a graded left contramodule $P$ over a graded coalgebra $C$,
a homogeneous element $p\in P$, and a homogeneous linear function
$b\:C\rarrow k$ we let $b*p\in P$ denote the result of the left
action of~$b$ on~$p$.
 We refer to~\cite[Section~4.1]{Pkoszul} for the sign rule.

 Now we can present our definition.
 Let $C^\cu=(C,d,h)$ be a CDG\+coalgebra over~$k$.
 A \emph{left CDG\+contramodule} $P^\cu=(P,d_P)$ over $C^\cu$ is
a graded left $C$\+contramodule endowed with
\begin{itemize}
\item an odd contraderivation $d_P\:P\rarrow P$ of degree~$1$
compatible with the coderivation~$d$ on~$C$
\end{itemize}
such that
\begin{enumerate}
\renewcommand{\theenumi}{\roman{enumi}}
\setcounter{enumi}{6}
\item the square of the differential~$d_P$ on $P$ is described by
the formula $d_P^2(p)=h*p$ for all $p\in P$.
\end{enumerate}

 Similarly to the theories of CDG\+modules and CDG\+comodules
(as in Sections~\ref{cdg-rings-subsecn}\+-\ref{cdg-coalgebras-subsecn}),
left CDG\+contramodules over a CDG\+coalgebra $C^\cu=(C,d,h)$ form
a DG\+cat\-e\-gory $C^\cu\contra$.
 Any morphism of CDG\+coalgebras $(f,a)\:(C,d_C,h_C)\rarrow
(D,d_D,h_D)$ induces a DG\+functor $C^\cu\contra\rarrow D^\cu\contra$
assigning to a CDG\+con\-tramodule $(P,d_P)$ the CDG\+contramodule
$(P,d'_P)$, with the graded $D$\+contramodule structure on $P$
obtained from the graded $C$\+contramodule structure on $P$ by
the contrarestriction of scalars (as mentioned in
Section~\ref{nonseparated-contramodules-subsecn}) and the twisted
differential~$d'_P$ given by the rule $d'_P(p)=d_P(p)+a*p$.
 So an isomorphism of CDG\+coalgebras induces an isomorphism of
the DG\+categories of CDG\+contramodules over them.

\subsection{Contraderived category of CDG-contramodules}
\label{contraderived-cdg-contramodules-subsecn}
 This section is a contramodule version of
Sections~\ref{co-contra-derived-cdg-modules-subsecn}\+-%
\ref{coderived-cdg-comodules-subsecn}.
 In the spirit of the discussion in the end of
Section~\ref{philosophy-of-second-kind-III-subsecn}, we will see
(in Theorem~\ref{contraderived-cdg-contramodules-thm}) that
the contraderived categories of contramodules are somewhat better
behaved than the contraderived categories of modules.
 On the other hand, it makes \emph{no} sense to consider
``coderived categories of contramodules'', as the functors of infinite
coproducts are usually \emph{not} exact in contramodule categories
(contramodule categories of homological dimension~$1$ being
a remarkable exception; see~\cite[Remark~1.2.1]{Pweak}). 

 Let $C^\cu=(C,d,h)$ be a CDG\+coalgebra over~$k$.
 Similarly to Sections~\ref{co-contra-derived-cdg-modules-subsecn}\+-%
\ref{coderived-cdg-comodules-subsecn}, one can speak of short exact
sequences $0\rarrow P^\cu\rarrow Q^\cu\rarrow R^\cu\rarrow0$ of left
CDG\+contramodules over $C^\cu$ and their totalizations (total
CDG\+contramodules) $\Tot(P^\cu\to Q^\cu\to R^\cu)$,
which are again left CDG\+contramodules over~$C^\cu$.

 Let $\Hot(C^\cu\contra)$ denote the homotopy category of left
CDG\+contramodules over a CDG\+coalgebra~$C^\cu$.
 Similarly to Sections~\ref{co-contra-derived-cdg-modules-subsecn}\+-%
\ref{coderived-cdg-comodules-subsecn}, the DG\+category $C^\cu\contra$
has shifts and cones, so its homotopy category $\Hot(C^\cu\contra)$
is triangulated.

\begin{defn}[{\cite[Section~4.2]{Pkoszul}}]
\label{absolute-derived-cdg-contramodules-definition}
 Let $C^\cu$ be a CDG\+coalgebra.
 A left CDG\+con\-tr\-amod\-ule over $C^\cu$ is said to be
\emph{absolutely acyclic} if it belongs to the minimal thick subcategory
of $\Hot(C^\cu\contra)$ containing the totalizations of short exact
sequences of left CDG\+contramodules over~$C^\cu$.
 The full subcategory of absolutely acyclic CDG\+contramodules is
denoted by $\Ac^\abs(C^\cu\contra)\subset\Hot(C^\cu\contra)$.
 The triangulated Verdier quotient category
$$
 \sD^\abs(C^\cu\contra)=\Hot(C^\cu\contra)/\Ac^\abs(C^\cu\contra)
$$
is called the \emph{absolute derived category} of left
CDG\+contramodules over~$C^\cu$.
\end{defn}

\begin{defn}[{\cite[Section~4.1]{Psemi}, \cite[Section~4.2]{Pkoszul}}]
\label{contraderived-cdg-contramodules-definition}
 A left CDG\+contramodule over $C^\cu$ is said to be
\emph{contraacyclic} if it belongs to the minimal triangulated
subcategory of $\Hot(C^\cu\contra)$ containing the totalizations of
short exact sequences of left CDG\+contramodules over~$C^\cu$
\emph{and closed under infinite products}.
 The thick subcategory of contraacyclic CDG\+contramodules is denoted
by $\Ac^\ctr(C^\cu\contra)\subset\Hot(C^\cu\contra)$.
 The triangulated Verdier quotient category
$$
 \sD^\ctr(C^\cu\contra)=\Hot(C^\cu\contra)/\Ac^\ctr(C^\cu\contra)
$$
is called the \emph{contraderived category} of left CDG\+contramodules
over~$C^\cu$.
\end{defn}

 In the context of the next theorem, it is helpful to recall that
the homological dimensions of the three abelian categories of (ungraded
or graded) left $C$\+comodules, right $C$\+comodules, and left
$C$\+contramodules coincide for any  (ungraded or graded) coalgebra $C$
over a field~$k$ (see~\cite[beginning of Section~4.5]{Pkoszul}).
 The common value of these three dimensions is called the \emph{global
dimension} of the (ungraded or graded) coalgebra~$C$.

\begin{thm} \label{fin-homol-dim-derived-cdg-contramodules}
 Let $C^\cu=(C,d,h)$ be a CDG\+coalgebra over~$k$ whose
underlying graded coalgebra $C$ has finite global dimension
(as a graded coalgebra; i.~e., the abelian category of graded left
$C$\+contramodules has finite homological dimension).
 Then the two classes of contraacyclic and absolutely acyclic
CDG\+contramodules over $C^\cu$ coincide,
$$
 \Ac^\ctr(C^\cu\contra)=\Ac^\abs(C^\cu\contra),
$$
and accordingly, the two derived categories of the second kind coincide,
$$
 \sD^\ctr(C^\cu\contra)=\sD^\abs(C^\cu\contra).
$$
\end{thm}

\begin{proof}
 This is~\cite[Theorem~4.5(b)]{Pkoszul}.
 For a generalization to exact DG\+categories,
see~\cite[Theorem~5.6(b) or Theorem~8.9(b)]{Pedg}.
\end{proof}

 Notice that the contramodule version of condition~($**$) from
Section~\ref{co-contra-derived-cdg-modules-subsecn} holds for
\emph{any} (graded) coalgebra $C$ over a field~$k$, because the class
of all projective $C$\+contramodules is closed under infinite products.
 Indeed, the projective $C$\+contramodules are the direct summands of
the free ones (see Sections~\ref{basics-of-contramodules-subsecn}\+-%
\ref{duality-analogy-subsecn}), and free left $C$\+contramodules have
the form $\Hom_k(C,V)$, where $V$ ranges over (graded) $k$\+vector
spaces; so free $C$\+contramodules obviously form a class closed under 
infinite products.

 Similarly to Sections~\ref{co-contra-derived-cdg-modules-subsecn}\+-%
\ref{coderived-cdg-comodules-subsecn}, we denote by
$\Hot(C^\cu\contra_\proj)\subset\Hot(C^\cu\contra)$
the full triangulated subcategory in the homotopy category formed by
all the CDG\+contramodules whose underlying graded $C$\+contramodules
are projective.
 Such CDG\+contramodules are called \emph{graded-projective}.

\begin{thm} \label{contraderived-cdg-contramodules-thm}
 Let $C^\cu=(C,d,h)$ be a CDG\+coalgebra over~$k$. 
 Then the composition\/ $\Hot(C^\cu\contra_\proj)\rarrow
\Hot(C^\cu\contra)\rarrow\sD^\ctr(C^\cu\contra)$ of the triangulated
inclusion functor\/ $\Hot(C^\cu\contra_\proj)\rarrow\Hot(C^\cu\contra)$
and the Verdier quotient functor\/ $\Hot(C^\cu\contra)\rarrow
\sD^\ctr(C^\cu\contra)$ is an equivalence of triangulated categories,
$$
 \Hot(C^\cu\contra_\proj)\simeq\sD^\ctr(C^\cu\contra).
$$
\end{thm}

\begin{proof}
 This is~\cite[Theorem~4.4(d)]{Pkoszul}.
 For a generalization to exact DG\+categories,
see~\cite[Theorem~5.10(b)]{Pedg}.
\end{proof}

\subsection{Derived co-contra correspondence}
 Let $C^\cu=(C,d,h)$ be a CDG\+coalgebra over~$k$.
 Then there is a natural pair of adjoint DG\+functors between
the DG\+categories of left CDG\+comodules and left CDG\+contramodules
over~$C^\cu$.
 The right adjoint DG\+functor is
$$
 \Hom_C(C^\cu,{-})\:C^\cu\comodl\lrarrow C^\cu\contra,
$$
while the left adjoint DG\+functor is
$$
 C^\cu\ocn_C{-}\,\:C^\cu\contra\lrarrow C^\cu\comodl.
$$

 Let us spell out the constructions of these functors.
 For any left CDG\+comodule $M^\cu$ over $C^\cu$, the underlying
graded left $C$\+contramodule $\Hom_C(C,M)$ of the CDG\+contramodule
$\Hom_C(C^\cu,M^\cu)$ is simply the graded left $C$\+contramodule of
homogeneous left $C$\+comodule morphisms $C\rarrow M$.
 This is the graded version of the construction of the functor
$\Hom_C(C,{-})$ from Section~\ref{underived-co-contra-subsecn}.
 So the degree~$n$ component $\Hom_C^n(C,M)$ of the graded
$C$\+contramodule $\Hom_C(C,M)$ is the vector space of all
homogeneous left $C$\+comodule homomorphisms $C\rarrow M$ of
degree~$n$ (for every $n\in\boZ$).
 The differential~$d$ on the left CDG\+contramodule
$\Hom_C(C^\cu,M^\cu)$ is given by the usual rule for the differential
on the Hom space of two precomplexes, as in
Sections~\ref{twisting-cochains-curved-subsecn}
and~\ref{cdg-contramodules-subsecn}.

 For any left CDG\+contramodule $P^\cu$ over $C^\cu$, the underlying
graded left $C$\+comodule $C\ocn_C P$ of the CDG\+comodule
$C^\cu\ocn_C P^\cu$ is simply the contratensor product of
the graded right $C$\+comodule $C$ and the graded left
$C$\+contramodule~$P$.
 Here one has to extend the construction of the contratensor product
functor from Section~\ref{contratensor-product-subsecn} to the realm
of graded coalgebras, graded comodules, and graded contramodules,
which is done in the most straightforward way
(see~\cite[Section~2.2]{Pkoszul} for the details).
 The graded left $C$\+comodule structure on the contratensor product
$C\ocn_C P$ is induced by the graded left $C$\+comodule structure
on $C$, as in Section~\ref{underived-co-contra-subsecn}.
 The differential~$d$ on the left CDG\+comodule $C^\cu\ocn_C P^\cu$
is given by the usual rule for the differential of the tensor product
of two precomplexes, as in Section~\ref{cdg-coalgebras-subsecn}.

 The standard notation for the comodule-contramodule correspondence
functors (as in~\cite{Psemi,Pkoszul,Pweak}) is
$\Psi_{C^\cu}=\Hom_C(C^\cu,{-})$ and $\Phi_{C^\cu}=C^\cu\ocn_C{-}$.

\begin{thm} \label{derived-co-contra-theorem}
 Let $C^\cu$ be a CDG\+coalgebra over~$k$.
 Then the adjoint DG\+functors\/ $\Hom_C(C^\cu,{-})$ and
$C^\cu\ocn_C{-}$, restricted to the full DG\+subcategories of
graded-injective CDG\+comodules and graded-projective
CDG\+contramodules, induce a triangulated equivalence
$$
 \Hot(C^\cu\comodl_\inj)\simeq\Hot(C^\cu\contra_\proj).
$$
 Consequently, the right derived functor of\/ $\Hom_C(C^\cu,{-})$
and the left derived functor of $C^\cu\ocn_C{-}$ are mutually
inverse triangulated equivalences between the coderived and
the contraderived category,
$$
 \boR\Hom_C(C^\cu,{-})\:\sD^\co(C^\cu\comodl)\simeq
 \sD^\ctr(C^\cu\contra)\,:\!C^\cu\ocn_C^\boL{-}.
$$
\end{thm}

\begin{proof}
 This is~\cite[Theorem~5.2]{Pkoszul}.
 To prove the first assertion, one observes that the adjunction
morphisms $C^\cu\ocn_C\Hom_C(C^\cu,M^\cu)\rarrow M^\cu$ and
$P^\cu\rarrow\Hom_C(C^\cu,\>C^\cu\ocn_C P^\cu)$ are (closed)
isomorphisms of objects in the respective DG\+categories for any
graded-injective left CDG\+comodule $M^\cu$ and graded-projective
left CDG\+contramodule $P^\cu$ over~$C^\cu$.
 As the property of a closed morphism of CDG\+comodules or
CDG\+contramodules to be an isomorphism depends only on
the underlying morphism of graded co/contramodules, the claim
essentially follows from the graded version of the underived
co-contra correspondence theory of
Section~\ref{underived-co-contra-subsecn}.
 In other words, the DG\+functors $\Hom_C(C^\cu,{-})$ and
$C^\cu\ocn_C{-}$ are mutually inverse equivalences between
the DG\+categories of graded-injective CDG\+comodules and
graded-projective CDG\+contramodules, $C^\cu\comodl_\inj\simeq
C^\cu\contra_\proj$; hence they induce an equivalence of
the respective homotopy categories.

 The second assertion follows from the first one in view of
Theorems~\ref{coderived-cdg-comodules-thm}
and~\ref{contraderived-cdg-contramodules-thm}.
 Essentially, the derived functor $\boR\Hom_C(C^\cu,{-})$ is
constructed by applying the functor $\Hom_C(C^\cu,{-})$ to
graded-injective CDG\+comodules, while the derived functor
$C^\cu\ocn_C^\boL{-}$ is constructed by applying the functor
$C^\cu\ocn_C{-}$ to graded-projective CDG\+contramodules.
\end{proof}

 For references to further discussions of the philosophy and
generalizations of the co-contra correspondence, see the end of
Section~\ref{underived-co-contra-subsecn}.

\subsection{Twisted differential on the graded Hom space}
\label{twisted-differential-on-graded-Hom-subsecn}
 This section is the contramodule version of
Section~\ref{twisted-differential-on-tensor-product-subsecn}.
 For the sake of simplicity of the exposition (to avoid
a detailed discussion of CDG\+bimodules and CDG\+bicomodules),
we restrict ourselves to the uncurved case.

 Let $A^\bu$ be a DG\+algebra and $C^\bu$ be a DG\+coalgebra over~$k$,
and let $\tau\:C^\bu\rarrow A^\bu$ be a twisting cochain (as defined in
Section~\ref{hom-dg-algebra-and-twisting-cochains-subsecn}).
 Let $M^\bu=(M,d_M)$ be a left DG\+module over $A^\bu$, and let
$Q^\bu=(Q,d_Q)$ be a left DG\+contramodule over~$C^\bu$.
 Consider the graded Hom space $\Hom_k(M,Q)$ of the underlying
graded vector spaces of $M^\bu$ and $Q^\bu$, and endow it with
the differential given by the formula
$$
 d(f)(x)=d_Q(f(x))-(-1)^{|f|}f(d_M(x))\pm d^\tau(f)
$$
for all homogeneous $k$\+linear maps $f\in\Hom_k^{|f|}(M,Q)$ and
homogeneous elements $x\in M^{|x|}$.
 Here $d^\tau\:\Hom_k(M,Q)\rarrow\Hom_k(M,Q)$ is the homogeneous
map of degree~$1$ given by the rule
$$
 d^\tau(f)(x)=\pi(c\mapsto f(\tau(c)x)),
$$
for all $f\in\Hom_k^{|f|}(M,Q)$ and $x\in M^{|x|}$.
 In this formula, $\pi\:\Hom_k(C,Q)\rarrow Q$ is the contraaction map,
while $c\longmapsto f(\tau(c)x)$ is the homogeneous $k$\+linear map
$C\rarrow Q$ assigning to an element $c\in C$ the image of
the element $\tau(c)x\in M$ under the map~$f$.
 The reader can consult with~\cite[Section~6.2]{Pkoszul} for
the sign rule.

 Then one can check that $d^2(f)=0$.
 So the graded vector space $\Hom_k(M,Q)$ endowed with
the differential~$d$ is a complex.
 We denote this complex by $\Hom^\tau(M^\bu,Q^\bu)$.
 
 Analogously, let $N^\bu=(N,d_N)$ be a left DG\+comodule over $C^\bu$,
and let $P^\bu=(P,d_P)$ be a left DG\+module over~$A^\bu$.
 Consider the graded Hom space $\Hom_k(N,P)$ of the underlying
graded vector spaces of $N^\bu$ and $P^\bu$, and endow it with
the differential given by the formula
$$
 d(f)(y)=d_P(f(y))-(-1)^{|f|}f(d_N(y))\pm d^\tau(f)
$$
for all homogeneous $k$\+linear maps $f\in\Hom_k^{|f|}(N,P)$ and
homogeneous elements $y\in N^{|y|}$.
 Here $d^\tau\:\Hom_k(N,P)\rarrow\Hom_k(N,P)$ is the homogeneous
map of degree~$1$ given by the rule that, for any homogeneous
linear map $f\in\Hom_k^{|f|}(N,P)$, the map $d^\tau(f)\in
\Hom_k^{|f|+1}(N,P)$ is equal to the composition
$$
 N\overset\nu\lrarrow C\ot_k N\overset{\tau\ot f}\lrarrow
 A\ot_k P\overset n\lrarrow P
$$
of the comultiplication map $\nu\:N\rarrow C\ot_k N$, the map
$\tau\ot f\:C\ot_kN\rarrow A\ot_kP$, and the multiplication map
$n\:A\ot_k P\rarrow P$.

 Once again, choosing the sign properly
(cf.~\cite[Section~6.2]{Pkoszul}), one can check that $d^2(f)=0$.
 So the graded vector space $\Hom_k(N,P)$ endowed with
the differential~$d$ is a complex.
 We denote this complex by $\Hom^\tau(N^\bu,P^\bu)$.

\subsection{Derived Koszul duality on the contramodule side}
\label{koszul-duality-contramodule-side-subsecn}
 This section is the contramodule version of
Section~\ref{koszul-duality-comodule-side-subsecn}.
 Let $B^\cu$ be a CDG\+algebra and $C^\cu$ be a CDG\+coalgebra over~$k$,
and let $\tau\:C^\cu\rarrow B^\cu$ be a twisting cochain
(as defined in Section~\ref{twisting-cochains-curved-subsecn}).

 Given a left CDG\+contramodule $Q^\cu$ over $C^\cu$, we consider
the graded Hom space $\Hom_k(B,Q)$ and endow it with
the differential~$d$ twisted with the twisting cochain~$\tau$ using
the same formulas as in
Section~\ref{twisted-differential-on-graded-Hom-subsecn}.
 Then $\Hom^\tau(B^\cu,Q^\cu)=(\Hom_k(B,Q),\,d)$ is a left
CDG\+module over~$B^\cu$.
 Here the left $B$\+module structure on $B$ has been eaten up in
the construction of the twisted differential on the Hom space,
but the right $B$\+module structure on $B$ remains and induces
the underlying graded left $B$\+module structure of the CDG\+module
$\Hom^\tau(B^\cu,Q^\cu)$.

 Similarly, given a left CDG\+module $P^\cu$ over $B^\cu$, we
consider the graded Hom space $\Hom_k(C,P)$ and endow it with
the differential~$d$ twisted with the twisting cochain~$\tau$ using
the same formulas as in
Section~\ref{twisted-differential-on-graded-Hom-subsecn}.
 Then $\Hom^\tau(C^\cu,P^\cu)=(\Hom_k(C,P),\,d)$ is a left
CDG\+contramodule over~$C^\cu$.
 Here the left $C$\+comodule structure on $C$ has been consumed in
the construction of the twisted differential on the Hom space,
but the right $C$\+comodule structure on $C$ stays and induces
the underlying graded left $C$\+contramodule structure of
the CDG\+contramodule $\Hom^\tau(C^\cu,P^\cu)$.

 Recall the notation $B^\cu\modl$ for the DG\+category of left
CDG\+modules over $B^\cu$ and $C^\cu\contra$ for the DG\+category
of left CDG\+contramodules over $C^\cu$ (see
Sections~\ref{cdg-rings-subsecn} and~\ref{cdg-contramodules-subsecn}).
 Dually to Section~\ref{koszul-duality-comodule-side-subsecn}, one
observes that the DG\+functor
$$
 \Hom^\tau(B^\cu,{-})\:C^\cu\contra\lrarrow B^\cu\modl
$$
is right adjoint to the DG\+functor
$$
 \Hom^\tau(C^\cu,{-})\:B^\cu\modl\lrarrow C^\cu\contra.
$$

\begin{thm} \label{contra-side-conilpotent-duality-thm}
 Let $A^\bu$ be a nonzero DG\+algebra and $(C^\cu,\gamma)$ be
a conilpotent CDG\+coalgebra over~$k$ (as defined in
Section~\ref{duality-DG-algebras-curved-DG-coalgebras-subsecn}).
 Let $\tau\:C^\cu\rarrow A^\bu$ be an acyclic twisting cochain
(as defined in Section~\ref{twisting-cochains-curved-subsecn});
this includes the condition that $\tau\circ\gamma=0$.
 Then the adjoint functors $P^\bu\longmapsto\Hom^\tau(C^\cu,P^\bu)$
and $Q^\cu\longmapsto\Hom^\tau(A^\bu,Q^\cu)$ induce a triangulated
equivalence between the conventional \emph{derived} category of
left DG\+modules over $A^\bu$ and the \emph{contraderived} category of
left CDG\+contramodules over~$C^\cu$,
$$
 \sD(A^\bu\modl)\simeq\sD^\ctr(C^\cu\contra).
$$
\end{thm}

\begin{proof}
 This is~\cite[Theorem~6.5(b)]{Pkoszul}.
 The particular case corresponding to the twisting cochain from
Examples~\ref{nonaugmented-bar-twisting-cochains-examples} can be
found in~\cite[Theorem~6.3(b)]{Pkoszul}, while
the case of the twisting cochain from
Example~\ref{curved-cobar-construction-twisting-cochain-example}
(for a conilpotent CDG\+coalgebra~$C^\cu$)
is considered in~\cite[Theorem~6.4(b)]{Pkoszul}.
 For the definition of the contraderived category,
see Section~\ref{contraderived-cdg-contramodules-subsecn} above.
\end{proof}

 The triangulated equivalence of
Theorem~\ref{contra-side-conilpotent-duality-thm} takes
the left CDG\+contramodule~$k$ over $C^\cu$ (with
the $C$\+contramodule structure on~$k$ defined in terms of
the coaugmentation~$\gamma$) to the cofree left DG\+module
$\Hom_k(A^\bu,k)$ over~$A^\bu$.

 Notice that there is no natural structure of a DG\+module over $A^\bu$
on the one-dimensional vector space~$k$, as the DG\+algebra $A^\bu$ is
not augmented.
 This fact is not unrelated to the fact that there is no natural
structure of a left CDG\+contramodule over $C^\cu$ on the free graded
left $C$\+contramodule $\Hom_k(C,k)$ (because of a mismatch of
the equations for the square of the differential involving
the curvature, similar to the one explained in the end of
Section~\ref{cdg-coalgebras-subsecn}).

\begin{thm} \label{contra-side-nonconilpotent-duality-thm}
 Let $C^\cu$ be a nonzero CDG\+coalgebra, and let $\tau\:C^\cu\rarrow
\Cb^\cu_w(C^\cu)=B^\cu$ be the twisting cochain from
Example~\ref{curved-noncoaugmented-cobar-twisting-cochain-example}.
 Then the adjoint functors $P^\cu\longmapsto\Hom^\tau(C^\cu,P^\cu)$
and $Q^\cu\longmapsto\Hom^\tau(B^\cu,Q^\cu)$ induce a triangulated
equivalence between the \emph{absolute derived} category of left
CDG\+modules over $B^\cu$ and the \emph{contraderived} category of
left CDG\+contramodules over~$C^\cu$,
$$
 \sD^\abs(B^\cu\modl)\simeq\sD^\ctr(C^\cu\contra).
$$
\end{thm}

\begin{proof}
 This is~\cite[Theorem~6.7(b)]{Pkoszul}.
 The definitions of the contraderived and absolute derived categories
are explained in Sections~\ref{contraderived-cdg-contramodules-subsecn}
and~\ref{co-contra-derived-cdg-modules-subsecn}.
 The absolute derived category $\sD^\abs(B^\cu\modl)$ coincides with
the contraderived category $\sD^\ctr(B^\cu\modl)$
by~\cite[Theorem~3.6(a)]{Pkoszul}; see
Theorem~\ref{fin-homol-dim-derived-cdg-modules}(a) above.
\end{proof}

\subsection{Koszul triality}  \label{koszul-triality-subsecn}
 In this section, we summarize the results formulated above in this
survey in the form of commutative diagrams of Koszul triality.
 There are two main triality theorems in~\cite{Pkoszul}:
the conilpotent and the nonconilpotent one.

\begin{thm} \label{conilpotent-triality-theorem}
 Let $A^\bu$ be a DG\+algebra and $(C^\cu,\gamma)$ be
a conilpotent CDG\+coalgebra over~$k$.
 Let $\tau\:C^\cu\rarrow A^\bu$ be an acyclic twisting cochain;
this includes the conditions that $A^\bu\ne0$ and $\tau\circ\gamma=0$.
 Then there is a commutative diagram of triangulated category
equivalences
\begin{equation} \label{conilpotent-triality-eqn}
\begin{gathered}
 \xymatrix{
  &&& \sD^\co(C^\cu\comodl) \ar@{=}[dd] \\
  \sD(A^\bu\modl) \ar@{=}[rrru] \ar@{=}[rrrd] \\
  &&& \sD^\ctr(C^\cu\contra)
 }
\end{gathered}
\end{equation}
where the upper diagonal double line
$$
 C^\cu\ot^\tau{-}\,\:\sD(A^\bu\modl)\simeq
 \sD^\co(C^\cu\comodl)\,:\!A^\bu\ot^\tau{-}
$$
is the comodule side conilpotent Koszul duality of
Theorem~\ref{nonaugmented-acyclic-twisting-cochain-duality-thm};
the lower diagonal double line
$$
 \Hom^\tau(C^\cu,{-})\:\sD(A^\bu\modl)\simeq
 \sD^\ctr(C^\cu\contra)\,:\!\Hom^\tau(A^\bu,{-})
$$
is the contramodule side conilpotent Koszul duality of
Theorem~\ref{contra-side-conilpotent-duality-thm}; and
the vertical double line
$$
 \boR\Hom_C(C^\cu,{-})\:\sD^\co(C^\cu\comodl)\simeq
 \sD^\ctr(C^\cu\contra)\,:\!C^\cu\ocn_C^\boL{-}.
$$
is the derived comodule-contramodule correspondence of
Theorem~\ref{derived-co-contra-theorem}.
\end{thm}

\begin{proof}
 This is~\cite[Theorem~6.5]{Pkoszul}.
 The particular case corresponding to the twisting cochain from
Examples~\ref{nonaugmented-bar-twisting-cochains-examples} can be
found in~\cite[Theorem~6.3]{Pkoszul}, while
the case of the twisting cochain from
Example~\ref{curved-cobar-construction-twisting-cochain-example}
(for a conilpotent CDG\+coalgebra~$C^\cu$)
is considered in~\cite[Theorem~6.4]{Pkoszul}.
 For the definitions of the coderived and the contraderived category,
see Sections~\ref{coderived-cdg-comodules-subsecn}
and~\ref{contraderived-cdg-contramodules-subsecn}.
 For discussions of some elements of the proof of the theorem,
see Remarks~\ref{convergent-spectral-sequence-remark}
and~\ref{motivation-coacyclics-to-contractibles-remark} above and
Section~\ref{comments-on-proof-subsecn} below.
\end{proof}

 Here the comodule side of the Koszul duality takes
the left CDG\+comodule~$k$ over $C^\cu$ (with the $C$\+comodule
structure on~$k$ defined in terms of the coaugmentation~$\gamma$)
to the free left DG\+module $A^\bu$ over~$A^\bu$.
 The contramodule side of the Koszul duality takes
the left CDG\+contramodule~$k$ over $C^\cu$ (with
the $C$\+contramodule structure on~$k$ similarly defined in terms of
the coaugmentation~$\gamma$) to the cofree left DG\+module
$\Hom_k(A^\bu,k)$ over~$A^\bu$.

 Now let us assume that $(A^\bu,\alpha)$ is an augmented
DG\+algebra and $(C^\bu,\gamma)$ is a conilpotent DG\+coalgebra.
 Let $\tau\:C^\bu\rarrow A^\bu$ be an acyclic twisting cochain
satisfying the equation $\alpha\circ\tau=0$ (in addition to
the previously assumed $\tau\circ\gamma=0$).
 For example, the twisting cochains from
Examples~\ref{bar-construction-twisting-cochains-examples},
\ref{cobar-construction-twisting-cochains-examples},
and~\ref{nonhomogeneous-quadratic-twisting-cochain}, as well as
generally ``acyclic twisting cochains'' in the sense of
Section~\ref{bar-cobar-and-acyclic-twisting-cochains-subsecn},
satisfy these additional conditions.

 Then the left DG\+module~$k$ over $A^\bu$ (with the $A$\+module
structure on~$k$ defined in terms of the augmentation~$\alpha$),
the cofree left DG\+comodule $C^\bu$ over $C^\bu$, and the free
left DG\+contramodule $\Hom_k(C^\bu,k)$ over $C^\bu$ correspond
to each other under the triality~\eqref{conilpotent-triality-eqn} of
Theorem~\ref{conilpotent-triality-theorem}.

\begin{thm} \label{nonconilpotent-triality-theorem}
 Let $C^\cu\ne0$ be a CDG\+coalgebra, and let $\tau\:C^\cu\rarrow
\Cb^\cu_w(C^\cu)=B^\cu$ be the twisting cochain from
Example~\ref{curved-noncoaugmented-cobar-twisting-cochain-example}.
 Then there is a commutative diagram of triangulated category
equivalences
\begin{equation} \label{nonconilpotent-triality-eqn}
\begin{gathered}
 \xymatrix{
  &&& \sD^\co(C^\cu\comodl) \ar@{=}[dd] \\
  \sD^{\co=\abs=\ctr}(B^\cu\modl)
  \ar@{=}[rrru] \ar@{=}[rrrd] \\
  &&& \sD^\ctr(C^\cu\contra)
 }
\end{gathered}
\end{equation}
where the upper diagonal double line
$$
 C^\cu\ot^\tau{-}\,\:\sD^{\abs=\co}(B^\cu\modl)\simeq
 \sD^\co(C^\cu\comodl)\,:\!B^\cu\ot^\tau{-}
$$
is the comodule side nonconilpotent Koszul duality of
Theorem~\ref{nonaugmented-nonconilpotent-duality-thm};
the lower diagonal double line
$$
 \Hom^\tau(C^\cu,{-})\:\sD^{\abs=\ctr}(B^\cu\modl)\simeq
 \sD^\ctr(C^\cu\contra)\,:\!\Hom^\tau(B^\cu,{-})
$$
is the contramodule side nonconilpotent Koszul duality of
Theorem~\ref{contra-side-nonconilpotent-duality-thm}; and
the vertical double line
$$
 \boR\Hom_C(C^\cu,{-})\:\sD^\co(C^\cu\comodl)\simeq
 \sD^\ctr(C^\cu\contra)\,:\!C^\cu\ocn_C^\boL{-}.
$$
is the derived comodule-contramodule correspondence of
Theorem~\ref{derived-co-contra-theorem}.
\end{thm}

\begin{proof}
 This is~\cite[Theorem~6.7]{Pkoszul}.
 For the definitions of the coderived, contraderived, and absolute
derived categories, see
Sections~\ref{co-contra-derived-cdg-modules-subsecn},
\ref{coderived-cdg-comodules-subsecn},
and~\ref{contraderived-cdg-contramodules-subsecn}.
 The notation $\sD^{\co=\abs=\ctr}(B^\cu\modl)$ is intended to
emphasize that the coderived, contraderived, and absolute derived
category of left CDG\+modules over $B^\cu$ coincide by
Theorem~\ref{fin-homol-dim-derived-cdg-modules}(a) (since the free
graded algebra $B$ has finite global dimension~$1$).
 For a discussion of some elements of the proof of the theorem,
see Remark~\ref{motivation-coacyclics-to-contractibles-remark}
and Section~\ref{comments-on-proof-subsecn}.
\end{proof}

\subsection{Some comments on the proof}
\label{comments-on-proof-subsecn}
 The aim of this section is to complement the proofs of the Koszul
duality theorems~\cite[Theorems~6.3, 6.4, 6.5, and~6.7]{Pkoszul}
(see Theorems~\ref{conilpotent-triality-theorem}
and~\ref{nonconilpotent-triality-theorem} above)
with some additional details.

 In the context of both the diagrams~\eqref{conilpotent-triality-eqn}
and~\eqref{nonconilpotent-triality-eqn}, there are three points to be
explained: that the diagonal (Koszul duality) functors are well-defined,
that the diagonal functors acting in the opposite directions are
mutually inverse, and that the triangular diagram is commutative.
 We have already commented upon the vertical (co-contra) equivalence in
the proof of Theorem~\ref{derived-co-contra-theorem}.

 The main thing to observe is that \emph{the Koszul duality functors are
exact and need not be derived}.
 There is \emph{no} need to replace a (C)DG\+module, CDG\+comodule, or
CDG\+contramodule with any resolution (adjusted version) of it before
applying any one of the functors $A^\bu\ot^\tau{-}$, \
$C^\cu\ot^\tau{-}$, etc.\ (defined in
Sections~\ref{koszul-duality-comodule-side-subsecn}
and~\ref{koszul-duality-contramodule-side-subsecn}).

 In fact, the Koszul duality functors have even better properties:
they \emph{take weakly trivial objects to contractible ones}.
 For example, in the context of the conilpotent duality of
Theorem~\ref{conilpotent-triality-theorem}, the functor
$A^\bu\ot^\tau{-}$ takes coacyclic CDG\+comodules to contractible
DG\+modules, while the functor $C^\cu\ot^\tau{-}$ takes acyclic
DG\+modules to contractible CDG\+comodules.
 Similarly, in the same theorem the functor $\Hom^\tau(A^\bu,{-})$
takes contraacyclic CDG\+contramodules to contractible DG\+modules,
while the functor $\Hom^\tau(C^\cu,{-})$ takes acyclic DG\+modules
to contractible CDG\+contramodules.

 By adjunction, the same property of the Koszul duality functors can be
expressed by saying that they \emph{take arbitrary objects to adjusted
ones} (meaning ``adjusted'' on the suitable side and for the respective
exotic derived category).
 For example, in the context of
Theorem~\ref{conilpotent-triality-theorem}, the functor
$A^\bu\ot^\tau{-}$ takes arbitrary CDG\+comodules to homotopy projective
DG\+modules, while the functor $C^\cu\ot^\tau{-}$ takes arbitrary
DG\+modules to graded-injective CDG\+comodules.
 Similarly, in the same theorem the functor $\Hom^\tau(A^\bu,{-})$
takes arbitrary CDG\+contramodules to homotopy injective DG\+modules,
while the functor $\Hom^\tau(C^\cu,{-})$ takes arbitrary DG\+modules
to graded-projective CDG\+contramodules.

 The fact that the Koszul duality functors take co/contraacyclic
objects to contractible ones is explained in
Remark~\ref{motivation-coacyclics-to-contractibles-remark}.
 Let us make a comment on the proofs of the claims that the functor
$C^\cu\ot^\tau{-}$ takes acyclic DG\+modules over $A^\bu$ to
contractible CDG\+comodules over $C^\cu$, while the functor
$\Hom^\tau(C^\cu,{-})$ takes acyclic DG\+modules over $A^\bu$
to contractible CDG\+contramodules over~$C^\cu$ (in the context of
the conilpotent Koszul duality of
Theorem~\ref{conilpotent-triality-theorem}).
 To complement the arguments in~\cite[Sections~6.3 and~6.4]{Pkoszul}
and Remark~\ref{convergent-spectral-sequence-remark}, let us spell out
one basic observation.

 For the sake of accessibility of the exposition, we formulate
the following lemma for CDG\+modules, while in fact parts~(a), (b)
are equally applicable to CDG\+comodules and parts~(a), (c) to
CDG\+contramodules.

\begin{lem} \label{filtrations-co-contra-acyclicity-lemma}
 Let $B^\cu=(B,d,h)$ be a CDG\+ring. \par
\textup{(a)} Let $M^\cu$ be a CDG\+module over $B^\cu$ endowed with
a finite filtration\/ $0=F_0M^\cu\subset F_1M^\cu\subset\dotsb\subset
F_nM^\cu=M^\cu$ by CDG\+submodules $F_iM^\cu\subset M^\cu$.
 Assume that the CDG\+module $F_iM^\cu/F_{i-1}M^\cu$ over $B^\cu$ is
absolutely acyclic for every\/ $1\le i\le n$.
 Then the CDG\+module $M^\cu$ over $B^\cu$ is absolutely acyclic. \par
\textup{(b)} Let $N^\cu$ be a CDG\+module over $B^\cu$ endowed with
an increasing filtration\/ $0=F_0N^\cu\subset F_1N^\cu\subset\dotsb
\subset F_iN^\cu\subset F_{i+1}N^\cu\subset\dotsb$ by CDG\+submodules
$F_iN^\cu\subset N^\cu$.
 Assume that the filtration $F$ is exhaustive, that is
$N^\cu=\bigcup_{i\ge0} F_iN^\cu$, and the CDG\+module
$F_iN^\cu/F_{i-1}N^\cu$ over $B^\cu$ is coacyclic for every $i\ge1$.
 Then the CDG\+module $N^\cu$ over $B^\cu$ is coacyclic. \par
\textup{(c)} Let $Q^\cu$ be a CDG\+module over $B^\cu$ endowed with
a decreasing filtration $Q^\cu=F^0Q^\cu\supset F^1Q^\cu\supset\dotsb
\supset F^iQ^\cu\supset F^{i+1}Q^\cu\supset\dotsb$ by CDG\+submodules
$F^iQ^\cu\subset Q^\cu$.
 Assume that the filtration $F$ is (separated and) complete, that is
$Q^\cu=\varprojlim_{i\ge1} Q^\cu/F^iQ^\cu$, and the CDG\+module
$F^iQ^\cu/F^{i+1}Q^\cu$ over $B^\cu$ is contraacyclic for every $i\ge0$.
 Then the CDG\+module $Q^\cu$ over $B^\cu$ is contraacyclic.
\end{lem}

\begin{proof}
 This lemma is implicit in~\cite{Psemi,Pkoszul}; cf.~\cite[proof of
Lemma~2.1]{Psemi} or~\cite[proof of Theorem~3.7]{Pkoszul}.
 Part~(a): proceeding by induction on~$n$, it suffices to prove that,
for any short exact sequence of CDG\+modules $0\rarrow K^\cu\rarrow
L^\cu\rarrow N^\cu\rarrow0$ over $B^\cu$, the CDG\+module $L^\cu$ is
absolutely acyclic whenever both the CDG\+modules $K^\cu$ and $N^\cu$
are.
 Indeed, the CDG\+module $T^\cu=\Tot(K^\cu\to L^\cu\to N^\cu)$ is
absolutely acyclic by the definition, and the CDG\+module $L^\cu$ is
homotopy equivalent to a CDG\+module which can be obtained from $K^\cu$,
$T^\cu$, and $N^\cu$ using (a small finite number of) passages to
shifts and cones.

 Part~(b): similarly to part~(a), one proves by induction on~$i$ that
the CDG\+module $F_iN^\cu$ is coacyclic for every $i\ge1$.
 Having that established, consider the telescope short exact sequence
$$
 0\lrarrow\bigoplus\nolimits_{i=1}^\infty F_iN^\cu\lrarrow
 \bigoplus\nolimits_{i=1}^\infty F_iN^\cu\lrarrow N^\cu\lrarrow0
$$
of CDG\+modules over $B^\cu$, and denote by $T^\cu$ its total
CDG\+module.
 Then the CDG\+module $T^\cu$ is absolutely acyclic by the definition,
while the CDG\+module $\bigoplus_{i=1}^\infty F_iN^\cu$ is coacyclic
since the CDG\+modules $F_iN^\cu$ are.
 It follows that the CDG\+module $N^\cu$ is coacyclic.

 Part~(c): similarly to part~(a), one proves by induction on~$i$ that
the CDG\+module $Q^\cu/F^iQ^\cu$ is contraacyclic for every $i\ge1$.
 Now one can consider the telescope short exact sequence
$$
 0\lrarrow Q^\cu\lrarrow\prod\nolimits_{i=1}^\infty Q^\cu/F^iQ^\cu
 \lrarrow\prod\nolimits_{i=1}^\infty Q^\cu/F^iQ^\cu\lrarrow0
$$
of CDG\+modules over~$B^\cu$.
 This sequence is exact, because the transition maps in the projective
system $Q^\cu/F^{i+1}Q^\cu\rarrow Q^\cu/F^iQ^\cu$ are surjective and
$Q^\cu=\varprojlim_{i\ge1} Q^\cu/F^iQ^\cu$.
 Denote by $T^\cu$ its total CDG\+module.
 Then the CDG\+module $T^\cu$ is absolutely acyclic by the definition,
while the CDG\+module $\prod_{i=1}^\infty Q^\cu/F^iQ^\cu$ is
contraacyclic since the CDG\+modules $Q^\cu/F^iQ^\cu$ are.
 It follows that the CDG\+module $Q^\cu$ is contraacyclic.
\end{proof}

 Our next comment is on the proofs of the claims that the diagonal
functors in the opposite directions are mutually inverse in
Theorems~\ref{conilpotent-triality-theorem}
and~\ref{nonconilpotent-triality-theorem}.
 The following observations are intended to complement the arguments
in the proofs of~\cite[Theorems~4.4, 6.3, and~6.4]{Pkoszul}.
 As above, we formulate the assertions for CDG\+modules, though those
of them not involving contraacyclicity are also applicable to
CDG\+comodules, while those not involving coacyclicity are applicable
to CDG\+contramodules.

\begin{lem} \label{freely-generated-cdg-module-lemma}
 Let $B^\cu=(B,d,h)$ be a CDG\+ring and $M^\cu=(M,d_M)$ be a left
CDG\+module over~$B^\cu$.
 Let $L\subset M$ be a graded $B$\+submodule.
 Assume that the composition $L\rightarrowtail M\overset{d_M}\rarrow
M\twoheadrightarrow M/L$ is an isomorphism of graded abelian groups
(or of graded left $B$\+modules) $L\rarrow (M/L)[1]$.
 Then the CDG\+module $M^\cu$ over $B^\cu$ is contractible.
\end{lem}

\begin{proof}
 Here the notation $L\rarrow (M/L)[1]$ means that the composition of
maps in question is a homogeneous map $L\rarrow M/L$ of degree~$1$.
 In fact, it is always a morphism of graded left $B$\+modules (with
the suitable sign rule).
 The map $L\rarrow (M/L)[1]$ being bijective means that the CDG\+module
$M^\cu$ over $B^\cu$ is \emph{freely generated} by the graded left
$B$\+module~$L$.
 In the notation of~\cite[proof of Theorem~3.6]{Pkoszul}, this is
expressed by the formula $M^\cu=G^+(L)$.
 All CDG\+modules over $B^\cu$ freely generated by graded $B$\+modules
are easily seen to be contractible.
\end{proof}

 Now let us consider a CDG\+module $M^\cu$ over $B^\cu$ endowed with
a filtration by graded $B$\+submodules
$$
\dotsb\subset F_{-i}M \subset F_{-i+1}M\subset\dotsb\subset F_0M
\subset F_1M \subset\dotsb\subset F_iM\subset F_{i+1}M\subset\dotsb
$$
 The differential $d_M\:M\rarrow M$ is \emph{not} supposed to preserve
this filtration; however, we assume that $d_M(F_iM)\subset F_{i+1}M$.
 Then for every $i\in\boZ$ we have a morphism of graded left
$B$\+modules $F_iM/F_{i-1}M\rarrow (F_{i+1}M/F_iM)[1]$ induced by~$d_M$.
 The collection of all such maps is a \emph{complex of graded left
$B$\+modules}
\begin{multline*}
 \dotsb\lrarrow(F_{i-1}M/F_{i-2}M)[i-1]\lrarrow
 (F_iM/F_{i-1}M)[i] \\ \lrarrow(F_{i+1}M/F_iM)[i+1]\lrarrow\dotsb
\end{multline*}

\begin{lem} \label{acyclic-filtration-lemma}
\textup{(a)} Let $N^\cu=(N,d_N)$ be a left CDG\+module over $B^\cu$,
and let\/ $0=F_0N\subset F_1N\subset F_2N\subset\dotsb$ be an exhaustive
increasing filtration of $N$ by graded $B$\+submodules $F_iN\subset N$.
 Assume that $d_N(F_iN)\subset F_{i+1}N$ for all $i\ge1$ and
the complex of graded left $B$\+modules
$$
 0\lrarrow F_1N[1]\lrarrow (F_2N/F_1N)[2]\lrarrow (F_3N/F_2N)[3]
 \lrarrow\dotsb
$$
with the differential induced by~$d_N$ is acyclic.
 Then the CDG\+module $N^\cu$ over $B^\cu$ is coacyclic. \par
\textup{(b)} Let $Q^\cu=(Q,d_Q)$ be a left CDG\+module over $B^\cu$,
and let $Q=F^0Q\supset F^1Q\supset F^2Q\supset\dotsb$ be a complete
decreasing filtration of $Q$ by graded $B$\+submodules $F^iQ\subset Q$.
 Assume that $d_Q(F^iQ)\subset F^{i-1}Q$ for all $i\ge2$ and
the complex of graded left $B$\+modules
$$
 \dotsb\lrarrow (F^2Q/F^3Q)[-2]\lrarrow (F^1Q/F^2Q)[-1]\lrarrow
 Q/F^1Q\lrarrow 0
$$
with the differential induced by~$d_Q$ is acyclic.
 Then the CDG\+module $Q^\cu$ over $B^\cu$ is contraacyclic.
\end{lem}

\begin{proof}
 Part~(a): define a new increasing filtration $G$ on $N$ by the rule
$G_iN=F_iN+d_N(F_iN)$.
 Then $0=G_0N^\cu\subset G_1N^\cu\subset G_2N^\cu\subset\dotsb$ is
an increasing filtration of $N^\cu$ by CDG\+submodules $G_iN^\cu
\subset N^\cu$.
 The filtration $G$ is exhaustive if and only if the filtration $F$ is.
 In view of Lemma~\ref{freely-generated-cdg-module-lemma}, the condition
of exactness of the complex of graded left $B$\+modules in part~(a)
implies that the CDG\+modules $M^\cu_i=G_iN^\cu/G_{i-1}N^\cu$ are
contractible for all $i\ge1$.
 According to Lemma~\ref{filtrations-co-contra-acyclicity-lemma}(b),
it follows that the CDG\+module $N^\cu$ is coacyclic.
 The proof of part~(b) is dual and based on
Lemmas~\ref{freely-generated-cdg-module-lemma}
and~\ref{filtrations-co-contra-acyclicity-lemma}(c): the new filtration
$G^iQ=F^{i+1}Q+d_Q(F^{i+1}Q)$ is a decreasing filtration by
CDG\+submodules of $Q^\cu$, it is complete if and only if the filtration
$F$ is, the condition of exactness of the complex implies that
the CDG\+modules $G^iQ^\cu/G^{i+1}Q^\cu$ are contractible, etc.
\end{proof}

\end{document}